\documentclass[10pt]{amsart}
\usepackage{amssymb,amscd}
\newcommand{\bbinom}[2]{\genfrac{[}{]}{0pt}{}{#1}{#2}}
\newcommand{\GL}{\mathrm{GL}}
\newcommand{\PGL}{\mathrm{PGL}}

\newcommand{\sur}{\twoheadrightarrow}
\newcommand{\iso}{\xrightarrow{\sim}}

\newcommand{\CC}{\mathbb C}
\newcommand{\ZZ}{\mathbb Z}
\newcommand{\Alg}{\mathrm{Alg}}

\newcommand{\cycl}{C}
\newcommand{\C}{{\mathcal C}}
\newcommand{\ad}{{\mathrm{ad}}}
\newcommand{\ev}{{\mathrm{ev}}}
\newcommand{\coev}{{\mathrm{coev}}}
\newcommand{\tensor}{\mathop{\otimes}}
\newcommand{\subspace}{\subset}
\newcommand{\act}{\mathop{\triangleright}}
\newcommand{\ract}{\mathop{\triangleleft}}
\newcommand{\actchi}{\act\nolimits_\chi}

\newcommand{\adact}{\act\nolimits_\ad}
\newcommand{\dirsum}{\oplus}
\newcommand{\co}[1]{\overline{#1}}
\newcommand{\tensorpow}{\otimes}
\newcommand{\dash}{\nobreakdash-\hspace{0pt}}
\newcommand{\Ob}{\mathrm{Ob}}
\newcommand{\YDE}[2]{{{\vphantom{X}}_{#1, #2}\mathcal{YD}^{#2}}}
\newcommand{\YD}[1]{{{\vphantom{X}}_{#1}\mathcal{YD}^{#1}}}
\newcommand{\Comod}[1]{{\mathcal{M}^{#1}}}

\newcommand{\Fun}{\mathit{Fun}}
\newcommand{\Vect}{\mathit{Vect}}
\newcommand{\lgen}{\text{$<$}}
\newcommand{\rgen}{\text{$>$}}

\newcommand{\lcprod}{\rtimes}
\newcommand{\rcprod}{\ltimes}

\DeclareMathOperator{\id}{id}
\DeclareMathOperator{\End}{End}

\DeclareMathOperator{\Wor}{Wor}

\DeclareMathOperator{\Id}{Id}

\DeclareMathOperator{\Hom}{Hom}
\DeclareMathOperator{\Aut}{Aut}
\DeclareMathOperator{\im}{im}
\DeclareMathOperator{\gr}{\mathtt{gr}}
\DeclareMathOperator{\GVect}{\Gamma\hyp\Vect}
\DeclareMathOperator{\hyp}{-}

\theoremstyle{plain}
\newtheorem{theorem}[subsection]{Theorem}
\newtheorem{proposition}[subsection]{Proposition}

\newtheorem{lemma}[subsection]{Lemma}
\newtheorem{corollary}[subsection]{Corollary}

\newtheorem{question}[subsection]{Question}

\theoremstyle{definition}
\newtheorem{remark}[subsection]{Remark}
\newtheorem{definition}[subsection]{Definition}
\newtheorem{example}[subsection]{Example}

\advance\oddsidemargin by-0.5in
\advance\evensidemargin by-0.5in
\advance\textwidth by 1in

\begin{document}

\title{Cocycle twists and extensions of braided doubles}

\author{Yuri Bazlov}
\address{School of Mathematics, University of Manchester, Oxford Road, Manchester, M13 9PL, UK}
\email{yuri.bazlov@manchester.ac.uk}

\author{Arkady Berenstein}
\address{Department of Mathematics, University of Oregon, Eugene, OR 97403, USA} 
\email{arkadiy@math.uoregon.edu}
\thanks{We acknowledge the support of the LMS Research in Pairs grant ref.\ 41024. 
The second named author was partially supported by the NSF grant DMS-1101507.%
}

\pagestyle{myheadings}
\markboth{Y.~BAZLOV and A.~BERENSTEIN}{COCYCLE TWISTS}
\begin{abstract}

It is well known that central extensions of a group $G$ correspond 
to $2$-cocycles on $G$. 
Cocycles can be used to construct extensions of $G$-graded 
algebras via a version of the Drinfeld twist 
introduced by Majid. We show how $2$-cocycles can 
be defined for an abstract monoidal category $\C$, following 
%
Panaite, Staic and Van Oystaeyen.
A braiding on $\C$ leads to analogues of Nichols algebras in $\C$,
and we explain how the recent work on twists of Nichols algebras 
by Andruskiewitsch, Fantino, Garcia and Vendramin fits in this 
context. 

Furthermore, we propose an approach to twisting
the multiplication in braided doubles, which are a class of 
algebras
with triangular decomposition over $G$. Braided doubles are not
$G$-graded, but may be embedded in a double of a Nichols algebra,
where a twist may be carried out if careful choices are made.
This is a source of new algebras with triangular decomposition.
As an example, we show how to twist 
the rational Cherednik algebra of the symmetric group
by the cocycle arising from the Schur covering group, 
obtaining the spin Cherednik algebra introduced by Wang.   
\end{abstract}
\maketitle
\setcounter{tocdepth}{1}
\tableofcontents

\setcounter{tocdepth}{1}
\setcounter{section}{-1}

\section{Introduction}

\subsection{Background}

Cohomology and homology groups arise in diverse areas 
of modern mathematics. 
It seems fair to say that homological algebra owes its current level of abstraction to category theory, which arose from papers by 
Eilenberg and Mac Lane in 1940s, for example \cite{EilenbergMacLane}.  
However, certain specific cohomology groups 
were studied by algebraists much earlier. A notable example is the 
Schur multiplier $M(G)$ of a finite group $G$, introduced by Schur in 
1904--07 in his work \cite{Sch1904,Sch1907}  
on projective representations. 
Modern understanding of 
the finite abelian group $M(G)$ 
as the second cohomology group $H^2(G,\CC^\times)$
came from papers such as Green \cite{Green1956} in mid-1950s. 
This led to a number of generalisations and new interpretations  
of this important group, see for example the theoretical physics 
paper \cite{VafaWitten} by Vafa and Witten.


In Schur's approach, what is now known as a \emph{$2$\dash cocycle} on $G$
appears in the guise of a factor set 
of a projective representation of $G$. In general, factor sets
correspond to  
%
\emph{central extensions} of $G$.
The equation defining a factor set --- and a  $2$\dash cocycle ---
is simply a manifestation of the associative law of multiplication. 

Because of this latter point, $2$\dash cocycles can be used 
to construct central extensions and twists 
of associative algebras, not only groups. 
One may recall \emph{deformations} of an algebra $A$ over
a field $k$ 
in the sense of Gerstenhaber \cite{Gerstenhaber1964}; they
are algebras over a complete local ring 
with field of residues $k$, and are 
governed by Hochschild cocycles on $A$.  
On the other hand, when $A$ is a $G$\dash graded algebra, \emph{twists} of $A$ by 
$2$\dash cocycles on $G$ with values in $k^\times$ 
can be constructed.  



In the present paper, cocycle twists of associative algebras 
are the main object. We pursue two principal goals: 
\begin{itemize}
\item
to explain  what are the Schur multiplier and cocycle twists in the most general 
setting  of a \emph{monoidal category}, 
and to apply cocycle twists to Nichols-type algebras if the category is braided;
\item
to construct cocycle twists of \emph{braided doubles} (for which 
the above general approach is not sufficient), thus obtaining new algebras 
with triangular decomposition.  
\end{itemize}
The rest of this Introduction serves to explain what is achieved with 
respect to the above two goals. We outline the structure of the paper and 
give some illustrative examples.  

\subsection{A brief overview}

We begin by recalling the necessary background
on cocycles and the Schur multiplier
of a finite group $G$, and study the Schur multiplier of $S_n$ in some detail.  
We explain how to twist a $G$\dash graded algebra by a cocycle.
To work with algebras in a more general setting, for example algebras  over 
commutative rings, we invoke the definition of an abstract 
monoidal category $\C$ of which $G$\dash graded vector spaces
are an example. It turns out that the Schur multiplier of $G$ 
can be defined more generally purely in terms of the monoidal structure 
on $\C$, leading to a group $H^2_\ell(\mathcal C)$.
In the context of Drinfeld's approach (made explicit by Majid) 
to twisting the monoidal product $\tensor$, the
group $H^2_\ell(\C)$ represents only those twists that leave
$\tensor$ unchanged; it is thus a vertex group in the groupoid of 
more general twists. 

A braiding on a monoidal category $\C$ leads to an analogue of Nichols 
algebras in $\C$, and we show that this class of algebras is closed 
under cocycle twisting. Based on our earlier work, 
we explain the link between Nichols algebras and braided doubles,
 a class of algebras with triangular decomposition. 
This suggests an approach to cocycle twisting of braided doubles. 
However, unlike the functorial construction of a twist on 
a Nichols algebra, twisting a braided double 
depends on carefully made choices, and is not guaranteed to work in general. 

We show how to make the necessary choices for the braided double known as the rational 
Cherednik algebra of $S_n$, due to Etingof and Ginzburg.
This sees the Schur covering group $T_n$ of $S_n$ 
play the role of a reflection group, acting on a free 
$\CC[z]/\lgen z^2-1\rgen$\dash module of rank $n$ 
rather than a vector space, with the root system
$$
\{ \pm(e_i - ze_j)\ \mid\  1\le i,j\le n,\  i\ne j\}.
$$
The twisted version of the rational Cherednik algebra is isomorphic to 
an algebra earlier obtained by Wang. 
We leave the more general case of a complex reflection group to our upcoming paper 
\cite{BBnew}.

The paper contains a certain amount of survey material on 
topics related to cocycle twisting, Nichols algebras and algebras with triangular decomposition. 
This is to help to introduce the reader to 
these topics and to put our results in context.
  
\subsection{Cocycles}
 
In Sections \ref{sect:basic}--\ref{sect:cocycle}, 
we review the construction of the second cohomology group 
$H^2(G,\Gamma)$. Recall that the group $Z^2(G,\Gamma)$ of 
normalised $2$\dash cocycles consists of solutions 
$\mu\colon G \times G \to \Gamma$
to the functional equations
$$
\mu(g,h)\mu(gh,k)=\mu (g, hk)\mu(h,k),\qquad \mu(1,g) =\mu(g,1) =1
$$
where $g,h,k\in G$. 
Here $\Gamma$ is an abelian group, 
so that the pointwise multiplication in $Z^2(G,\Gamma)$ is 
automatically commutative. 
The group $H^2(G,\Gamma)$ is the quotient of $Z^2(G,\Gamma)$ modulo a suitably
defined subgroup of coboundaries. The central extension of $G$ by $\Gamma$ 
corresponding to $\mu\in Z^2(G,\Gamma)$ is constructed in two steps:
first, consider the trivial extension $\Gamma\times G$; then 
introduce a new, $\mu$\dash \emph{twisted}, group multiplication on $\Gamma\times G$.

We would like to construct extensions of algebras, so  our main case is $\Gamma=R^\times$, the 
multiplicative group of a commutative ring $R$. 
A cocycle $\mu\in Z^2(G,R^\times)$
can be used to twist the multiplication in a $G$\dash graded associative algebra
$(A,\cdot)$ over $R$, obtaining a new associative algebra $A_\mu=(A,\star_\mu)$ where 
$$
a\star_\mu b = \mu(g,h)ab,\qquad\text{if } a\in A_g,\ b\in A_h.
$$
If $R$ is itself an algebra over another commutative ring, say $\CC$, then 
a $G$\dash graded $\CC$\dash algebra $A$ can be trivially extended to an $R\tensor_\CC A$, then twisted by $\mu$. We denote the result by $\widetilde A_\mu$ and refer to it as the \emph{cocycle extension} of $A$ by $\mu$.

We note that the cocycle twist construction works for algebras 
that have $G$\dash grading. Of course, the group algebra of $G$ 
(or of any group
$H$ equipped with a group homomorphism $H\to G$) has $G$\dash grading. 
We will now look at two other classes of natural examples of $G$\dash graded 
algebras.

\subsection{Example: semidirect products,
central simple algebras, relative Brauer group}

\label{subsect:ex1}

If $A$ is an $R$\dash algebra with a covariant action of $G$, one can form 
the semidirect product algebra 
$$
A\lcprod G,
$$
see \ref{subsect:defproblem}. This has a natural 
$G$\dash grading where $ag$ ($a\in A$, $g\in G$) has degree $g$.
The algebra $A\lcprod G$  
can therefore 
be twisted by a cocycle $\mu\in Z^2(G,R^\times)$, giving
rise to a new algebra $(A\lcprod G)_\mu$.

The following important example of twisted $A\lcprod G$ 
arose in 1930s in the study of central simple algebras over a field. 
Let $K/F$ be a finite Galois extension of fields with Galois group $G$. 
Each $\mu\in Z^2(G,K^\times)$ gives rise to the $F$\dash algebra
$$
(K\lcprod G)_\mu,
$$
which turns out to be a \emph{central simple $F$\dash 
algebra} split by $K$. Moreover, the map which sends $\mu\in Z^2(G,K^\times)$ to the class of $(K\lcprod G)_\mu$ in the 
\emph{relative Brauer group} $\mathit{Br}(K/F)$ is a group isomorphism between 
$H^2(G,K^\times)$ and $\mathit{Br}(K/F)$. 
In particular, the algebra $K\lcprod G$ (where $\mu=1$ is the trivial cocycle) 
is isomorphic to the matrix algebra $M_n(F)$ (which represents 
the trivial class in the Brauer group), with $n=|G|$. 
Central simple $F$\dash algebras of the form $(K\lcprod G)_\mu$ 
are called \emph{crossed products}. Their story is now classical, 
with details given in algebra textbooks such as 
\cite[ch.\ 17]{DummitFoote3rdEd}. 
They are a good illustration of the cocycle twist technique for 
semidirect products, but we will not consider them in the present paper.

\subsection{Example: Nichols algebras over $G$}

\label{subsect:ex2}

\emph{Nichols algebras}  over $G$ are a class 
of $G$\dash graded algebras, which are quite different from the 
semidirect products in the previous example, and are of a much more
recent origin. 
Nichols algebras play an important role in the present paper.
To define  a Nichols algebra, one needs a \emph{Yetter\dash Drinfeld module} 
$Y$ over $\CC G$, which is a $G$\dash graded space with a compatible $G$\dash 
action, see \ref{def:YD}. The grading and the action give rise to a 
\emph{braiding} on $Y$, i.e., an invertible linear map
$\Psi\colon Y\tensor Y \to Y \tensor Y$ which solves the braid equation 
$$
(\Psi\tensor \id_Y)(\id_Y\tensor\Psi)(\Psi\tensor \id_Y)
=(\id_Y\tensor\Psi)(\Psi\tensor \id_Y)(\id_Y\tensor\Psi)
$$
(both sides are endomorphisms of $Y^{\tensorpow 3}$). 
The Nichols algebra of $Y$ is defined as  
$$
\mathcal B(Y)=T(Y)/\ker I_\Psi,
$$
meaning that it is generated by $Y$ as an algebra, 
with relations determined in a certain way 
by $\Psi$.
In fact, $I_\Psi\subset T(Y)$ is the kernel of the braided symmetriser 
attached to $\Psi$, originally due to Woronowicz
\cite{Wor}; and $\mathcal B(Y)$ is a graded $G$\dash graded algebra. 

Following a suggestion of Kaplansky, Nichols algebras were so named
by Andruskiewitsch and Schneider 
in	\cite{AS} because their precursor arose in the work of
	Nichols \cite{N}. 
Andruskiewitsch and Schneider pioneered the use of Nichols algebras
in classifying 	pointed Hopf algebras; their comprehensive survey
\cite{ASpointed} accurately reflects the state of the art at the time.
	A key property of Nichols algebras --- that they are
	\emph{braided Hopf algebras}  --- was not addressed in \cite{N} as the
	formalism of Hopf algebras in braided categories 
was yet to be introduced.
Nichols algebras as braided Hopf algebras appeared in 
	Majid \cite{Mfree} where they arise naturally 
	in the context of a duality pairing in a braided
	monoidal category. 
	
Note that the relations in the Nichols algebra 
$\mathcal B(Y)$ are given implicitly. Hence, finding the
Hilbert series of $\mathcal B(Y)$ or even determining whether 
$\dim \mathcal{B}(Y)<\infty$ can be highly non-trivial
(and important in Hopf algebra classification problems).
In fact, in one of the key examples of a Nichols algebra 
for $G=S_n$, considered below, these questions are still open for 
$n\ge 6$. 
Classification of 
\emph{finite\dash dimensional} Nichols algebras over a 
finite group $G$ 
is a major direction of research into Nichols algebras.

It turns out that cocycle twists of Nichols algebras over $G$ 
are again Nichols algebras. 
For cocycles with values in the base field, this follows for example from a result of Majid and Oeckl \cite{MO}. 
Cocycle twists are thus a useful tool in studying Nichols algebras, 
because they produce new Nichols algebras with the same Hilbert series. 
In the present paper we perform a cocycle twist of a Nichols algebra
in a more general context of a braided monoidal category.

%
%
%
%
%
%
%
\subsection{The lazy cohomology of a monoidal category}
\label{subsect:monoidal}

In Section~\ref{sect:lazy}, we deal with a far\dash reaching 
generalisation of the 
Schur multiplier $M(G)$ of a finite group $G$ ---  
the \emph{lazy cohomology} of an abstract monoidal category 
$(\mathcal C,\tensor)$, denoted $H^2_\ell(\mathcal C,\tensor)$. 
This is a group which is not necessarily abelian. 
The $M(G)=H^2(G,\CC^\times)$ becomes a particular case of
a lazy cohomology group 
by considering $\mathcal C=\Comod G$, the category of $G$\dash graded
$\CC$\dash vector spaces, with a standard tensor product $\tensor=\tensor_\CC$. 

If the category $(\mathcal C,\tensor)$ is linear over a field $k$, 
then it (subject to some technical assumptions)
embeds in a category of comodules of a Hopf algebra $H$ over $k$. 
This is the subject of the Lyubashenko\dash Majid reconstruction theory, see 
\cite[chapter 9]{Mbook}. In particular, 
the group algebra $\CC G$ is reconstructed from the monoidal 
category $\Comod G$. 
For a Hopf algebra $H$ over $k$, 
the lazy cohomology was introduced by Bichon and Carnovale \cite{BiCar} 
based on a work \cite{SchauenburgHopfBimodules} by Schauenburg.
Even more generally, for  arbitrary monoidal categories
the lazy cohomology is defined in terms of \emph{laycles} (lazy
cocycles) 
by Panaite, Staic and Van Oystaeyen 
\cite{PSV}.  
Such  general approach allows one to work, for example, with a category 
of $G$\dash graded $R$\dash modules where $R$ is a commutative ring. 
Among other possible areas of application, let us 
point out the monoidal category of sets
as used, for 
instance, in Gateva\dash Ivanova -- Majid \cite{IvanovaMajidLevelTwo}. 

\subsection{The group of lazy cocycles vs.\ the groupoid of Drinfeld twists}

The categorical construction of cocycles and twists
should be compared to the Drinfeld twists,  introduced by Majid in 
\cite[section 2.3]{Mbook} for a module category of a Hopf algebra $H$ 
over $k$. If $\Delta\colon H \to H\tensor H$ denotes the coproduct on $H$, then  Drinfeld twists essentially are invertible elements 
$F\in H\tensor H$ such that $\Delta_F(h):=F(\Delta h)F^{-1}$ 
is again a coproduct on the same algebra $H$. This condition leads to an
equation on $F$ which is weaker than the quantum Yang\dash Baxter equation. 

One should note that Drinfeld twists in general do not form a group under multiplication. Indeed, in \cite[2.3]{Mbook} Majid defines the second 
non\dash abelian cohomology of a Hopf algebra as a set not a group. 
We observe, however, that Drinfeld twists form a \emph{groupoid}. 
The vertices of the groupoid are all possible coproducts $\Delta$ on a given
algebra $H$. The arrows between two coproducts 
$\Delta$ and $\Delta'$ are all the Drinfeld twists $F\in H\tensor H$ 
such that $\Delta'=\Delta_F$. 

In this picture, the lazy cocycles of the module category of the Hopf 
algebra $H$ (with fixed coproduct $\Delta$) 
are precisely the Drinfeld twists $F\in H\tensor H$ such that 
$\Delta_F=\Delta$. (These twists commute with $\Delta h$ for all $h\in H$ and
are sometimes called \emph{invariant twists}, see 
Guillot -- Kassel \cite{GuillotKassel}.)
Thus, the group of lazy cocycles in $H\tensor H$  
is a vertex group of the groupoid of Drinfeld 
twists. 

This picture should generalise to abstract categories, 
where to a category $\mathcal C$ there should be associated a groupoid 
whose vertices are monoidal products $\tensor$ on $\C$, and arrows are
natural isomorphisms $\tensor \to \tensor '$ that satisfy the abstract 
Drinfeld twist constraint. 

Although we do not consider non-invariant Drinfeld twists in the present paper, 
we would like to say a few words about them in this Introduction. 
Drinfeld twists are used in the quantum group theory and modern theoretical physics
to deform the associative product on a group\dash graded algebra, 
more generally a module algebra $A$  of a Hopf algebra  $H$. 
If $F\in H\tensor H$ is a twist, the action of $F^{-1}$ on 
$A\tensor A$ intertwines the original product $\cdot$ on $A$ with a new
			associative product $\star=\star_F$. 
			If $G$ is a finite group, 
			a  $G$\dash graded algebra is the same as an algebra with a 
covariant 	action of the
			commutative Hopf algebra $\CC(G)=\Fun(G,\CC)$. The twist equation 
and 			invertibility for $\mu\in \CC(G)\tensor \CC(G)\cong \CC(G\times G) 
$			mean precisely that $\mu$ is a $2$\dash cocycle with values in
			$\CC^\times$.  
			
			The Drinfeld twist expresses deformation quantisation in a Hopf
			algebraic language, and is a powerful formalism for introducing
			noncommutativity. See the survey physics paper \cite{Aschieri} 
where the Moyal quantisation is explained in terms of a Drinfeld twist, 
and \cite{BrainMajid} where the twistor theory and the Penrose\dash Ward 
transform are quantised using Drinfeld twists.

%
%
%

\subsection{Braided categories. Cocycle twists of braidings. Nichols algebras}

In Section~\ref{sect:nichols}, we work with a monoidal category 
with a \emph{braiding}, which is a type of a commutativity constraint.
Braidings on monoidal categories arise naturally in the theory of quantum 
groups. 
In the reconstruction picture where monoidal categories 
are thought of as (subcategories of) categories of modules 
over a Hopf algebra $H$, see \ref{subsect:monoidal}, 
braidings correspond to quasitriangular structures on Hopf algebras, which 
essentially are universal solutions to the quantum Yang\dash Baxter equation. 
It turns out that cocycles on a monoidal category $\C$ 
can be used to twist braidings on $\C$: the group $Z^2(\C)$ of 
cocycles acts on the class of braidings. 

It is in Section~\ref{sect:nichols} that we introduce the 
categorical analogue of Nichols algebras, discussed above 
in \ref{subsect:ex2}. If $\Psi$ is a braiding on a monoidal category 
$\C$, then, subject to some restrictions, to each object of $V$ of $\C$ 
there is associated an algebra $\mathcal B(V,\Psi)$ in $\C$. 
The Nichols algebras over a group $G$ are a particular case of this construction
where $\mathcal C=\YD G$ 
is the category of Yetter\dash Drinfeld modules over $G$.
However, our technique of working with general monoidal categories 
allows us to consider Nichols algebras over a commutative ring $R$, 
something which has not been done so far in the Hopf algebra literature.  

\subsection{Yetter-Drinfeld modules and their twists.
Main examples for $S_n$}

In Section~\ref{sect:YD} we explain the construction of the 
category of Yetter\dash Drinfeld modules. 
This is a braided monoidal category, 
and it can be said that raison d'\^etre of Yetter\dash Drinfeld modules 
is that that they provide an example of such a category, 
extremely useful in 
the quantum group theory. It is a result of Majid 
that the category $\YD H$ of Yetter\dash Drinfeld modules 
over a Hopf algebra $H$, finite\dash dimensional over a field $k$, 
is equivalent to the category of modules over the Drinfeld double $D(H)$ 
of $H$, which is a standard example of a quasitriangular Hopf algebra. 

In the present paper we only consider Yetter\dash Drinfeld modules 
over a group algebra of $G$. We do not assume $k$\dash linearity, 
working in a slightly
more general situation of free $R$\dash modules where $R$
is a commutative ring; this leads to the 
braided monoidal category $\YDE R G$. 
It turns out that every cocycle $\mu\in Z^2(G,R^\times)$ 
can be viewed as a lazy cocycle in $Z^2(\YDE R G)$. 
So, unlike in the case of an abstract monoidal category, 
twists can be made quite explicit in $\YDE R G$. 
And it is for Yetter\dash Drinfeld modules that the result, mentioned above 
in \ref{subsect:ex2}, holds: if the Nichols\dash 
type algebra $\mathcal B(Y)$, where $Y\in \Ob(\YDE R G)$, is 
twisted by $\mu$, the result is $\mathcal B(Y_\mu)$ where $Y_\mu\in\Ob(\YDE R G)$ 
is a certain other Yetter\dash Drinfeld module called the twist of $Y$ by $\mu$. 
In the $k$\dash linear situation, twisting of Yetter\dash Drinfeld 
modules is explored 
by Andruskiewitsch, Fantino,  Garc\'{\i}a  and Vendramin  
in \cite{AFGV}.
It should be noted that the twist of a Yetter\dash Drinfeld module is a functor, 
in particular  a completely canonical construction. 

Section~\ref{sect:YD} is concluded by an example where working over 
a ring $R$ which is not a field is essential. 
We put $R=\CC \cycl_2=\CC[z]/\lgen z^2-1\rgen$ and construct 
a Yetter\dash Drinfeld module $(\widetilde{Y_n})_{[1,z]}$ for $RS_n$. 
Here $[1,z]$ denotes a certain non\dash trivial cocycle 
in $Z^2(S_n,\cycl_2)$, which comes from the Schur covering group
$T_n$ of $S_n$.
The corresponding Nichols\dash type algebra 
$\mathcal B(\widetilde{Y_n})_{[1,z]}$
is a simultaneous cocycle extension of two Nichols algebras over $S_n$
that appeared in literature. One is 
$\mathcal B(Y_n)$  (see \cite{B,BBadv}),  
related to the Fomin\dash Kirillov algebra
$\mathcal E_n$, the coinvariant algebra of $S_n$ and 
the rational Cherednik algebra. The other Nichols algebra is $\Lambda_n^w$,   
due to Majid \cite{MajidQuadratic},
related to a quantum differential calculus on the group $S_n$.
In a recent paper by Vendramin \cite{V} it was shown that
these two Nichols algebras are related to each other by a 
twist by a cocycle (with scalar values).
 
This prepares the ground for the cocycle extension of the rational Cherednik 
algebra, constructed at the end of the present paper. 

\subsection{Algebras with triangular decomposition}

In the second part of the paper (sections 
\ref{sect:bdoubles}--\ref{sect:cherednik}) we show how  
the cocycle extension and cocycle twisting theory can be applied 
to produce new \emph{algebras with triangular decomposition}.  
These are associative algebras with a presentation of a specific form. 
Suppose that $A$ is an algebra with three chosen subalgebras $A_{-1}$,
$A_0$ and $A_1$. For $i=-1,0,1$, let $X_i$ be a basis of $A_i$ over
the ground field. If $A_{-1}A_0$ and
$A_0A_1$ are subalgebras with bases $X_{-1}X_0$ and $X_0X_1$,
respectively, 
and $X_{-1}X_0X_1$ is a basis of $A$, then $A$ is said to have 
triangular decomposition into subalgebras $A_{-1}$, $A_0$ and $A_1$. 

This property strongly affects the representations $A$, and 
algebras with triangular decomposition account for many
extensively studied objects in modern representation theory. 
Consider the following examples (over $\CC$): 

--- the universal enveloping algebras, $U(\mathfrak g)$, of 
semisimple Lie algebras $\mathfrak g$;  

--- their quantisations $U_q(\mathfrak g)$, known as Drinfeld\dash 
Jimbo quantum groups;

--- and, more recently, double affine Hecke algebras of Cherednik
\cite{C} and rational Cherednik algebras of Etingof and Ginzburg
\cite{EG}.  
Crucially, one can induce representations of $A$ from representations
of $A_0$. This allows one to work with the category $\mathcal O$,
defined for each of the above classes of algebras.  
 
The triangular decomposition property puts significant constraints on  the relations between the chosen generators $X_{-1}\cup X_0\cup X_1$ of $A$. For example, the triangular decomposition of $U(\mathfrak g)$ is a consequence of the Poincar\'e\dash Birkhoff\dash Witt theorem; the latter holds because the defining commutation relations in $U(\mathfrak g)$ come from a Lie bracket that obeys the Jacobi identity. The role of a (generalised) Jacobi identity is explained by Braverman and Gaitsgory in \cite{BG}. 

To study a broad class of defining relations that lead to a triangular decomposition of $A$, we introduced a class of algebras called \emph{braided doubles} in \cite{BBadv}. In a braided double $A=A_{-1}A_0A_1$, the middle subalgebra $A_0=H$ is a Hopf algebra, and $A_{\pm 1}$ are graded $H$\dash module algebras generated in degree $1$. The relations between $H$ and $A_{\pm 1}$ are semidirect product relations. A braided double $A$ is written as 
$$
  A \cong A_{-1}\tensor H \tensor A_1 = A_{-1}\lcprod H \rcprod A_1 
$$
as a vector space, with commutation relations of specific type between $A_1$ and $A_{-1}$.
We formally define braided doubles below in \ref{subsec:bd}. 
Enveloping algebras and Cherednik algebras in the example above are braided doubles, where the Hopf algebra $H$ is either a universal enveloping algebra of an abelian Lie algebra, $H=U(\mathfrak h)$, or a group algebra, $H=\CC G$. 
Braided doubles over $\CC G$ were further investigated in \cite{BBselecta}.

\subsection{Twists of braided doubles: the key ingredients}

In general, a braided double over a group $G$ is a $G$\dash module algebra 
but not a $G$\dash graded algebra, so cannot be twisted by a cocycle on $G$.
Our aim is to develop an approach which would produce 
algebras with triangular decomposition and capture the behaviour of ``real'' 
cocycle twists of $G$\dash graded algebras as fully as possible.

There is a special class of braided doubles, found in 
\cite{BBadv} and termed 
the \emph{braided Heisenberg doubles}.
These have triangular decomposition of the form 
$$ \mathcal H_Y = \mathcal B(Y)\lcprod \CC G\rcprod \mathcal B(Y^*),
$$
where $\mathcal B(Y)$, $\mathcal B(Y^*)$ are \emph{Nichols algebras}
of two dually paired Yetter\dash Drinfeld modules for $\CC G$.
To introduce $G$\dash grading on the 
braided Heisenberg double, we prove Theorem~\ref{thm:weyl}: $\mathcal H_Y$ is 
isomorphic to a semidirect product of $G$ and Majid's braided Weyl algebra of 
$Y$ \cite{Mfree}. This leads to a non\dash trivial extension of $G$\dash 
grading from $\CC G$ onto $\mathcal H_Y$, as shown in 
Proposition~\ref{prop:twist-h}.

This allows us to twist a braided Heisenberg double over $G$ by a cocycle of 
$G$. We stress that this procedure is again functorial and a cocycle twist
$(\mathcal H_Y)_\mu$
of a braided Heisenberg double $\mathcal H_Y$ is defined in a canonical way. 

In contrast to the above, 
a twist of braided doubles which are not Heisenberg
is not functorial. In fact, we do not define a way to twist an abstract 
braided double $A$. Instead, we consider 
$A$ together with a braided double morphism
$$
A\xrightarrow{f} \mathcal H_Y
$$
into a braided Heisenberg double. That such morphisms 
often exist, and are under favourable conditions injective, 
is the meaning of the Embedding Theorem \ref{thm:embedding} based 
on a result from \cite{BBadv}. 

Assume that the cocycle $\mu$ 
takes values in some  abelian group $\Gamma$. Then every $G$\dash graded 
$\CC$\dash 
algebra $B$ has a \emph{cocycle extension} $\widetilde B_\mu$, which is simply 
the twist by $\mu$ of the $\CC\Gamma$\dash algebra $\CC\Gamma\tensor_\CC B$. 
Now a $\mu$\dash extension of the braided double $A$, covering the morphism 
$f$, is a braided double morphism 
$$
\widetilde A \xrightarrow{\widetilde f}(\widetilde{\mathcal H_Y})_\mu ,
$$
which specialises to $f$ modulo the augmentation ideal 
$\CC\Gamma_+$ of $\CC \Gamma$. 
Here, $\widetilde A$ is required to be just $\CC\Gamma\tensor_\CC A$ as 
a $\CC\Gamma$\dash module, to reflect what is happening with 
$G$\dash graded algebras. 

Although a lot of information about the algebra $\widetilde A$ can be read off
the relations in $(\widetilde{\mathcal H_Y})_\mu$, for example 
the main commutator relation, there are still choices to be made 
in order to construct $\widetilde A$. 
In fact, the relations in $\widetilde A$ are dictated by a particular lift 
of the $f$\dash images of generators of $A$ to $(\widetilde{\mathcal H_Y})_\mu$.
If such a lift is chosen incorrectly, the algebra 
$\widetilde A$ ``blows up'' and is much bigger than the required
$\CC\Gamma\tensor_\CC A$. At the moment, we do not know whether 
a ``correct'' lifting exists for every non\dash trivial cocycle $\mu$.

A cocycle twist of a braided double $A$ by $\mu\in Z^2(G,\CC^\times)$ is 
defined as a result of taking a cocycle extension of $A$ by 
a cocycle with values in a finite cyclic group $\cycl_k=\langle z\mid
z^k=1\rangle$, then specialising to $z=q$ where $q\in \CC^\times$, 
$q^k=1$. Here we are again motivated by the properties of cocycle twists
of $G$\dash graded algebras: 
all of them are quotients of cocycle extensions by cyclic groups, by a result 
we prove here as Corollary~\ref{cor:twist}.

\subsection{Application to rational Cherednik algebras}

The last section of the paper, Section~\ref{sect:cherednik}, is devoted to 
constructing a cocycle extension of the rational Cherednik algebra
$H_{0,c}$ of the group $S_n$. By results of \cite{BBadv}, there is a morphism 
$$
H_{0,c}\xrightarrow{f} \mathcal H_{Y_n}
$$
where $Y_n$ is the Yetter\dash Drinfeld module for $S_n$ mentioned above.
We construct an extension of $H_{0,c}$ by the cocycle $\mu=[1,z]$ which covers 
the morphism $f$. We know the algebra 
$ (\widetilde{\mathcal H_{Y_n}})_{[1,z]}$ from what is done in 
Section~\ref{sect:YD}, but we need to find a specific lift
of the \emph{Dunkl elements}
$\theta_1,\ldots,\theta_n\in Y_n$ to 
$(\widetilde{Y_n})_{[1,z]}$. 
This is achieved by choosing 
$\widetilde \theta_1,\ldots,\widetilde \theta_n\in 
(\widetilde{Y_n})_{[1,z]}$ in such a way that 
their specialisation at $z=-1$ are Majid's flat connections 
in the Nichols algebra $\Lambda_n^w$.

The resulting algebra $\widetilde H_{0,c}$ over $\CC \cycl_2$ can be seen 
to coincide with an algebra constructed by Khongsap 
and Wang in \cite{KhongsapWang2}. Our approach via cocycles 
should work for other imprimitive complex reflection groups $G(m,p,n)$
resulting in covering Cherednik algebras attached to those groups. 
We leave this generalisation to the upcoming paper \cite{BBnew}.

\subsection{Notation}
\label{subsect:notation}

All algebras are associative with $1$. 
Angular brackets $\lgen\, \rgen$ denote a two\dash sided ideal of a
given algebra, generated by the elements listed between the brackets.  
They are not the same as the brackets in the $\langle
\mathit{generators} \mid \mathit{relations} \rangle$ form of a
presentation of a group; the context always helps to distinguish
between these.
We use the symbols $\act$ and $\ract$ to denote a  left, respectively
right,  action of a group $G$ on $X$ where
$X$ is a set, an abelian group, a vector space or an algebra; we
always assume that $G$ acts by automorphisms of~$X$.  


\section{Cocycles, central extensions, the Schur multiplier, twisted group algebras}

\label{sect:basic}

In this section we define $2$\dash cocycles on a finite group $G$ 
and review the related constructions. 
Let $\Gamma$ be an abelian
group, written multiplicatively. 
We denote by $\Fun(G^n,\Gamma)$ the set
of all functions from $G^n$ (the set of all $n$-tuples of elements of
$G$) to $\Gamma$. We write $\cycl_k$ to denote the cyclic group of
order $k$. That is, $\cycl_k=\langle z\mid
z^k=1\rangle$. We write $\CC^\times$ for the multiplicative group
$\CC\setminus \{0\}$.

\subsection{The group $Z^2(G,\Gamma)$ of cocycles}

The set $\Gamma \times G = \{(z,g): z\in \Gamma, g\in G\}$ has the
obvious group structure, the direct product of groups $\Gamma$ and
$G$. Suppose that we want to ``deform'' the group multiplication on
$\Gamma\times G$ in such a way that $\Gamma$ is still central.  
Let $\mu\in \Fun(G\times G,\Gamma)$. 
The product  $\star=\star_\mu$ on the set 
$\Gamma \times G$ given by 
$$
(z,g)\star(z', h)=(zz' \mu(g,h), gh)\qquad\text{for $z,z'\in \Gamma$,
  $g,h\in G$,}  
$$
is a group multiplication if $\mu$ satisfies the following equations:
$$
\mu(g,h)\mu(gh,k)=\mu (g, hk)\mu(h,k),\qquad \mu(1,g) =\mu(g,1) =1
\qquad\text{for all }g,h,k\in G. 
$$
A solution $\mu\in \Fun(G^2,\Gamma)$ of these equations is called a
\emph{normalised $2$-cocycle} on $G$ with values in $\Gamma$, see
\cite[ch.\ IV, 3]{Br}.     
The abelian group of normalised $2$-cocycles with respect to pointwise
multiplication is denoted $Z^2(G,\Gamma)$. 

If $\phi\in\Fun(G,\Gamma)$ is any function such that $\phi(1)=1$, then  
$d\phi(g,h)=\phi(h)\phi(gh)^{-1}\phi(g)$
is a normalised $2$-cocycle called the \emph{coboundary} of
$\phi$. Coboundaries form a subgroup $B^2(G,\Gamma)\subset
Z^2(G,\Gamma)$. Two cocycles $\mu, \nu\in Z^2(G,\Gamma)$ that coincide
modulo $B^2(G,\Gamma)$ are said to be \emph{cohomologous}. The abelian
group 
$$
H^2(G,\Gamma)=Z^2(G,\Gamma) / B^2(G, \Gamma)
$$
is the second cohomology group of $G$ with coefficients in $\Gamma$. 
Readers familiar with cohomology of groups should note that at this stage, 
the group $\Gamma$ of coefficients has a trivial action of $G$.

\subsection{Central extensions of $G$ by $\Gamma$}
\label{subsect: cex}

A \emph{central extension} of $G$ by $\Gamma$ is a short exact sequence 
$$
1\to \Gamma \xrightarrow{\iota} E \xrightarrow{\pi}  G \to 1
$$
of groups such that $\iota(\Gamma)$ is a subgroup of the centre of $E$. 
Central extensions play an important role in group theory, for example
in the classification of finite simple groups \cite{Atlas}.  
Another central extension, $1\to \Gamma
\xrightarrow{\iota'}E'\xrightarrow{\pi'}G\to 1$, is \emph{equivalent}
to the one given above if there is a group homomorphism (necessarily
an isomorphism) $E\to E'$ which intertwines $\iota,\iota'$ and
$\pi,\pi'$.  
%

A cocycle $\mu\in Z^2(G,\Gamma)$ gives rise to a specific central
extension of $G$ by $\Gamma$,
$$
1\to \Gamma \xrightarrow{\iota} \widetilde G_\mu \xrightarrow{\pi}  G \to 1
\qquad \text{with the group $\widetilde G_\mu=(\Gamma\times G, \star_\mu)$,}
$$ 
where $\iota(z)=(z,1)$ and $\pi((z,g))=g$. 
We will refer to this central extension, and often to the group
$\widetilde G_\mu$ itself, as the \emph{extension of $G$ by the
  cocycle $\mu$}.  
To simplify notation, we will omit $\iota()$, identifying $z\in
\Gamma$ with $(z,1)\in \widetilde G_\mu$ and viewing $\Gamma$ as the
subgroup $\ker \pi$ of $\widetilde G_\mu$.

Cohomologous cocycles give rise to equivalent cocycle
extensions. Indeed, let
$\mu$ and $\nu=\mu \, d\phi$ be cohomologous cocycles in
$Z^2(G,\Gamma)$, where $\phi\in\Fun(G,\Gamma)$, $\phi(1)=1$. 
Then the the bijective map $\Gamma\times G \to \Gamma\times G$,
given by $(z,g)\mapsto (z\phi(g), g)$ is a group homomorphism between
$\widetilde G_\mu$ and $\widetilde G_\nu$, which affords an 
equivalence of the cocycle extensions $\widetilde G_\mu$ and $\widetilde
G_\nu$ of $G$.  

Conversely, if the two cocycle extensions $\widetilde G_\mu$ and $\widetilde
G_\nu$ are equivalent, $\mu$ and $\nu$ are cohomologous. By
definition, an equivalence of extensions is necessarily a map
$\Gamma\times G \to 
\Gamma \times G$ which sends $(z,1)$ to $(z,1)$ and $(1,g)$ to
$(\phi(g),g)$ for some $\phi\in\Fun(G,\Gamma)$, $\phi(1)=1$. That this
map is a group homomorphism is easily seen to imply that $\nu = \mu \, d\phi$. 

Furthermore, every central extension of $G$ by $\Gamma$ is equivalent
to a cocycle extension. Indeed, let $T\colon G \to E$ be a set\dash theoretic
\emph{section} of the extension  
$1\to \Gamma \xrightarrow{\iota} E \xrightarrow{\pi}  G \to 1$. That
is, $T$ is a map of sets such that $\pi T=\id_G$. We further require
$T$ to be a \emph{normalised} section, that is, $T(1)=1$. Define
$\mu\in\Fun(G^2,\Gamma)$ by  
$$
           T(g)T(h)=\mu(g,h)T(gh) \qquad\text{for }g,h\in G,
$$
observing that $T(g)T(h)T(gh)^{-1}\in \ker \pi = \Gamma$. 
That is, $\mu$ measures the failure of the section $T$ to be a group
homomorphism. 
Then it is easy to see that $\mu\in Z^2(G,\Gamma)$ is a normalised
$2$\dash cocycle, and the 
map $\widetilde G_\mu \to E$ given by $(z,g)\mapsto zT(g)$, affords an
equivalence of the extensions $\widetilde G_\mu$ and $E$.  

The above argument establishes an important fact: 
the cohomology group $H^2(G,\Gamma)$ is in bijection with 
the equivalence classes of central 
extensions of $G$ by $\Gamma$. To a normalised cocycle $\mu$
representing a cohomology class in $H^2(G,\Gamma)$, the bijection
associates the central extension $\widetilde G_\mu$. 
In particular, the trivial cocycle $1\in Z^2(G, \Gamma)$ corresponds
to the \emph{split extension} $\widetilde G = \Gamma\times G$.  
This bijection is a particular case of Schreier's theorem, see 
\cite[Theorem 7.34]{Rotman}.


\subsection{The Schur multiplier of $G$}

Cocycle extensions of $G$  yield groups larger than
$G$. A variation of the above construction can be used to deform the
multiplication in the \emph{group algebra} $\CC G$ of $G$.
This is explained below in \ref{subsect:twist}; 
the input is a normalised $2$\dash cocycle on $G$ with values in
$\CC^\times$. 
Such cocycles, modulo coboundaries, form an important  
finite abelian group 
$$
    M(G) := H^2(G,\CC^\times),
$$ 
known as the \emph{Schur multiplier} of the finite group $G$. 
(In the case where $G$ is infinite, which we do not consider here,
there are competing non\dash equivalent definitions of the Schur
multiplier of $G$.) 

\begin{example}[Schur multiplier of a finite abelian group]
\label{ex:abelian}
To compute $M(G)$ where $G$ is a finite abelian group, 
one can use an alternative definition of $M(G)$ via group homology,
$M(G)=H_2(G, \ZZ)$, see \cite[ch.\ I, Proposition 5.5]{BT}. 
By \cite[Theorem 6.4(iii)]{Br}, this is the same as the abelian group
$\wedge^2 G$, the quotient of the group $G\tensor_{\mathbb Z}G$ by the
subgroup generated by $g\tensor_{\mathbb Z}g$, $g\in G$.
The group $G$ is non\dash canonically isomorphic to
$G^*=\{$homomorphisms from $G$ to $\CC^\times\}$, and  
a more careful analysis identifies $M(G)$ with $\wedge^2
G^*$ which is the group of all \emph{bicharacters} $\mu\in
\Fun(G^2,\CC^\times)$ modulo the group of symmetric
bicharacters. A function $\mu\colon G^2 \to \CC^\times$ is a
bicharacter, if $\mu(g,-)$ and $\mu(-,g)$ are homomorphisms $G\to
\CC^\times$ for all $g\in G$; it is symmetric if $\mu(g,-)=\mu(-,g)$
for all $g$. Bicharacters form an abelian group with respect to
pointwise multiplication. 

To give an explicit example, consider $G=\cycl_p^n$ where $p$ is
prime. A convenient presentation of $G$ is 
$$
G=\langle \gamma_1,\ldots,\gamma_n\mid \gamma_i^p=1, \gamma_i \gamma_j
= \gamma_j \gamma_i, 1\le i,j\le n\rangle.
$$
Let $\omega$ be a primitive $p$th root of unity in
$\CC^\times$.  
If $1\le i<j\le n$, denote by $b_{ij}$ the bicharacter of $G$ defined
on generators by 
$$
b_{ij}(\gamma_k,\gamma_l) = \begin{cases}
\omega, & \text{if } (k,l)=(j,i), \\
1, &\text{if } (k,l)\ne (j,i).
\end{cases}
$$
It is easy to see that the bicharacters  $b_{ij}$, $1\le i<j\le n$, generate
an abelian group isomoprhic to $\cycl_p^{\binom{n}{2}}\cong \wedge^2
G$ which does not 
contain non-trivial symmetric bicharacters.  Thus, 
$$
M(\cycl_p^n)\cong \cycl_p^{\binom{n}{2}}.
$$
\end{example}

Schur multipliers of many finite non\dash abelian groups are known;
see in particular \cite{Griess} for Schur multipliers of all finite
simple groups. The Schur multipliers of symmetric groups and
alternating groups were determined by Schur in \cite{Sch}. 

\subsection{Twisted group algebras}
\label{subsect:twist}

Using a cocycle $\nu\in Z^2(G,\CC^\times)$, one can construct 
a \emph{twisted group algebra}
(a cocycle twist $\CC G_\nu$ of the group algebra $\CC G$ by $\nu$).
The underlying vector space of $\CC G_\nu$ is $\CC G$, and
the associative multiplication  
$\star_\nu$ is defined on the basis $\{g\mid g\in G\}$ by 
$g\star_\nu h = \nu(g,h)gh$. This is the same formula as in the
cocycle extension, but the values of $\nu$ are viewed as scalars.   
Cohomologous cocycles lead to isomorphic twists of $\CC G$, thus
twisted group algebras are parametrised, up to isomorphism, by
elements of the Schur multiplier $M(G)$ of $G$.

In a sense, cocycle extensions of $G$ are more general than cocycle
twists of $\CC G$, because  
any cocycle twist of $\CC G$ is a quotient of the group algebra of an extension of $G$ by a cyclic group:

\begin{lemma}
\label{lem:cyclic}
Let $G$ be a finite group and $\nu\in Z^2(G,\CC^\times)$. There exist
a finite cyclic group $\cycl_m=\langle z| z^m=1\rangle$, a cocycle
$\mu\in Z^2(G,\cycl_m)$ and an $m$th root of unity $q\in \CC^\times$,
such that  
$$
\CC G_\nu\cong \CC \widetilde G_\mu/\lgen z-q\rgen.
$$ 
\end{lemma}
\begin{proof}
By \cite[lemma 7.65]{Rotman}, the subgroup $B^2(G,\CC^\times)$ of
$Z^2(G,\CC^\times)$ has a complement $M$ which is a finite subgroup of
$Z^2(G,\CC^\times)$ isomorphic to $H^2(G,\CC^\times)$. That is,
$Z^2(G,\CC^\times)$ is a direct product of $M$ and
$B^2(G,\CC^\times)$. Let $\sigma\in M$ be the cocycle cohomologous to
$\nu$ so that $\CC G_\nu\cong \CC G_\sigma$, and let $m$ be the order
of $\sigma$ in the finite group $M$. Then all values of $\sigma$ are
$m$th roots of unity in $\CC^\times$, hence of the form $q^k$ where
$q$ is a primitive $m$th root of $1$ and $0\le k \le m-1$. Define  
$\mu\in Z^2(G,\cycl_m)$ by $\mu(g,h)=z^k$ whenever $\sigma(g,h)=q^k$. 
This guarantees that the map $\pi_q\colon \CC \widetilde G_\mu \to \CC
G_\sigma$ given by $(z^k,g)\mapsto q^k g$ is an algebra
homomorphism. It is surjective and its kernel contains $z-q\cdot 1$,
so $\dim \CC \widetilde G_\mu / \lgen z-q\rgen \ge \dim \CC \widetilde
G_\mu /\ker \pi_q = |G|$. On the other hand, $(z^k, g)$ is the same as
$q^k(1,g)$ in $\CC \widetilde G_\mu /\lgen z-q\rgen$, hence $\dim \CC
\widetilde G_\mu /\lgen z-q\rgen\le |G|$.  
We conclude that $\ker \pi_q=\lgen z-q\rgen$ so that $\CC \widetilde
G_\mu /\lgen z-q\rgen \cong \CC G_\sigma \cong \CC G_\nu$.  
\end{proof}

%

\subsection{Example: the Clifford algebra is a twisted group algebra}


Let $A$ be the group algebra of the group $K=\cycl_2^n$. That is, $A$
is generated by    
$\gamma_1,\ldots,\gamma_n$ subject to relations $\gamma_i^2=1$ and
$\gamma_i\gamma_j=\gamma_j \gamma_i$ for all $i,j\in
\{1,\ldots,n\}$. There is another, noncommutative, associative product $\star$
on the vector space $A$ such that  
$$
\gamma_i\star \gamma_i=1, \qquad  \gamma_i \star \gamma_j = \gamma_i
\gamma_j, \qquad \gamma_j \star \gamma_i = -\gamma_i\gamma_j, \qquad
1\le i <j \le n.  
$$
Note that $(A,\star)$ is the Clifford algebra of rank $n$ (named 
after W.~K.~Clifford). 
It is a
cocycle twist of 
$\CC K$ by the cocycle $\mu\in Z^2(K,\CC^\times)$ given on generators by  
$$
    \mu(\gamma_i,\gamma_j) = \begin{cases} +1&\text{ if $i\le j$,} \\
-1&\text{ if $i>j$,}
\end{cases}
$$
and extended to $K\times K$ as a bicharacter. In the notation of
Example~\ref{ex:abelian}, the cocycle $\mu$ is $\prod_{1\le i < j \le
  n}b_{ij}$.

As prescribed by Lemma~\ref{lem:cyclic}, the Clifford algebra is the
quotient,  modulo the relation $z=-1$, of the group algebra of a non-abelian 
central extension of $\cycl_2^n$ by $\cycl_2=\{1,z\}$.


\section{The Schur multiplier of $S_n$}

We now describe the above constructions more explicitly 
in the case of $G=S_n$, the symmetric group of degree $n$. Historically, this was the first group 
for which the Schur multiplier was found;  see Schur \cite{Sch1904,Sch1907,Sch}.

\subsection{Projective representations} 

Initially, let $G$ be an arbitrary finite group. 
Projective representations of $G$ are closely associated to central
extensions of $G$ and $2$\dash cocycles on $G$. 
Let $V$ be a finite\dash dimensional vector space
over $\CC$. A \emph{projective representation} of $G$ on $V$ is a
homomorphism $\rho\colon G\to \PGL(V)$ where $\PGL(V)$ is the quotient
group of $\GL(V)$ modulo its centre $\CC^\times \cdot \id_V$. By choosing
representatives for $\rho(g)$, $g\in G$, in $\GL(V)$ such that
$\rho(1)=\id_V$ --- or choosing a normalised section $\PGL(V)\to
\GL(V)$ --- we may as well regard $\rho$ as a map $G\to \GL(V)$
which satisfies  
$$
   \rho(g)\rho(h)=\mu(g,h)\rho(gh), \qquad g,h\in G,
$$
for a function $\mu\in \Fun(G^2,\CC^\times)$ called a \emph{normalised
  factor  set}. The associative law implies that $\mu\in
Z^2(G,\CC^\times)$. Thus, normalised factor set is just another name for a
normalised $2$\dash cocycle with values in $\CC^\times$.  

It is easy to see that projective representations of $G$ with a given factor
set $\mu$ are the same as representations of the twisted
group algebra $\CC G_\mu$, introduced above.

Projective representations arise naturally in the representation theory of
groups. An example is the Clifford theory which deals with extending a
representation of a normal subgroup $H$ of $G$ to $G$, see the paper
\cite{Clifford} by A.~H.~Clifford.  

\subsection{Schur covers}

A way to construct projective representations of $G$ is as follows. 
Let $1\to \Gamma \to E \to G \to 1$ be a central extension of $G$ with
normalised section $T\colon G \to E$. Let $r\colon E \to \GL(V)$ be a
representation of $E$. Then $\rho=rT\colon G \to \GL(V)$ is a projective   
representation of $G$. In this case, it is said that the
projective representation $\rho$ of $G$ is \emph{lifted} to the
(linear) representation $r$ of $E$. 

By a result of Schur \cite{Sch1904}, for any finite group $G$ there
exists at least one central extension
$$
1 \to \Gamma \hookrightarrow E \to G \to 1 
$$
of $G$, such that \emph{all} the projective
representations of $G$ are lifted to $E$. Extensions with this
\emph{projective lifting property} where the
order of $E$ is smallest possible are known
as \emph{Schur covers} of $G$, while the group $E$ is termed a Schur covering
group, or a representation group of $G$. 

Equivalently, a central extension is a Schur cover if 
it is a stem extension and $\Gamma\cong M(G)$. 
\emph{Stem extension} means that $\Gamma\subseteq [E,E]$, the
derived subgroup of $E$. This and other equivalent characterisations of
Schur covers can be found in \cite[ch.\ II, III]{BT} and
\cite[ch.\ 7]{Rotman}. 

\subsection{Group algebras of Schur covering groups}

Two Schur covers of $G$ may not be equivalent as extensions, and
moreover not isomorphic as abstract groups. The simplest example is
$$
G=\cycl_2\times \cycl_2,
$$ 
the Klein $4$-group. By
Example~\ref{ex:abelian}, 
$$
M(\cycl_2\times \cycl_2)=\cycl_2.
$$
The group $G$ has two
non-isomorphic Schur covering groups: $D_8$, the dihedral group of
order $8$, and $Q_8$, the quaternion group, see \cite[Example
  7.17]{Rotman}. 


Nevertheless, any two Schur covers of $G$ are \emph{isoclinic} to each
other. Isoclinism of abstract groups, introduced by P.~Hall \cite{H}, is an
equivalence relation weaker that isomorphism: for example, all abelian
groups form one isoclinism class. See  \cite[ch.\ III]{BT} for the
definition of isoclinism of abstract groups and of central extensions. 

We will need the following observation about group algebras of
isoclinic groups. Let $H_1$, $H_2$ be two finite groups, and let
$m_i(n)$ denote the number of irreducible complex characters of the group
$H_i$ of degree $n$. If $H_1$ is isoclinic to $H_2$,  then by
\cite[Corollary 5.8]{BT}, $m_1(n)|H_2|_p=m_2(n)|H_1|_p$ for all $n$
and for all primes $p$. In particular, if $H_1$, $H_2$ have the same
order, $m_1(n)=m_2(n)$ for all~$n$. 

Recall from the representation theory of finite groups over $\CC$ 
that the group algebra $\CC H_i$ is
isomorphic to the direct sum of matrix algebras of size $d\times d$
where $d$ runs over the multiset of degrees of 
irreducible characters of $H_i$. That is, the function $m_i(n)$
determines the isomorphism class of the group algebra $\CC H_i$. We
conclude that if $H_1$ and $H_2$ are isoclinic groups of the same
finite order, $\CC H_1\cong \CC H_2$. 

In particular, this isomorphism holds if $H_1$ and $H_2$ are Schur
covering groups of the same finite group $G$. Another way to explain
the equality of $m_1(n)$ and $m_2(n)$ in this case is to observe that
irreducible representations of $H_i$ are, by the projective lifting
property, precisely the irreducible projective representations of
$G$. Hence we obtain the following

\begin{lemma}
\label{lem:isom}
Let $G$ be a finite group. The group algebras, over $\CC$, of all Schur covering
groups of $G$ are pairwise isomorphic. 
\qed
\end{lemma}

\subsection{The Schur covering groups of $S_n$}
\label{subsect:cover_symm}

Let now $G=S_n$.
We describe two non-equivalent covers 
of $S_n$, $n\ge 4$, found by Schur \cite{Sch} in
1911. The importance
of a Schur cover for us is that it allows us to
choose a normalised section and to compute a non\dash trivial $2$\dash
cocycle on $S_n$ with values in $M(S_n)$. Schur found that $M(S_n)$ is
trivial if $n\le 3$ and 
$$
M(S_n)=\cycl_2 \qquad\text{if }n\ge 4.
$$
It can be shown that there are no non-trivial cocycles on $S_n$
with values in $\cycl_p$, $p$ prime, $p\ne 2$.
   
As was explained above in~\ref{subsect: cex}, central extensions of
$S_n$ by 
$$
\cycl_2=\{1,z\}
$$ 
are parametrised, up to equivalence, by the
cohomology group $H^2(S_n,\cycl_2)$. 
One can show, in the same fashion as Schur did in \cite{Sch}, that
a cocycle $\mu\colon S_n\times S_n\to \cycl_2$ can be chosen, 
up to a coboundary, so that for all $i,k,l\in \{1,\ldots,n-1\}$, $k<l$,
$$
\mu(s_i,s_i)=\alpha, \qquad
\mu(s_k,s_l)=1, \qquad
\mu(s_l,s_k)=\beta, \qquad \alpha,\beta\in\{1,z\}.
$$
Here $s_i=(i\ i+1)$ is a simple transposition. 
%
It follows that
$$
  H^2(S_n,\cycl_2)\cong \cycl_2\times \cycl_2 \qquad\text{for }n\ge 4,
$$
where a cohomology class in $H^2(S_n,\cycl_2)$ is represented
by $[\alpha,\beta]$.  
We will abuse the notation slightly and write the cohomology class 
$[\alpha,\beta]$ instead of its representative $\mu$. 
The four pairwise non-equivalent extensions 
of $S_n$ by $\cycl_2$ are then as described in

\begin{theorem}[Schur \cite{Sch}]
\label{thm:schur}
If $ \alpha,\beta\in\{1,z\}$, the group
\begin{align*}
  (\widetilde {S_n})_{[\alpha,\beta]} = \langle t_1,\ldots,t_{n-1},z \mid
     \ & t_i^2=\alpha,\ z^2=1,\ zt_i=t_i z,
  \\ & t_j t_{j+1}t_j = t_{j+1}t_j t_{j+1},\ t_k t_l = \beta t_l t_k:
  \\ & 1\le i \le n-1,\ 1\le j \le n-2,\ 1\le k <l-1\le n-2 
  \rangle
\end{align*}
is a central extension of $S_n$ by $\cycl_2$ via the map 
$(\widetilde{S_n})_{[\alpha,\beta]} \sur S_n$ given by $t_i
\mapsto (i\ i+1)$, $z\mapsto 1$. The groups 
$(\widetilde{S_n})_{[1,z]}$ and  $(\widetilde{S_n})_{[z,z]}$ 
are Schur covers of $S_n$, whereas 
 $(\widetilde{S_n})_{[1,1]}$ and $(\widetilde{S_n})_{[z,1]}$ 
are not stem extensions of $S_n$.
\qed   
\end{theorem}
To clarify the notation used, recall from \ref{subsect: cex} that 
as a set, 
the group $(\widetilde {S_n})_{[\alpha,\beta]} $ is the cartesian 
product $\cycl_2\times S_n$. It is generated by $z=(z,1)$ 
and by $t_i=(1,s_i)$.

Let us briefly comment on why the cocycles denoted $[1,1]$ and $[z,1]$ do not 
give rise to stem extensions of $S_n$.
Clearly, $(\widetilde{S_n})_{[1,1]}=S_n\times \cycl_2$ 
is the trivial extension. One can show that the cocycle $[z,1]$ is 
given by the formula
$$
[z,1](g,h) = z^{(\ell(g)+\ell(h)-\ell(gh))/2},\qquad g,h\in S_n.
$$
Here $\ell(g)$ is the length (the number of inversions) of the
permutation $g$, and one notes that $\ell(g)+\ell(h)-\ell(gh)$ is
always even.  
Write $S_n$ as $A_n\lcprod C_2$ where $A_n$ is the alternating
group and $C_2=\{\id, (12)\}$. Then $(\widetilde
{S_n})_{[z,1]}=A_n\lcprod C_4$ is obtained in a straightforward
way from the extension $1\to C_2 \to C_4 \to C_2 \to 1$.  
Both for $(\widetilde
{S_n})_{[1,1]}$ and for $(\widetilde
{S_n})_{[z,1]}$, the element $z$ lies in the kernel of the
extension but not in the derived subgroup.

One is left with two  non-isomorphic Schur covering groups of 
 $S_n$  (assuming $n\ge 4$): in Schur's notation, 
 $$
T_n =(\widetilde {S_n})_{[1,z]},\qquad
T'_n =(\widetilde {S_n})_{[z,z]}.
 $$ 
By Lemma~\ref{lem:isom}, 
$$
\CC T_n \cong \CC T'_n \qquad\text{(this isomorphism is not canonical).}
$$

\subsection{The spin symmetric group}
\label{subsect:spinsymm}

Let us now apply the cocycle twist to deform the associative product
on the group algebra $\CC S_n$. Such a twist is given, up to an
algebra isomorphism, by an element of $H^2(S_n,\CC^\times)=\cycl_2$.  Note that 
the map 
$$
H^2(S_n,\cycl_2)\to H^2(S_n,\CC^\times), \qquad
z\mapsto -1
$$ 
is not injective. Indeed, although the cocycle
$[z,1]\in Z^2(S_n,\cycl_2)$ considered above is not
cohomologous to the trivial cocycle $[1,1]$, its image in  $Z^2(S_n,\CC^\times)$   
can be written as $\mathrm i^{\ell(g)}\mathrm i^{-\ell(gh)}\mathrm
i^{\ell(h)}$ where $\mathrm i=\sqrt{-1}\in\CC^\times$. 
This is clearly a coboundary. Therefore, up to isomorphism, there is
only one non\dash trivial cocycle twist of $\CC S_n$: 
$$
(\CC S_n)_{[1,-1]} = \CC T_n / \lgen z+1 \rgen 
\quad \cong \quad 
(\CC S_n)_{[-1,-1]} = \CC T_n'/\lgen z+1 \rgen. 
$$
Although the isomorphism between $\CC T_n$ and $\CC T_n'$ is not
explicit, there is a straightforward isomorphism between the quotients
$\CC T_n / \lgen z+1\rgen$ and  $\CC T_n' / \lgen z+1\rgen$ which sends the
generator $t_j$ of $\CC T_n / \lgen z+1\rgen$ to $\mathrm i t_j'$ in  $\CC
T_n' / \lgen z+1\rgen$. 
 
The twisted group algebra $(\CC S_n)_{[1,-1]}\cong (\CC S_n)_{[-1,-1]}$ 
is referred to as the \emph{spin symmetric group} in Wang \cite{Wang}.  

\subsection{The $1$-cocycle $\chi_\mu$ and its calculation for $S_n$}
 
\label{subsect:chi} 
 
Given a cocycle $\mu\in Z^2(G,\Gamma)$ on a finite group $G$, 
we are interested in the values of the function 
$$
\chi_\mu(g,h)=\mu(g,h)\mu(ghg^{-1},g)\colon G\times G \to \Gamma.
$$
This formula for $\chi_\mu(g,h)$ guarantees that the function 
$$
G \to \Fun(G,\Gamma), \qquad g \mapsto \chi_\mu(g,-)
$$
is a $1$\dash cocycle, where $\Fun(G,\Gamma)$ is viewed as a  
 right $G$\dash module via the adjoint action of $G$ on itself. 
(We have not defined cocycles with coefficients in a non\dash trivial 
module, but see the definition in \cite[ch.\ III, 1, example 3]{Br}.)
We thus have a group homomorphism 
$$
Z^2(G,\Gamma)\to Z^1(G,\Fun(G,\Gamma)), \qquad
\mu\mapsto\chi_\mu.
$$
It is not difficult to check that it induces a homomorphism 
$$
H^2(G,\Gamma)\to H^1(G,\Fun(G,\Gamma)).
$$
The reason why we are interested in this homomorphism
will become apparent later in Section \ref{sect:YD} 
and has to do with twisting Yetter\dash Drinfeld modules. 

We have just analysed the group $H^2(S_n,\cycl_2)$, 
so now is a good point to calculate $\chi_\mu$ for future use. 
For $\mu=[1,z]$ in the notation of \ref{subsect:cover_symm}, 
the corresponding cocycle in 
$Z^1(S_n,\Fun(S_n,\cycl_2))$ 
was explicitly found by Vendramin \cite[Lemma 3.7 and Theorem 3.8]{V}:

\begin{theorem}[Vendramin] 
If $\sigma,\tau$ are transpositions in $S_n$, 
 $\tau=(i\,j)$, $i<j$, then 
$$
\chi_{[1,z]}(\sigma,\tau) = \begin{cases}
z,&\text{if }\sigma(i)<\sigma(j),\\
1,&\text{if }\sigma(i)>\sigma(j).
\end{cases}
$$
\qed
\end{theorem}
The expression for $\chi_{[z,z]}(\sigma,\tau)$ is more complicated: one can 
check that 
$$
\chi_{[z,z]}(\sigma,\tau) = z^{(\ell(\sigma\tau\sigma^{-1})-\ell(\tau))/2}
\chi_{[1,z]}(\sigma,\tau).
$$


\section{Cocycle twists and cocycle extensions of group-graded algebras}

\label{sect:cocycle}


\begin{definition}[$G$-graded algebra over $R$]

Let $R$ be a commutative ring and $A$ be an associative algebra over $R$.  
We say that $A$ is \emph{graded} by a finite group $G$, if $
A=\bigoplus_{g\in G}A_g
$,
where $\dirsum$ means a direct sum of $R$\dash modules;  
$A_gA_h\subset A_{gh}$ for all $g,h\in G$;  and the identity element
of $A$ is in $A_1$. 
\end{definition}
\begin{example}
The group algebra $R G$ is $G$\dash graded, with $(RG)_g=Rg$. 
\end{example}

\subsection{Cocycle twists of $A$} 
\label{subsect:cocycle_twists}

Let $A$ be a $G$\dash graded algebra over $R$. 
%
%
The multiplication in $A$ can be twisted by a $2$\dash cocycle on $G$
with values in $R^\times$.  
Every cocycle $\mu\in Z^2(G,R^\times)$ gives rise to a new associative
product $\star=\star_\mu$ on the underlying $R$\dash module $A$ given by  
$$
a\star_\mu b=\mu(g,h)ab\qquad\text{for}\ a\in A_g,\ b\in A_h,\ g,h\in G. 
$$
We refer to the algebra $A_\mu = (A,\star_\mu)$ 
as the \emph{cocycle twist} of $A$ by $\mu$.  

\subsection{Realisation of the cocycle twist via coaction}
\label{subsect:realisation}

The twists of the group algebra of $G$, discussed in the previous section,
are in a sense the main example of cocycle twists of a group-graded
algebra. We will now explain this. We 
will write the $G$\dash grading as a coaction: 

\begin{definition}[Coaction]
	\label{def:coaction}
	If $V$ is a $G$\dash graded $R$\dash module, $V=\bigoplus_{g\in G}V_g$,
	the \emph{coaction} of $G$ on $V$ is the $R$\dash module map 
	$$
		\delta\colon V \to V \tensor_R R G, \qquad 
		\delta(v)=v\tensor_R g\quad\text{if }v\in V_g.
	$$
\end{definition}

It follows from the definition of a $G$\dash graded algebra that 
the coaction
$\delta\colon A \to A \tensor_R R G$ is a homomorphism of algebras over $R$, where $A \tensor_R R G$ is viewed as a tensor product of two algebras (the tensor factors commute).  

Now let $\mu\in Z^2(G,R^\times)$. 
View the codomain of $\delta$ as the algebra $A\tensor_R R G_\mu$
where $A$ and $R G_\mu$ commute. 

\begin{lemma}\label{lem:realisation}
	$\delta$ is an algebra 
	isomorphism between $(A,\star_\mu)$ and the subalgebra $\delta(A)$ of
	$A\tensor_R R G_\mu$. 
\end{lemma}
\begin{proof}
	Note that $\delta$ is injective because $(\id_A\tensor_R 
\epsilon)\delta=\id_A$ where $\epsilon\colon R G\to R$ is the 
augmentation map. It is enough to check that 
	$\delta(ab)=\delta(a)\delta(b)$ where $a\in A_g$, $b\in A_h$. 
	Both sides are equal to $ab\tensor_R gh$.
\end{proof}
In other words, it is enough to twist $R G$,
and the twisted product on every $G$\dash graded $R$\dash algebra is induced
via the coaction.

\subsection{Cocycle extensions of $\CC$\dash algebras}
\label{subsect:extension}

In the rest of this section, we will look at the case where $R$ is a commutative algebra over $\CC$. 
Let $A_0$ be a $G$\dash graded algebra over $\CC$, 
and let $\mu\in Z^2(G,R^\times)$. The \emph{cocycle extension} of $A_0$ by
$\mu$ is an algebra over $R$. It can be constructed in two
steps. First, consider $R\tensor_\CC A_0$ as an algebra over $R$, 
with the $G$\dash grading coming from $A_0$. The algebra
$R\tensor_\CC A_0$  is a trivial extension of $A_0$ (extension of
scalars with no cocycle involved). 

 Then twist the
multiplication in $R\tensor_\CC A_0$ as in \ref{subsect:cocycle_twists},
obtaining $( R\tensor_\CC A_0)_\mu$. We denote the $R$\dash algebra $(
R\tensor_\CC A_0)_\mu$ by $(\widetilde A_0)_\mu$ and call it the
extension of $A_0$ by the cocycle $\mu$. Note that a cocycle extension
of a $\CC$\dash algebra by a cocycle with values in $R^\times$ is
always free as an $R$\dash module.

\subsection{Flat $\Gamma$-deformations of $\CC$\dash algebras. Specialisation}
\label{subsect:flat}

Let $A$ be an algebra over $\CC$ and $\Gamma$ be an abelian group.
We further restrict $R$ to being the group algebra $\CC \Gamma$. 
A  \emph{flat $\Gamma$\dash deformation} of $A$ is an algera
$\widetilde A$ over the ring $\CC \Gamma$ which 
is a free $\CC \Gamma$\dash module, 
together with an isomorphism 
$\widetilde A/\CC \Gamma_+ \widetilde  A\xrightarrow{\sim} A$
of algebras.   
Here $\CC \Gamma_+$ is the augmentation ideal of $\CC \Gamma$, that
is, the subspace of $\CC \Gamma$ spanned by $\{z-1 \mid
z\in\Gamma\}$. Note that $\CC \Gamma_+\widetilde A$ is a two\dash sided ideal of
$\widetilde A$. Cf.\ \cite[Definition 1.1]{E}. 
Any algebra $A$ has a trivial flat  $\Gamma$-deformation $\CC
\Gamma\tensor_\CC A$. 

In the case where $\Gamma$ is a cyclic group $\cycl_m=\langle z\mid
z^m=1\rangle$, it is convenient to denote the algebra $\widetilde A /
\CC \Gamma_+ \widetilde A$ by $\widetilde A|_{z=1}$. More generally, 
for any $q\in \CC^\times$ such that $q^m=1$, one has the algebra 
$$
\widetilde A | _{z=q} := \widetilde A / (z-q)\widetilde A
$$
over $\CC$. The algebra $\widetilde A | _{z=q}$ is the
\emph{specialisation} of $\widetilde A$ at $z=q$.


Note that any extension of $A$ by a cocycle $\mu\in Z^2(G,\Gamma)$ is a flat 
$\Gamma$\dash deformation of $A$. 
Lemma~\ref{lem:cyclic} and Lemma~\ref{lem:realisation} now 
imply
\begin{corollary}
\label{cor:twist}
	For every cocycle $\nu\in Z^2(G,\CC^\times)$, there exists a finite cyclic 
	group $\cycl_m={\langle z\mid z^m=1\rangle}$, a cocycle $\mu\in 
Z^2(G,\cycl_m)$ 
	and an $m$th root of unity $q\in \CC^\times$ such that 
	the twist $A_\nu$  of $A$
	is isomorphic to a specialisation of $\widetilde A_\mu$ at $z=q$. 
	\qed
\end{corollary}%




%
%
%


\section{Cocycle twists and the lazy cohomology of a monoidal category}

\label{sect:lazy}

A twist of $G$\dash graded algebras by a
cocycle representing an element of $M(G)$ can be viewed as an endofunctor 
of the category of $G$\dash graded algebras over $\CC$. 
%
The cocycle extension, too, should be understood as a functor, but
between two different categories of algebras.
This motivates us to review
the necessary formalism of algebras in
monoidal categories. 
We show how
a $2$\dash cocycle can be defined categorically as 
a natural automorphism of the monoidal product on a
category $\C$; it 
induces an endofunctor on the category $\C\hyp\Alg$ of algebras in $\C$.


\subsection{Algebras in a monoidal category}
\label{subsect:moncat}

Recall that a category $\C$ is \emph{monoidal} if it is equipped with
a monoidal product bifunctor  $\tensor\colon \C\times \C \to \C$   
and a unit object $\mathbb I$ that satisfy the axioms in Mac Lane
\cite[ch.\ VII]{MacLane}. The first section of Deligne -- Milne  
article \cite{DeM} also provides a clear and
concise introduction to monoidal categories.  
A standard example is $R\hyp \mathrm{Mod}$
where $R$ is a commutative ring, with $\tensor=\tensor_R$ and  
$\mathbb I = R$. All monoidal categories $\C$ that we consider are
equipped with a faithful functor to $R\hyp \mathrm{Mod}$ which  
preserves the monoidal product and the unit object. In particular, the
associativity constraint  
$\Phi\colon \tensor \circ (\tensor \times \id_\C) \cong \tensor \circ
(\id_\C \times \tensor )$ is  
canonical and will be omitted from formulae. 
We explicitly consider the most
basic case where $R\cong \CC \dirsum \CC \dirsum \ldots \dirsum 
\CC$  
is the group algebra $\CC \Gamma$ of a finite abelian group $\Gamma$,
and in the main applications we will have $\Gamma=\cycl_2$.    

An \emph{algebra}  (or \emph{monoid}) in a monoidal category $\C$ is a
triple $(A,m,\eta)$ 
where $A\in\Ob\ \C$ and the  multiplication morphism 
$m\colon A\tensor A\to A$ and the unit morphism $\eta\colon \mathbb I
\to A$ satisfy the usual associativity and unitality axioms. See 
\cite[ch.\ VII, 3]{MacLane}.
Morphisms between algebras in $\C$ are $\C$\dash morphisms
between the underlying objects 
that intertwine the multiplication morphisms and the unit morphisms. Thus,
algebras in a monoidal category $\C$ form a category $\C\hyp \Alg$
equipped with the forgetful faithful functor $(A,m,\eta)\mapsto A$
from $\C\hyp \Alg$ to $\C$. 
For example, $R\hyp \mathrm{Mod}\hyp \Alg$ is the
category of algebras over the commutative ring $R$. 

\subsection{Example: the free algebra of an object of $\C$}

\label{subsect:free}

Assume that the monoidal category $\C$ has countable direct sums which
are preserved by 
the monoidal structure: that is, $X\tensor \bigoplus_{i\in \mathbb N}
Y_i$ is a direct sum of $\{X\tensor Y_i\mid i\in \mathbb N\}$. 
Then every object $V\in\Ob \ \C$ gives rise to
a \emph{free algebra} $T(V)\in\Ob\ \C\hyp \Alg$. As an object in $\C$, 
$$
T(V)=\bigoplus_{n\in \mathbb N} V^{\tensorpow n}
$$
where the tensor powers are defined inductively by $V^{\tensorpow 0}=\mathbb 
I$ and
$V^{\tensorpow n}=V^{\tensorpow n-1}\tensor V$. By definition of a direct sum, a morphism $f\colon T(V)\to X$ in $\C$ 
is the same as a collection of morphisms $f|_{V^{\tensorpow n}}\colon V^{\tensorpow n}\to X$, so 
the multiplication
on the tensor algebra is the unique morphism $m\colon T(V)\tensor
T(V)\to T(V)$ such that $m|_{V^{\tensorpow m}\tensor V^{\tensorpow n}}$
is the canonical isomorphism onto $V^{\tensorpow m+n}$ given by
the associativity of $\tensor$. See
\cite[ch.\ VII, 3, Theorem 2]{MacLane}.  

Observe that $T$ is a functor from $\C$ to $\C\hyp \Alg$. If $f\colon V \to W$ 
is a morphism in $\C$, then 
$$
T(f)\colon T(V)\to T(W),\qquad T(f)|_{V^{\tensorpow n}} = f^{\tensor n},
$$
is the corresponding morphism in $\C$\dash $\Alg$. 

The following notion of a $2$\dash cocycle on a 
monoidal category $\C$ is the same as \emph{laycle} 
(lazy cocycle) in \cite{PSV}.

\begin{definition}[cocycle]
	Let $\mu\colon \tensor \to \tensor$ be an automorphism 
	of the monoidal product on 
$\C$; that is, a family of automorphisms $\mu_{X,Y}\colon X\tensor Y\to 
X\tensor Y$ natural in $X,Y$. It is a \emph{normalised $2$\dash cocycle on 
$\C$} if
$$\mu_{X,Y\tensor Z}(\id_X\tensor \mu_{Y,Z})
	= \mu_{X\tensor Y,Z}(\mu_{X,Y}\tensor \id_Z), 
	\qquad 
	\mu_{X,\mathbb I}=\mu_{\mathbb I,X}=\id_X
$$ 
	for all $X,Y,Z\in\Ob\ \C$.
\end{definition}
\begin{remark}
	Recall that naturality of $\mu_{X,Y}$ in $X,Y$ means that, 
	given two morphisms $f\colon X\to X'$ and $g\colon Y \to Y'$ in $\C$, one 
has 
	$(f\tensor g)\mu_{X,Y}=\mu_{X',Y'}(f\tensor g)$.
\end{remark}
\begin{remark}
	\label{rem:cocycle}	
	The cocycle condition allows us to consider
	$\mu_{X,Y,Z}\in \End(X\tensor Y \tensor Z)$ which is defined 
	as both 
	$\mu_{X,Y\tensor Z}(\id_X\tensor \mu_{Y,Z})$ and 
	$\mu_{X\tensor Y,Z}(\mu_{X,Y}\tensor \id_Z)$.
More generally, $\mu_{X_1,\ldots,X_n}$, $n\ge 3$,
is defined recursively as $\mu_{X_1,\ldots,X_{n-2},X_{n-1}\tensor X_n}
(\id_{X_1}\tensor \dots \tensor \id_{X_{n-2}}\tensor \mu_{X_{n-1},X_n})$ 
but has various other factorisations due to the cocycle condition. 
\end{remark}
 
\begin{definition}[cocycle twist]
	Let $(A,m,\eta)$ be an algebra in $\C$ and $\mu$ be a normalised $2$\dash 
cocycle on $\C$. Its \emph{twist by $\mu$} is defined to be the triple 
$(A,m\mu_{A,A},\eta)$.
\end{definition}

\begin{theorem}[the cocycle twist functor]
	\label{thm:functoriality}
	If $\mu$ is a normalised $2$\dash cocycle on a monoidal category $\C$,
the twist by $\mu$ of an algebra in $\C$ is again an algebra in $\C$. The 
twist by $\mu$ is a functor from $\C\hyp \Alg$ to $\C\hyp \Alg$ which is 
identity on morphisms. 
\end{theorem}
\begin{proof}
Consider the twist $(A,	m\mu_{A,A},\eta)$ of $(A,m,\eta)\in\Ob\ \C\hyp \Alg$. 
The morphism $m\mu_{A,A}$ is associative:
\begin{align*}
m\mu_{A,A}(m\mu_{A,A}\tensor \id_A)& =
m(m\tensor \id_A) \mu_{A\tensor A,A}(\mu_{A,A}\tensor \id_A)
\\
&= m(\id_A\tensor m)\mu_{A,A\tensor A}(\id_A\tensor \mu_{A,A})
= m\mu_{A,A}(\id_A\tensor m\mu_{A,A}),
\end{align*}
where the first and the third steps are by naturality of $\mu$, while
the second step is by associativity of $m$ and the definition of cocycle. 
To show that $m\mu_{A,A}$ is unital with respect to $\eta$, use naturality of 
$\mu$ and unitality of $m$.

It remains to show that a morphism $f\colon (A,m_A,\eta_A)\to (B,m_B,\eta_B)$ 
in $\C\hyp \Alg$ remains a morphism between $(A,m_A\mu_{A,A},\eta_A)$ and 
$(B,m_B\mu_{B,B},\eta_B)$. We are given that $fm_A=m_B(f\tensor f)$ and want 
to prove that $fm_A\mu_{A,A}=m_B\mu_{B,B}(f\tensor f)$. 
But this clearly follows from naturality of $\mu$. Furthermore, 
we are given that $f\eta_A=\eta_B$, but this same property shows that $f$ 
intertwines the unit maps of the twisted algebras.
\end{proof}

\begin{remark}
	\label{rem:multprod}
	Note that it is sometimes convenient to use the iterated product
	$m_n\colon A^{\tensorpow n}\to A$ for an algebra $A$, well\dash defined 
	by associativity of $m$: recursively, 
	$m_n = m_{n-1}(\id_A^{\tensorpow n-2}\tensor m)$ for $n\ge 2$. 
	We observe that the iterated product on the twist of $A$ by $\mu$
	is equal to $m_n \mu_{A,A,\ldots,A}$ where the $n$\dash fold version 
	$\mu_{A,A,\ldots,A}$ is defined in Remark~\ref{rem:cocycle}.
\end{remark}

\subsection{Coboundaries. The lazy cohomology of $\C$}

Let us denote the collection of normalised $2$\dash cocycles on a mono\-idal 
category $\C$ by $Z^2(\C)$. 
Observe that natural automorphisms of the bifunctor $\tensor$ form a group 
$\Aut \tensor$ under composition. (Warning: 
a group in this context may not be a set but is rather a class.) 

We have $Z^2(\C)\subset \Aut\tensor$. 
We claim that $Z^2(\C)$ is a subgroup of $\Aut\tensor$. 
Indeed, observe that if $\mu,\nu\colon \tensor \to \tensor$ are  natural 
transformations, then $\nu_{X,Y\tensor Z}$ and $\id_X\tensor \mu_{Y,Z}$ 
commute in $\End(X\tensor Y \tensor Z)$, precisely by naturality of $\nu$ in 
the second subscript. If  $\mu,\nu\in Z^2(\C)$, then 
\begin{align*}
\mu_{X,Y\tensor Z}\nu_{X,Y\tensor Z}
(\id_X\tensor \mu_{Y,Z}\nu_{Y,Z})
&=
\mu_{X,Y\tensor Z}(\id_X\tensor \mu_{Y,Z})
\nu_{X,Y\tensor Z}(\id_X\tensor \nu_{Y,Z} )
\\
&=
\mu_{X\tensor Y,Z}(\mu_{X,Y}\tensor \id_Z)
\nu_{X\tensor Y,Z}(\mu_{X,Y}\tensor \id_Z)
\\
&=
\mu_{X\tensor Y,Z}\nu_{X\tensor Y,Z}
(\mu_{X,Y}\nu_{X,Y}\tensor \id_Z),
\end{align*}
which shows that $Z^2(\C)$ is closed under composition of natural 
transformations. 
Clearly $Z^2(\C)$ is also closed under inverses, 
hence is a group, possibly non\dash abelian. 

We are going to define the notion of a $2$\dash coboundary. 
Let $\Aut \Id_\C$ be the group of all natural automorphisms of the identity 
functor of $\C$. We require all such automorphisms to be compatible 
with the associativity constraint of $\tensor$, 
to allow us to write $\phi_{X\tensor Y \tensor Z}$ for $\phi\in \Aut\C$, 
$X,Y,Z\in \Ob\ \C$. 
Define the map $d\colon \Aut \Id_\C \to \Aut \tensor$ 
by 
$$
(d\phi)_{X,Y}=\phi_{X\tensor Y}^{-1} (\phi_X\tensor \phi_Y),
$$
and define the \emph{$2$\dash coboundaries of $\C$} to be $B^2(\C)=\im 
d\subset  Z^2(\C)$. 
\begin{proposition}
$d$ is a group homomorphism, and 
$B^2(\C)$ is a central subgroup of $Z^2(\C)$. 
\end{proposition}
\begin{proof}
Note that for any $\phi,\psi\in\Aut\Id_\C$, 
$\psi_{X\tensor Y}^{-1}$ commutes with $\phi_X\tensor \phi_Y$
by naturality of $\psi$. This immediately implies that 
$d(\phi\psi)=(d\phi)(d\psi)$. Hence $d$ is a homomorphism, and $B^2(\C)$ is a 
subgroup of $\Aut\tensor$. 
Moreover, $d\phi\in Z^2(\C)$ because
\begin{align*}
(d\phi)_{X,Y\tensor Z}(\id_X\tensor(d\phi)_{Y,Z})
&=
\phi_{X\tensor Y\tensor Z}^{-1} (\phi_X\tensor \phi_{Y\tensor Z})
(\id_X\tensor \phi_{Y\tensor Z}^{-1} (\phi_Y\tensor \phi_Z))
\\
& =\phi_{X\tensor Y\tensor Z}^{-1} 
(\phi_X\tensor \phi_Y\tensor \phi_Z)
=
(d\phi)_{X\tensor Y,Z}
((d\phi)_{X,Y}\tensor \id_Z).
\end{align*}
It remains to show that $d\phi$ commutes with $\mu\in Z^2(\C)$.
But $\phi_{X\tensor Y}^{-1}$ commutes with $\mu_{X,Y}$ by naturality of 
$\phi$, and $\phi_X\tensor \phi_Y$ commutes with $\mu_{X,Y}$ by naturality of 
$\mu$ in both $X$ and $Y$. 
\end{proof}
\begin{definition}
	The \emph{lazy cohomology group} of a monoidal category $\C$
	is the group 
	$$
	H_\ell^2(\C):=Z^2(\C)/B^2(\C).
	$$
\end{definition}
This notion,
introduced by Panaite, Staic and Van Oystaeyen in 
\cite{PSV}, generalises the construction of the lazy cohomology group
of a Hopf algebra, introduced by Bichon and Carnovale \cite{BiCar} 
based on a work \cite{SchauenburgHopfBimodules} by Schauenburg. 
Currently there are not enough explicit calculations of this recently 
introduced analogue of the Schur multiplier, but see   
Bichon -- Kassel
\cite{BiKas} and references therein. 
If $G$ is a finite group, 
the lazy cohomology of the category 
of $G$\dash graded $\CC$\dash vector spaces is 
$M(G)$, a finite abelian group. 
However, the monoidal category of 
$\CC G$\dash modules is a natural object to be considered, 
and 
there are examples of finite groups $G$ for which $H^2_\ell(\CC 
G\text{-modules})$ is a non\dash abelian group. 

Note that if two cocycles on $\C$ are in the same class in 
$H^2_\ell(\C)$, 
the twist functors they induce on $\C\hyp \Alg$ are naturally 
isomorphic.

\subsection{Example: the category $R\hyp \Comod G$}
\label{subsect:RMG}
Let $R$ be a commutative ring and $G$ be a finite group. 
Denote by $R\hyp \Comod G$ the monoidal category of $R$\dash modules with 
$G$\dash 
grading. In other words (see Definition~\ref{def:coaction}), 
\begin{itemize}
	\item objects of 
	$R\hyp \Comod G$ are pairs $(V,\delta)$ where $V$ is an $R$\dash module 
and 
	$\delta\colon V \to V\tensor_R RG$ is an $R$\dash module map;
	\item
	morphisms $f\colon (V,\delta_V)\to (W,\delta_W)$ are $R$\dash module 
maps $f\colon V\to W$ such that $(f\tensor \id_{RG})\delta_V=\delta_Wf$;
	\item 
	the monoidal product of $(V,\delta_V)$ and $(W,\delta_W)$ is $V\tensor_R 
W$ where the coaction $\delta_{V\tensor_R W}$ is $\delta_V\tensor_R \delta_W$ 
followed by the group multiplication map $RG\tensor_R RG\to RG$;
\item 
the unit object is $\mathbb I=R$ with coaction $\delta=\id_R\tensor_R 1_G$. 
\end{itemize}
Then each normalised $2$\dash cocycle $\mu\in Z^2(G,R^\times)$ 
gives rise to a categorical cocycle on $R\hyp \Comod G$ via
$$
\mu_{X,Y} = \sum_{g,h\in G} 
\mu(g,h)(\id_X\tensor_R \pi_g\tensor_R \id_Y\tensor_R 
\pi_h)(\delta_X\tensor_R \delta_Y)\qquad\in\End_R(X\tensor _R Y),
$$
where $\pi_g\colon RG \to R$ is an $R$\dash module map 
defined by $\pi_g(g)=1$, $\pi_g(k)=0$, $g,k\in G$, $g\ne k$. 

The category $R\hyp \Comod G\hyp \Alg$ is the category of $G$\dash graded 
algebras 
over $R$. It remains to note that  the categorical  twist 
of $A\in \Ob\ R\hyp \Comod G\hyp \Alg$ by $\mu_{X,Y}$ 
is the same as the twist $A_\mu$ considered in the previous section. 

\subsection{The cocycle extension functor for $\CC$-algebras}

The categorical setup accommodates the cocycle extension construction 
from \ref{subsect:extension}, which will now be interpreted as a functor. 
We write $\Comod G$ to denote $\CC\hyp \Comod G$. 

Recall that if $\mu\in Z^2(G,\Gamma)$ where $\Gamma$ is an abelian group,
and $A$ is an algebra over $\CC$, 
the cocycle extension 
$\widetilde A_\mu$ of $A$ is the algebra $(\CC\Gamma\tensor A)_\mu$
over $\CC \Gamma$. Clearly, $\CC\Gamma\tensor -$ is a faithful 
functor from $\Comod G\hyp \Alg$ to $\CC\Gamma\hyp \Comod G\hyp \Alg$. 
Recall also the specialisation at $z=q$ which 
is a functor from $\CC\cycl_m\hyp \Comod G\hyp \Alg$ to $\Comod G\hyp \Alg$.
We write this functor as $(\,\cdot\,)|_{z=q}$. 
%
We summarise the so far developed 
categorical interpretation of extensions and twists 
of $\CC$\dash algebras in the following
	\begin{theorem}
		\label{thm:cocycle_ext}
		Let $G$ be a finite group and $\Gamma$ be an abelian group.
		
		1. To each cocycle $\mu\in Z^2(G,\Gamma)$ there corresponds a cocycle
		extension functor 
		$$
		\widetilde{(\,\cdot \,)}_\mu \colon \Comod G\hyp \Alg \to
		\CC\Gamma\hyp \Comod G\hyp \Alg,   
		\qquad A \mapsto \widetilde A_\mu.
		$$
		The algebra $\widetilde A_\mu$ is a flat $\Gamma$\dash deformation of
		$A$. 
		
		2. If $\mu_1$, $\mu_2$ are cohomologous cocycles, the functors $
		\widetilde{(\,\cdot \,)}_{\mu_1}$ and $ \widetilde{(\,\cdot
			\,)}_{\mu_2}$ are naturally isomorphic. 
			
			3. To each cocycle $\nu\in Z^2(G,\CC^\times)$ there corresponds a
			cocycle twist functor 
			$$
			(\,\cdot \,)_\nu \colon \Comod G\hyp \Alg \to \Comod 
			G\hyp \Alg,
			\qquad A\mapsto A_\nu.
			$$
			
			4. For each $\nu\in Z^2(G,\CC^\times)$ one can find a finite cyclic
			group $\cycl_m=\langle z\mid z^m=1\rangle$, an $m$th root of unity 
			$q\in\CC^\times$ and a cocycle
			$\mu\in Z^2(G,\cycl_m)$ such that there is a natural isomorphism  
			$$
			(\,\cdot \,)_\nu \cong (\, \cdot \,)|_{z=q} \circ \widetilde{(\, 
				\cdot
				\,)}_\mu
				$$
				of functors. 
				\qed
			\end{theorem}

%
%


\section{Cocycle twists of braidings and Nichols algebras}

\label{sect:nichols}

In this section, we work with a monoidal category 
with a \emph{braiding}.
If $\Psi$ is a braiding on an abelian monoidal category 
$\C$, to each object of $V$ of $\C$ 
there is associated an algebra $\mathcal B(V,\Psi)$ in $\C$. 
This is a general categorical version of \emph{Nichols algebra}, 
a quantum group\dash theoretic construction which has become
quite popular among Hopf algebra theorists. 
We introduce  $\mathcal B(V,\Psi)$ in a category which is not necessarily 
$k$\dash linear for a field $k$; we use this construction in 
the category $R\operatorname{-}\YD G$ of Yetter\dash Drinfeld 
modules for a group $G$ over a commutative ring $R$. 
One of the goals of this section is to show how $\mathcal B(V,\Psi)$ 
behaves under a cocycle twist. 

\subsection{Hopf algebras in a braided category}

A \emph{braiding} on a monoidal category $\C$ is a natural isomorphism
$\Psi\colon \tensor \to \tensor^{\mathrm{op}}$ that satisfies the hexagon
condition. A monoidal category equipped with a braiding is a
\emph{braided} category.  
The standard reference is Joyal -- Street \cite{JS}, see also an
expository paper \cite{Savage} by 
Savage. The braiding yields, for each pair of objects $X,Y\in\Ob
\ \C$, an isomorphism  
$$
\Psi_{X,Y}\colon X\tensor Y \to Y \tensor X.
$$
For example, the category $R\hyp \mathrm{Mod}$ has the trivial braiding,
$\Psi=\tau$, defined by $\tau_{X,Y}(x\tensor_R y)=y\tensor_R x$. 

A braiding on a monoidal category $\C$ gives rise to a monoidal
product on the category $\C\hyp \Alg$.  
Namely, if $(A,m_A,\eta_A)$, $(B,m_B,\eta_B)\in\Ob(\C\hyp \Alg)$,
the object $A\tensor B$ 
becomes an algebra  
with respect to the  multiplication map 
$$
m_{A\tensor B} = (m_A\tensor m_B)(\id_A\tensor
\Psi_{B,A}\tensor \id_B)\colon A \tensor B \tensor A \tensor B \to A \tensor B
$$
and the unit map $\eta_A\tensor \eta_B$. 
The associativity of $m_{A\tensor B}$ follows from the hexagon condition 
and the naturality of $\Psi$. 
The unit object of $\C\hyp\Alg$ is $(\mathbb I, \mathbb I\tensor 
\mathbb I \iso 
\mathbb I, \id_{\mathbb I})$.

A \emph{bialgebra in a braided category} $\C$ is an algebra
$(B,m,\eta)$ in $\C$ equipped with two extra morphisms, the coproduct
$\Delta\colon B \to B \tensor B$ which is coassociative in the usual
sense and is a morphism of algebras, and
$\epsilon\colon B \to \mathbb I$ satisfying the counit 
condition with respect to $\Delta$. Note that the braiding $\Psi$ is
involved in the algebra structure on $B\tensor B$ and therefore
affects the definition of $\Delta$. A \emph{Hopf algebra in $\C$} is a
bialgebra equipped with a $\C$\dash morphism $S\colon B \to B$
satisfying the definition of an antipode. It is easy to observe that 
the trivial algebra $\mathbb I$ as above is also a Hopf algebra in $\C$.

An introduction to Hopf 
algebras in braided categories can be found in Majid \cite[ch.\ 14]{Mprimer}. 

\subsection{Duals and rigidity. Half-adjoints. Kernels of a pairing}
\label{subsect:pairing}

Let $X$ be an object in a monoidal category $\C$. A \emph{right dual}
of $X$ is an object $X^*$ together with two morphisms, 
$$
\ev_X\colon X^*\tensor X \to \mathbb I,
\qquad
\coev_X\colon \mathbb I \to X \tensor X^*
$$
which satisfy 
\begin{align*}
(X\cong \mathbb I \tensor X \xrightarrow{ \coev_X \tensor \id_X} X\tensor 
X^*\tensor X \xrightarrow{\id_X\tensor \ev_X} X) \quad &= 
\quad X\xrightarrow{\id_X} X,
\\
(X^* \cong X^*\tensor \mathbb I \xrightarrow{\id_{X^*}\tensor \coev_X} 
X^*\tensor X \tensor X^* 
\xrightarrow{\ev_X\tensor \id_{X^*}} X^*) \quad &= \quad X^* 
\xrightarrow{\id_{X^*}} X^*.
\end{align*}
A right dual is unique up to an isomorphism. A left dual $\vphantom{X}^*X$ of 
$X$ is a right dual of $X$ in the category $(\C,\tensor^{\mathrm{op}},\mathbb 
I)$. In particular, $X$ is a right dual of $\vphantom{X}^*X$. 

An object $X$ is \emph{rigid} if it has right and left duals in $\C$. Rigid 
objects form a full subcategory of~$\C$. 


If $X$ and $Y$ are rigid objects in a monoidal category $\C$, each morphism $X 
\tensor Y \xrightarrow{f} Z$ has two \emph{half\dash adjoints}:
\begin{align*}
X \xrightarrow{f^\flat} Z \tensor Y^* \quad & = \quad 
(X \xrightarrow{\id_X\tensor \coev_Y}X\tensor Y\tensor Y^*
\xrightarrow{f\tensor \id_{Y^*}} 
Z\tensor Y^*),
\\
Y\xrightarrow{f^\sharp} \vphantom{X}^*X\tensor Z \quad &= \quad
(Y\xrightarrow{\coev_{^*X}\tensor \id_Y}
\vphantom{X}^*X\tensor X \tensor Y  
\xrightarrow{\id_{^*X}\tensor f} \vphantom{X}^*X\tensor Z). 
\end{align*}
A \emph{pairing} between two objects $X$, $Y$ in $\C$ is a morphism 
$\kappa\colon X\tensor Y \to \mathbb I$. 
For example, $\ev_X$ is a pairing between $X^*$ and $X$ if $X$ is
rigid. If $\C$ is additive, there is a zero pairing $X\tensor Y
\xrightarrow{0}\mathbb I$ between any two objects. 
If the category $\C$ is abelian, we can define the right and left
\emph{kernels of the pairing} $\kappa$ by  
$$
\ker_l\kappa = \ker \kappa^\flat, \qquad \ker_r \kappa = \ker \kappa^\sharp. 
$$ 
Note that a pairing $\kappa\colon X\tensor Y \to \mathbb I$ gives rise
to a pairing between $X\tensor X$ and $Y\tensor Y$:
$$
\kappa(\id_X\tensor \kappa\tensor \id_Y)\colon X\tensor X \tensor Y
\tensor Y \to \mathbb I,
$$
that is, the rightmost copy of $X$ is paired with the leftmost copy of
$Y$, and vice versa. In the same fashion $\kappa$ defines a pairing
between $X^{\tensorpow n}$ and $Y^{\tensorpow n}$
by first acting on the innermost copy of $X\tensor Y$ in $X^{\tensorpow n}
\tensor Y^{\tensorpow n}$.

\subsection{Hopf duality pairing}

Let $A$, $B$ be two bialgebras in a braided category $\C$. A
\emph{duality pairing} between $A$ and $B$ is a pairing $\kappa\colon
A\tensor B \to \mathbb I$ such that 
\begin{align*}
&\kappa(\Delta_A\tensor \id_B\tensor\id_B)=\kappa(\id_A\tensor
m_B)\colon A\tensor B \tensor B \to \mathbb I, 
\qquad \kappa(\eta_A\tensor \id_B) = \epsilon_B\colon B\to \mathbb I,
\\
&\kappa(\id_A\tensor \id_A\tensor\Delta_B)=\kappa(m_A\tensor
\id_B)\colon A\tensor A \tensor B \to \mathbb I, 
\qquad \kappa(\id_A\tensor \eta_B) = \epsilon_A\colon A\to \mathbb I.
\end{align*}
If $A$ and $B$ are Hopf algebras, one also requires that 
$\kappa(S_A\tensor \id_B)=\kappa(\id_A\tensor S_B)$. 

\subsection{The free braided Hopf algebra}
	\label{subsect:freehopf}

Now assume that the monoidal category $\C$  satisfies the conditions in \ref{subsect:free} --- that is, admits countable direct sums --- and is \emph{additive}. 
The free algebra
$T(V)$ of an object $V$ of $\C$ has a canonical structure of a Hopf
algebra in $\C$. Namely, let 
$$
i_n\colon V^{\tensorpow n}\to T(V)
$$
denote the canonical injection. Consider 
$$
d_1\colon V \xrightarrow{\sim}\mathbb I \tensor V
\xrightarrow{i_0\tensor i_1} T(V)\tensor T(V),
\quad
d_2 \colon V \xrightarrow{\sim}
V\tensor \mathbb I \xrightarrow{i_1\tensor i_0} T(V)\tensor T(V).
$$
Define $\Delta|_V\colon V \to T(V)\tensor T(V)$ as $\Delta|_V
=d_1+d_2$. Also, define $\Delta|_{\mathbb I}$ to be
$\mathbb I\xrightarrow{\sim} \mathbb I \tensor \mathbb I
\xrightarrow{i_0\tensor i_0}T(V)\tensor T(V)$. 
Following Majid \cite{Mfree}, see also \cite[ch.\ 10]{Mbook},
one checks that
there is a unique family of morphisms 
$$
\Delta_n:=\Delta|_{V^{\tensorpow
    n}}\colon V^{\tensorpow n}\to T(V),
    \qquad n\ge 2, 
$$
such that the
resulting morphism $\Delta\colon T(V)\to T(V)\tensor T(V)$ is
a morphism of algebras. The coassociativity of
$\Delta$ then automatically follows from the coassociativity of $\Delta_1$.

The counit morphism $\epsilon$ on $T(V)$ is defined via
$\epsilon|_{\mathbb I}=\id_{\mathbb I}$ and $\epsilon|_{V^{\tensorpow
    n}}=0$ for $n\ge 1$. The antipode is $S|_{V}=i_1\circ (-\id_V)$ extended
to $T(V)$ as a braided antialgebra map. The details can be found in \cite[ch.\ 
10]{Mbook}; here is an explicit formula for the component $\Delta_n$ of the 
coproduct which is wholly in terms of the braiding $\Psi:=\Psi_{V,V}$ on $V$. 
Majid introduces the \emph{braided integers}
$$
[n]_\Psi = \id_V^{\tensorpow n} + \Psi_{n-1,n} + \Psi_{n-1,n}\Psi_{n-2,n-1} 
+\ldots+\Psi_{n-1,n}\Psi_{n-2,n-1}\cdots \Psi_{1,2} 
\qquad\in\End V^{\tensorpow n},
$$
where the leg notation $\Psi_{i,i+1}$  stands for $\id_V^{\tensorpow 
i-1}\tensor \Psi\tensor \id_V^{\tensorpow n-i-1}$. 
In particular, $[1]_\Psi = \id_V$ and $[2]_\Psi = \id_{V ^{\tensorpow 2}}
+\Psi$. Furthermore, he defines the \emph{braided binomial coefficients}
$ \bbinom{n}{k}_\Psi\in\End V^{\tensorpow n}$
recursively by 
$$
\bbinom{n}{0}_\Psi=\id_{V^{\tensorpow n}}, 
\quad
\bbinom{n-1}{n}_\Psi=0,
\quad
\bbinom{n}{k}_\Psi =  \Psi_{k,k+1}\cdots \Psi_{n-1,n}
(\bbinom{n-1}{k-1}_\Psi\tensor \id_V) + \bbinom{n-1}{k}\tensor\id_V.
$$
It is convenient to use the isomorphism $V^{\tensorpow n}\iso
V^{\tensorpow k}\tensor V^{\tensorpow n-k}$ to assume that 
$$
\bbinom{n}{k}_\Psi\colon V^{\tensorpow n}\to 
V^{\tensorpow k}\tensor V^{\tensorpow n-k}, 
$$ 
which leads to 
$$
\Delta_n =  \sum_{k=0}^n (i_k\tensor 
i_{n-k})\bbinom{n}{k}_\Psi \colon V^{\tensorpow n}\to T(V)\tensor T(V).
$$

\subsection{The duality pairing between $T(V)$ and $T(V^*)$}

If $V$ is a rigid object of $\C$, there is a unique pairing 
$\kappa\colon T(V^*)\tensor T(V)\to \mathbb I$ such that 
$$
\kappa|_{V^*\tensor V}=\ev_V,\qquad\kappa\text{ is a Hopf duality pairing.}
$$
Necessarily the pairing between $V^{\tensorpow n}$ and $V^{*\tensorpow m}$ is 
$0$ unless $n=m$. 
It is shown in \cite{Mfree} that $\kappa$ is given by
$$
\kappa_n := \kappa|_{V^{*\tensorpow n}\tensor V^{\tensorpow n} }= 
\ev_{V^{\tensorpow n}}\circ (\id_{V^{*\tensorpow n}} \tensor 
[n]!_\Psi) =
\ev_{V^{\tensorpow n}}\circ ([n]!_{\Psi^*}\tensor 
\id_{V^{\tensorpow n}})
$$
where $\Psi=\Psi_{V,V}$ and $\Psi^*=\Psi_{V^*,V^*}$, which is necessarily the 
adjoint of $\Psi$ with respect to $\ev_{V\tensor V}$. 
The \emph{braided factorial} is the endomorphism
$$
[n]!_\Psi = ([1]_\Psi \tensor \id_{V^{\tensorpow n-1}})
([2]_\Psi \tensor \id_{V^{\tensorpow n-2}})\cdots 
[n]_\Psi
$$ 
of $V^{\tensorpow n}$. The braided factorial is also known as the 
\emph{braided} (\emph{Woronowicz}) \emph{symmetriser} of degree $n$.
One can see explicitly that $\kappa$ is a Hopf duality pairing by observing 
that the \emph{braided binomial theorem} holds \cite[proof of Proposition 
10.4.13]{Mbook}:
$$
([k]!_\Psi\tensor[n-k]!_\Psi)\bbinom{n}{k}_\Psi =[n]!_\Psi
$$
modulo the isomorphism $V^{\tensorpow k}\tensor V^{\tensorpow n-k} 
\iso V^{\tensorpow n}$.

\begin{remark}[braided symmetriser]
We give the braided symmetriser $[n]!_\Psi$ as a product of braided 
integers, but its expansion (or, rather, expansions, as they depend on
a choice of a reduced word for each element of $S_n$) is also useful. 
For the purposes of this remark only, write $\Psi_{i}$ to denote the 
endomorphism $\Psi_{i,i+1}$ of 
	$V^{\tensorpow
		n}$. It follows from the hexagon axiom
		satisfied by $\Psi$ that the operators
		$\Psi_1,\ldots,\Psi_{n-1}$ satisfy the \emph{braid relations}  
		\[
		\Psi_i \Psi_{i+1} \Psi_i = \Psi_{i+1} \Psi_i \Psi_{i+1},\quad1\le i\le 
n-2;
		\qquad
		\Psi_i \Psi_j = \Psi_j \Psi_i,\quad j>i+1.
		\]
		If $\sigma\in S_n$ is a permutation of $\{1,\ldots,n\}$, write
		$\sigma=s_{i_1}s_{i_2}\ldots s_{i_l}$ in a shortest possible way,
		where $s_i=(i,\, i+1)$. Put  
		$\Psi_\sigma=\Psi_{i_1}\Psi_{i_2}\ldots \Psi_{i_l}$; this does not
		depend on the choice of the reduced (i.e., shortest) word
		$i_1,i_2,\ldots,i_l$ for 
		$\sigma$ because of the braid relations satisfied by the $\Psi_i$. 
		One then has
		\[
		[n]!_\Psi = \sum_{\sigma\in S_n} \Psi_\sigma
		\qquad\in\quad\End V^{\tensorpow n}. 
		\]
\end{remark}

\subsection{The functor $\mathcal B$. Nichols algebras}

\label{subsect:functornichols}

We now come to an important class of Hopf algebras in a braided category $\C$. 
We have seen in \ref{subsect:freehopf} shows that if $\C$ is additive, 
to each $V\in\Ob\ \C$ there is associated 
a free braided Hopf algebra $T(V)$. 
If $V$ is rigid, there is a canonical 
Hopf duality pairing $T(V^*)\tensor T(V)\to \mathbb I$. 

Now assume that $\C$ is an \emph{abelian} braided category.
The free braided Hopf algebra $T(V)$ of a rigid object $V$ is in general not a rigid object, but $V^{\tensorpow n}$ is. The half\dash adjoint of $\kappa_n$ is 
$$
\kappa_n^\flat = [n]!_\Psi \colon  V^{\tensorpow n}\to V^{\tensorpow n},
$$
where $\Psi=\Psi_{V,V}$. Consider the total Woronowicz symmetriser 
$$
\Wor(\Psi)=\bigoplus_{n=1}^\infty  i_n \kappa_n^\flat =
\bigoplus_{n=1}^\infty  i_n [n]!_\Psi \qquad \in\quad \End T(V).
$$
Following Majid's approach in \cite{Mfree,Mbook}, one can quotient 
out the kernel of $\Wor(\Psi)$ to kill the right kernel of the pairing 
$\kappa_n$ for all $n$. Denote
$$
\mathcal B(V,\Psi)=T(V)/\ker \Wor(\Psi).
$$
This is an object in $\C$. 
Moreover, one can check that the product  on a Hopf algebra  $B$
in $\C$ induces a product on the quotient of $B$ 
modulo the (left or right) 
kernel of a Hopf duality 
pairing; same with coproduct. Therefore, $\mathcal B(V,\Psi)$ is a braided 
Hopf algebra.

Moreover, by naturality of the braiding $\Psi$, for each 
morphism $f\colon V \to W$ in $\C$ and for all $n$ one has $f^{\tensorpow 
n}\circ [n]!_{\Psi_{V,V}} = 
[n]!_{\Psi_{W,W}}\circ f^{\tensorpow n}$. This shows that the construction 
of $\mathcal B(V,\Psi)$ is functorial, that is, we are dealing with a functor 
$$
\mathcal B(-,\Psi)\colon \C\to\C\hyp \Alg.
$$
Note that although $\C\hyp \Alg$ is a monoidal category due to the braiding 
on 
$\C$, the functor $\mathcal B$ is not a monoidal functor.

In the setting where $\C$ is a $\Bbbk$\dash linear tensor category 
over some field $\Bbbk$ --- that is, all objects are $\Bbbk$\dash vector 
spaces with some additional structure, and $\tensor$ is $\tensor_\Bbbk$ on 
underlying spaces --- the $\Bbbk$\dash algebra $\mathcal B(V,\Psi)$ 
is known as the \emph{Nichols algebra} of $V$, following Andruskiewitsch and 
Schneider \cite{AS}. 

\begin{remark}
	\label{rem:not abelian}
	To construct $\mathcal B(V,\Psi)$, the category $\C$ was assumed to be 
	abelian, which guarantees the existence of quotient object;
	remember, $\mathcal B(V,\Psi)$ is the quotient of $T(V)$ by the kernel of 
	$\oplus_n [n]!_\Psi$. 
	However, if the category $\C$ is merely additive, the quotient 
	of $T(V)$ by the said kernel may or may not exist in $\C$ 
	for a given braiding $\Psi$.
Sometimes we would like to work in an
		additive monoidal category which is not abelian: think of 
		\emph{free} $R$\dash modules with respect to $\tensor_R$
		where $R$ is a commutative ring. 
		The existence of $\mathcal B(V,\Psi)$ in such a category 
		is not guaranteed \emph{a priori} but will be proved by other means in 
		special cases. 
\end{remark}

\subsection{The cocycle twist of a braiding and of $\mathcal B(V,\Psi)$}

Recall that a cocycle $\mu\in Z^2(\C)$ is a natural automorphism of the 
monoidal 
product $\tensor$ on $\C$, and a braiding $\Psi$ is a natural transformation
$\tensor\to\tensor^{\mathrm{op}}$. 
In an obvious way, $\mu^{\mathrm {op}}$  given by $\mu^{\mathrm {op}}_{X,Y}
	=\mu_{Y,X}$ is a natural automorphism of $\tensor^{\mathrm{op}}$.
The following is then the way to twist $\Psi$ by $\mu$.
It can be checked directly that the result is again a braiding (or 
see \cite[Proposition 2.7]{PSV}):
\begin{lemma}
	\label{lem:psi_mu}
	$\Psi_\mu = \mu^{\mathrm {op}}\circ \Psi\circ \mu^{-1}$ is a braiding 
	on $\C$. 
	\qed
\end{lemma}
Thus, the new braiding $\Psi_\mu$ is given by the formula 
$$
(\Psi_\mu)_{X,Y}=\mu_{Y,X}\Psi_{X,Y}\mu_{X,Y}^{-1}
\qquad\text{for }X,Y\in\Ob\ \C.
$$
Let us find out the relationship between 
the algebras $\mathcal B(V,\Psi)$ and $\mathcal B(V,\Psi_\mu)$.

\begin{theorem}[twisting of $\mathcal B(V,\Psi)$]
	\label{thm:twisting}
	Let $V$ be an object of an abelian monoidal category $\C$ which admits 
	free algebras. Let $\Psi$ be a braiding on $\C$ and $\mu\in Z^2(\C)$. Then
	$$
	\mathcal B(V,\Psi_\mu)\cong \mathcal B(V,\Psi)_\mu
	$$
	 as algebras
	in $\C$, where $ \mathcal B(V,\Psi)_\mu$ is the twist of 
	$\mathcal B(V,\Psi)$ by the cocycle $\mu$. 
\end{theorem} 
\begin{proof} 
	Denote by $\mu_n$ the morphism $\mu_{V,V,\ldots,V}\colon 
	V^{\tensorpow n}\to V^{\tensorpow n}$ introduced in 
Remark~\ref{rem:cocycle}. For consistency, write $\mu_0=\id_{\mathbb I}$ and 
$\mu_1=\id_V$. 
	Observe the following propety of $\mu_n$, 
	$n\ge 3$, obtained by iterating the cocycle condition:
	$\mu_{V^{\tensorpow k},V^{\tensorpow n-k}}(\mu_k\tensor \mu_{n-k})
	=\mu_n i_{k}^n$, where $i_{k}^n$ denotes 
	the canonical isomorphism between 
	$V^{\tensorpow k}\tensor V^{\tensorpow n-k}$ and $V^{\tensorpow n}$. 
	
	It follows that the $\C$\dash morphism 
	$\oplus_{n=0}^\infty i_n \mu_n \colon T(V)\to T(V)$
	is in fact an isomorphism between the algebras $T(V)$ and $T(V)_\mu$. 
	
	Indeed, the product in $T(V)$ restricted to
	 $V^{\tensorpow k}\tensor V^{\tensorpow n-k}$
	 coincides with $i_k^n$. 
	The product on $T(V)_\mu$, which has the same underlying 
	object $\oplus_{n=0}^\infty V^{\tensorpow n}$ as $T(V)$, 
	is defined on $V^{\tensorpow k}\tensor V^{\tensorpow n-k}$ as 
	$i_k^n \mu_{V^{\tensorpow k},V^{\tensorpow n-k}}$. 
	But then we have
	$$
	i_k^n \mu_{V^{\tensorpow k},V^{\tensorpow n-k}}
	(\mu_k\tensor \mu_{n-k}) = \mu_n i_k^n,
	$$ which 
	precisely says that $\oplus_{n=0}^\infty i_n \mu_n$
        intertwines 
the two 
products on the object $T(V)$, 
	the $T(V)$\dash product and the $T(V)_\mu$\dash product. 
	
	To get the desired isomorphism between 
	$\mathcal B(V,\Psi)$ and $\mathcal B(V,\Psi_\mu)$, we have to pass to the 
	quotient algebras. 
	It follows from the definition of $\Psi_\mu$ and of $[n]!_\Psi$ that 
	$$
		[n]!_{\Psi_\mu} = 
		\mu_n \, [n]!_\Psi\,  \mu^{-1}_n,
	$$ 
which means that the isomorphism $\oplus_{n=0}^\infty i_n \mu_n$ between 
$T(V)$ and $T(V)_\mu$ induces an isomorphism 
between $\mathcal B(V,\Psi_\mu)=T(V)/\oplus_n \ker [n]!_{\Psi_\mu}$ and 
$(T(V)/\oplus_n \ker [n]!_\Psi)_\mu=\mathcal B(V,\Psi)_\mu$.
\end{proof}

\subsection{Twists of a braiding by monoidal automorphisms of the category}
\label{subsect:autom}

Let us point out that there is another way to twist a braiding $\Psi$ on a 
given monoidal category $\C$. 
Suppose that $F\colon \C \to \C$ is a monoidal functor which is strictly 
invertible; that is, an automorphism of $\C$ as a monoidal category. 
Define a new braiding $\Psi^F$ by
$$
\Psi^F_{X,Y}=F^{-1}(\Psi_{F(X),F(Y)}),\qquad X,Y\in\Ob\ \C.
$$
It is straightforward to check that $\Psi^F$ is again a braiding on $\C$. 

In one of the main examples below, a cocycle twist of a braiding coincides 
with a twist by an automorphism:
$$
  \Psi_\mu = \Psi^F. 
$$
In this case, there is more we can say about the algebra 
$\mathcal B(V,\Psi)_\mu$. By Theorem~\ref{thm:twisting},
$\mathcal B(V,\Psi)_\mu\cong \mathcal B(V,\Psi_\mu)$. 
But clearly $F(\mathcal B(V,\Psi^F))=\mathcal B(F(V),\Psi)$; here $F$
on the left\dash hand side 
is viewed as an automorphism of the category $\C\hyp \Alg$. 
We conclude that 
$$
F(\mathcal B(V,\Psi)_\mu) = \mathcal B(F(V), \Psi).
$$
This observation is used in the next section.


\section{Extensions and twists of Yetter-Drinfeld modules}
\label{sect:YD}

We have described a general approach to algebras $\mathcal B(V,\Psi)$ 
and their cocycle twists that works in arbitrary 
braided monoidal categories. On the other hand, 
the mainstream research into Nichols algebras focuses 
on the braided category $\YD H$ of Yetter\dash Drinfeld modules 
over a $\Bbbk$\dash linear Hopf algebra $H$, denoted $\YD G$ when $H=\Bbbk G$ 
is a group algebra. (There are four versions of the Yetter\dash Drinfeld 
module category:
$\mathcal{YD}_H^H$, $_H\mathcal{YD}^H$, $^H\mathcal{YD}_H$, $_H^H\mathcal{YD}$,
but the differences between them are immaterial.) 

We would like to get rid of the $\Bbbk$\dash linearity assumption. 
Let $R$
be a commutative ring, and 
consider Yetter\dash Drinfeld modules 
which are free $R$\dash modules --- a mild generalisation of 
vector spaces --- 
with an action and a coaction of a finite group $G$. 
This category is the appropriate target category 
for the cocycle extension functor for Yetter\dash Drinfeld modules, 
associated to $\mu\in Z^2(G,R^\times)$.
This functor generalises the cocycle twist construction for Yetter\dash 
Drinfeld modules, originally due to Majid and Oeckl \cite{MO}, 
now widely used in the literature on Nichols algebras --- 
see for example \cite{AFGV}.






\subsection{The monoidal category $\protect\YDE{R}{G}$}
\label{def:YD}
Let $R$ be a commutative ring and $G$ be a finite group. 
The category $\YDE R G$ 
of $R$\dash free Yetter\dash Drinfeld modules for $RG$ 
consists of 
$G$\dash graded $RG$\dash modules 
$$
Y=\bigoplus_{g\in G} Y_g,
$$
where each $V_g$ is a \emph{free} $R$\dash module. The action 
$\act$ of $G$ and the grading are compatible in the sense
$$
g\act Y_h=Y_{ghg^{-1}}
$$
for all $g,h\in G$.

Recall that a $G$\dash grading on $Y$ can be written as
a coaction $\delta\colon Y \to Y \tensor_R R G$, see
Definition~\ref{def:coaction}. 
For coactions we will use the Sweedler sigma notation \cite{Sw}, but
without the sigma. Such notation is often found in modern Hopf algebra
literature, see for example \cite{Mprimer}. We write 
\[
  \delta(v) = v^{(0)}\tensor_R v^{(1)} \quad \in Y\tensor R G,
\]
where the summation is implicit. The Yetter\dash Drinfeld module
condition for $Y$ is then written as  
\begin{align*}
   \delta(h\act v) & = (h\act v^{(0)})\tensor_R hv^{(1)}h^{-1} 
                \\ & = h\act \delta(v)
   \qquad \forall h\in G,\forall v\in Y.
\end{align*}
The last equality assumes that $Y\tensor_R R G$ is viewed as the tensor 
product of $RG$\dash modules $Y$ and $(R G)_\ad$.  

The tensor product $Y\tensor_R Z$ of two Yetter\dash Drinfeld modules
$Y$, $Z$ is again a Yetter\dash Drinfeld module,  with coaction 
 given by $\delta(y\tensor_R z)=(y^{(0)}\tensor_R z^{(0)})\tensor_R
 y^{(1)}z^{(1)}$. Note that the tensor product, over $R$, of two free $R$\dash 
 modules is again a free $R$\dash module. 
 The category $\YDE R G$ is thus a monoidal category. 

We do not compute the lazy cohomology 
$H^2_\ell(\YDE R G)$ but observe that each cocycle $\mu\in Z^2(G,R^\times)$ 
gives rise to a categorical cocycle on $\YDE R G$.

\begin{lemma}
	\label{lem:lazy}
	A group cocycle 
	$\mu\in Z^2(G,R^\times)$ 
	gives rise to a categorical cocycle on $\YDE{R}{G} $ defined as 
follows:
	$\mu_{X,Y}\in \End(X\tensor_R Y)$ acts on $X_g\tensor_R Y_h$
	by $\mu(g,h)\in R^\times$, where $g,h\in G$. 
\end{lemma}
\begin{proof}
	Clearly $\mu_{X,Y}$ is natural in both $X$ and $Y$ and is an automorphism
	of $X\tensor_R Y$ 
	(commutes with the $G$\dash action, preserves the $G$\dash grading). 
The categorical cocycle condition for $\mu_{X,Y}$ 
is the same as the cocycle condition for $\mu(g,h)$. 
\end{proof}

\subsection{$\protect\YDE{R}{G}$ as a braided category}
\label{subsect:yd-braided}

 The category $\YDE R G$ is braided: 
 the standard braiding is given by 
 $$
 \Psi_{Y,Z}(y\tensor_R z)=(g\act z)\tensor_R y
 \qquad\text{for }
 y\in Y_g,z\in Z.
 $$
Yetter\dash Drinfeld modules which are free $R$\dash modules of finite rank 
are rigid objects in $\YDE R G$.   The dual object to $Y$ is $Y^*=\Hom(Y,R)$ 
as an 
$RG$\dash module, with grading $(Y^*)_g=(Y_{g^{-1}})^*$.  

\begin{example}[Modules with trivial $G$-grading]

Any $RG$\dash module $V$ is a Yetter\dash Drinfeld module with respect
to the trivial grading, $\delta(v)=v\tensor_R 1$ for $v\in V$. This
provides an embedding of the category of $R$\dash free $RG$\dash modules as
a full subcategory in $\YDE R G$. 

\end{example}

\begin{example}[The adjoint module $(R G)_\ad$]
\label{ex:adjoint}

Let $(R G)_\ad$ denote the free $R$\dash module $R G$ where the group $G$
acts by conjugation. We refer to $(R G)_\ad$ as the
\emph{adjoint module} for $G$. Observe that 
$(R G)_\ad$ is a Yetter\dash Drinfeld module for
$G$ with respect to the grading given by $(R G)_g=R g$ for $g\in
G$. This grading yields the coaction $\delta(g)=g\tensor_R g$.

\end{example}

\begin{example}[The Yetter-Drinfeld module $V\tensor_R  R G$]

Take the tensor product of the above two examples: to each $G$\dash
module $V$ there is associated a Yetter\dash Drinfeld  module
$V\tensor_R R G$ with $G$\dash action $h\act(v\tensor_R g)=(h\act
v)\tensor_R hgh^{-1}$ and coaction $\delta(v\tensor_R g)= 
v\tensor_R g\tensor_R g$. Here $v\in V$ and  $g,h\in G$. 

Various modifications of the construction of $V\tensor_R R G$ provide a
supply of useful Yetter\dash Drinfeld modules. For instance, if
$C\subset G$ is a conjugacy class, $V\tensor_R R C$ is clearly a
Yetter\dash Drinfeld submodule of $V\tensor_R R G$ in $\YDE R G$.
When $R=\Bbbk$ is a field, this is
an important example of a rack with a cocycle as in \cite{AFGV}.  
\end{example}

\subsection{Nichols algebras of Yetter-Drinfeld modules}

As we mentioned earlier, if $R=\Bbbk$ is a field, the category 
$\YDE R G$ is one of the most typical examples of braided categories
used to construct Nichols algebras $\mathcal B(Y,\Psi)$. 
If $Y$ is a Yetter\dash Drinfeld module, we will often write 
$\mathcal B(Y)$ to denote $\mathcal B(Y,\Psi)$ when $\Psi$ is 
the standard braiding given in~\ref{subsect:yd-braided}. 
We work with 
$\Bbbk=\CC$ and denote the category $\YDE\CC G$ simply by $\YD G$.
Recall from Section~\ref{sect:nichols} that the Nichols algebra 
of $Y\in \Ob\ \YD G$ is 
$$
\mathcal B(Y) = T(Y)/\oplus_{n} \ker [n]!_\Psi, \qquad \Psi=\Psi_{Y,Y}.
$$
For general $R$, however, it is not clear why the algebra $\mathcal B(Y)$ 
exists in the category $\YDE R G$, because this category may not be abelian.
It exists if the quotient object exists in the category; the algebra
structure is then automatic. We leave the following question open:

\begin{question}
	\label{q:free}
	Let $Y$ be a Yetter\dash Drinfeld module, over a commutative ring $R$,
	 for a group $G$. Assume that $Y$ is free as $R$\dash module. 
	 Does the algebra
$\mathcal B(Y)$ exist in the category of free $R$\dash modules?  
\end{question}

We obtain some positive evidence related to Question~\ref{q:free}, showing 
that if $R$ is a $\CC$\dash algebra, 
then cocycle extensions of $\mathcal B(Y)$ 
where $Y\in\Ob\ \YD G$ are Nichols algebras of certain extensions of $Y$ in 
the category $\YDE R G$. Note that a cocycle extension of a $\CC$\dash algebra 
is automatically free as $R$\dash module.

The rest of this section 
is devoted to extending a Nichols algebra of $Y\in\Ob\ \YD G$ by a cocycle 
$\mu\in Z^2(G,R^\times)$, where $R$ is a commutative $\CC$\dash algebra. 
The result turns out to be an algebra
$\mathcal B(\widetilde Y_\mu)$ of a certain object $\widetilde{Y}_\mu$ 
of $\YDE{R}{G}$. According to the scheme outlined in 
\ref{subsect:extension}, \ref{thm:twisting} and \ref{subsect:autom}, 
the extension is carried out in three steps.

\subsection{Step 1: extension by the trivial cocycle (extension of scalars)}

Let $Y\in\Ob\ \YD G$. As any $\CC$\dash algebra, 
$\mathcal B(Y)$ has trivial extension $R\tensor_\CC\mathcal B(Y)$
which is an algebra over $R$. 
The next Lemma shows that this $R$\dash algebra is $\mathcal B(R\tensor _\CC 
Y)$.
That is,
the extension of scalars functor commutes with the
functor $\mathcal B$.

\begin{lemma}[extension of scalars for Yetter-Drinfeld modules]
\label{lem:scalars}
	Let $R$ be a $\CC$\dash algebra and $Y\in\Ob\ \YD G$.
Let the $R$\dash module $\widetilde Y=R\tensor_\CC Y$  be 
the extension of scalars of $Y$, with the action 
and grading extended from $Y$ by 
$$
g\act (r\tensor_R y) = r\tensor_R (g\act y),
\quad r\in R;
\qquad
\widetilde Y=\bigoplus_{g\in G}
R\tensor_\CC Y_g.
$$
Then $\widetilde Y \in\Ob\  \YDE R G$.
Moreover, 
$R\tensor_\CC \mathcal B(Y)=\mathcal B(\widetilde Y)$.
\end{lemma}
\begin{proof}
By construction, $\widetilde Y=R\tensor_\CC Y$ is a free $R$\dash module.
	It is easy to see that the Yetter\dash Drinfeld condition for the action 
and grading on $R\tensor_\CC Y$ follows from the Yetter\dash Drinfeld 
condition on $Y$. 

Denote by $T_R\colon \YDE R G \to \YDE{R}{G}\hyp\Alg$ the free algebra functor 
in the category 
$\YDE R G$. 
It is clear, e.g., by considering the basis of the free $R$\dash module 
$(\widetilde Y)^{\tensorpow_R^n}$, that 
$T_R(\widetilde Y)=R\tensor _\CC T(Y)$. 
Let 
$\Psi=\Psi_{Y,Y}\in \End_\CC (Y\tensor Y)$, 
$\widetilde \Psi=\Psi_{\widetilde Y,\widetilde Y}\in \End_R (\widetilde 
Y\tensor_R \widetilde Y)$ be the braidings on $Y$ and on $\widetilde Y$. 
The definition of a braiding on a Yetter\dash Drinfeld module 
implies that 
$\widetilde \Psi\in\End_R (\widetilde Y\tensor_R 
\widetilde Y) =\End_R (R\tensor_\CC Y \tensor_\CC Y)$ is the same as 
$\id_R\tensor_\CC \Psi$. It follows that 
the braided factorials $[n]!_{\widetilde\Psi}$ are also expressed 
as $\id_R\tensor_\CC [n]!_{\Psi}$. 
Because $R \tensor_\CC (-)$ is an exact functor 
from $\CC$\dash vector spaces to $R$\dash modules,
$\ker [n]!_{\widetilde \Psi}=R\tensor_\CC \ker [n]!_{\Psi}$ 
and 
$$
\mathcal B(\widetilde Y)  = (R\tensor_\CC T(Y))/(R\tensor_\CC
\ker [n]!_{\Psi})
=R\tensor_\CC(T(Y)/
\ker [n]!_{\Psi})=R\tensor_\CC\mathcal B(Y).
\qedhere
$$
\end{proof}
 
\subsection{Step 2: twisting $\mathcal B(\widetilde Y)$ by $\mu\in 
Z^2(G,R^\times)$}
 

Let now $Z\in \Ob\ \YDE{R}{G}$ and $\mu\in Z^2(G,R^\times)$. 
We know how to twist $\mathcal B(Z)$ 
by a categorical cocycle, and Lemma~\ref{lem:lazy} explains
how to interpret $\mu$ as a categorical cocycle.
Applying Theorem~\ref{thm:twisting} and the formula for $\Psi_\mu$ from 
Lemma~\ref{lem:psi_mu}, we obtain
\begin{corollary}
	\label{cor:psi_mu}
	The twist of $\mathcal B(Z)$ by $\mu$ is the algebra 
	$\mathcal B(Z,\Psi_\mu)$ in $\YDE{R}{G}$. 
	The braiding $\Psi_\mu$ on the category $\YDE{R}{G}$ 
	is given by the formula 
	$$
	(\Psi_\mu)_{X,Y} (x \tensor_R y) = \mu(ghg^{-1},g)\mu(g,h)^{-1}
(g\act y)\tensor_R x, \qquad x\in X_g,\ y\in Y_h,\ g,h\in G.
$$
\end{corollary}

\subsection{Step 3: realisation of $\Psi_\mu$ via an automorphism 
of the Yetter\dash Drinfeld category}

Finally, we observe that $\Psi_\mu$, given in Corollary~\ref{cor:psi_mu},
can be realised as $\Psi^F$ where $F\colon \YDE{R}{G}\to \YDE{R}{G}$ is a 
strictly invertible functor. 
Indeed, the formula for $(\Psi_\mu)_{X,Y}$ 
shows that this is a standard braiding between 
$F(X)$ and $F(Y)$ where the functor $F=F_\mu$ is defined in 


\begin{lemma}
\label{lem:Fmu}
Let $\mu\in Z^2(G,R^\times)$.
There exists a functor $F_\mu\colon \YDE{R}{G} \to \YDE{R}{G}$
such that 

--- $F_\mu(X)$ is $X$ as an $R$\dash module and 
$G$\dash comodule;

--- if $\act\colon G\times X \to X$ is the $G$\dash action, then the $G$\dash 
action on $F(X)$ is given by 
$ g\actchi x= \chi(g,k) (g\act x)$, $g,k\in G$, $x\in X_k$;

--- $F_\mu$ is identity on morphisms.	
Here 
$$
\chi(g,k)=\mu(gkg^{-1},g)\mu(g,k)^{-1}\in R^\times \qquad\text{for }
g,k\in G.
$$ 
The functor $F_\mu$ is strictly invertible with inverse $F_{\mu^{-1}}$. 
\end{lemma}

\begin{remark}
The functor $F_\mu$ arises from the Majid\dash Oeckl construction
\cite[Theorem 2.7]{MO}, in a special case when a Hopf algebra $H$ is the group
algebra. Strictly 
speaking, in \cite{MO}
the construction is over a field, while our setup allows us to carry it
over in $\YDE{R}{G}$. 
The formula for $\chi$ in terms of $\mu$ is an adaptation of 
the Hopf algebra formula from \cite{MO} to the case of a group algebra
 of $G$. Note that we have already seen the function 
 $\chi$ in \ref{subsect:chi} in the case $G=S_n$.
  
Moreover, in the linear case ($R=\Bbbk$ is a field), 
the extension construction for Nichols algebras 
becomes a twist functor 
$(\,\cdot \,)_\mu\colon \YD G\hyp\Alg \to \YD G\hyp \Alg$. 
The construction can be carried out for a 
Nichols algebra associated to a rack with cocycle. 
In that setting, $\mu$ is a cocycle that arises from the cohomology of a rack
--- this is similar to the cohomology of a finite group $G$
which in degree $2$ yields the Schur multiplier $M(G)$.
See for example  \cite{AFGV}. 
The formula for $\chi$ in the form given in Lemma~\ref{lem:Fmu} is the same as 
\cite[equation (20)]{AFGV} and apparently goes back to \cite{EGr}.

\end{remark}
\begin{proof}[Proof of Lemma~\ref{lem:Fmu}]
It is enough
	to show that $F_\mu$ is well defined, 
all axioms of a functor being fulfilled automatically.	
	We need to check that 
	the formula for $\actchi$ defines an action of $G$. 
%
The condition to be checked is 
$\chi(g,hkh^{-1})\chi(h,k)=\chi(gh,k)$.
This is the same as 
$\frac{\chi(gh,k)}{\chi(g,h\adact k)}=\chi(h,k)$ for all 
$g,h,k\in G$. 
	(We write $\frac{z}{t}$ to denote $zt^{-1}$ where $z,t\in R^\times$.) 
	Rewrite the left\dash hand side in terms of $\mu$ and use the cocycle 
equation for $\mu$:
	$$
	\frac{\mu(gh,k)}{\mu((gh)\adact k, gh)} \cdot \frac{\mu((gh)\adact k, 
g)}{\mu(g,h\adact k)}
	=
	\frac{\mu(gh,k) \mu(g,h)}{\mu(ghkh^{-1},h)\mu(g,hkh^{-1})}
	= 
	\frac{\mu(g,hk)\mu(h,k)}{\mu(g,hkh^{-1}h)\mu(hkh^{-1},h)} 
	$$
	which is $\chi(h,k)$ as required. 
It remains to observe that morphisms $X\to Y$ compatible 
with the $G$\dash action 
$\act$ and the $G$\dash grading are also compatible 
with $\actchi$. 
\end{proof}
\begin{remark}
	We repeat an observation made in \ref{subsect:chi}: 
	the condition on $\chi$, checked in the proof of the Lemma, 
	can be interpreted to read that 
	$\chi$ is a $1$\dash cocycle on $G$ with coefficients in 
	$\Fun(G,R^\times)$ viewed as a right module for $G$ with non\dash 
	trivial action.
\end{remark}

The above implies the following 

\begin{corollary}
\label{cor:Ymu}
Let $Y\in \Ob\ \YD G$, $R$ be a commutative $\CC$\dash algebra 
and $\mu\in Z^2(G,R^\times)$. Let $\chi=\chi_\mu$ be  as in Lemma~\ref{lem:Fmu}. 
Define 
$\widetilde Y_\mu\in \Ob\ \YDE R G$ to be $R\tensor_\CC Y$ as 
$R$\dash module with $G$\dash action 
$\actchi$ and  $G$\dash grading naturally extended from $Y$. Then 
$$
\widetilde{\mathcal B(Y)}_\mu = \mathcal B(\widetilde Y_\mu)
$$ 
as objects in the category $\YDE R G\hyp\Alg$.
\qed
\end{corollary}

\subsection{Fomin-Kirillov algebras, Majid algebras, and cocycles}

\label{subsect:fkalgebras}

We finish this section by considering an explicit example of 
a cocycle extension 
of a Nichols algebra. Let $S_n$ be the symmetric group of degree $n$. 
Denote by $X_n$ the conjugacy class of all transpositions in $S_n$. 

Suppose that $q\colon S_n\times X_n\to \CC^\times$
is a function satisfying
$$
q(\rho\sigma,\tau)=q(\rho,\sigma\tau\sigma^{-1})q(\sigma,\tau),
\qquad\forall \rho,\sigma\in S_n,\ \forall \tau\in X_n.
$$
This means that $\sigma \mapsto q(\sigma,-)$ is a normalised $1$\dash 
cocycle in $Z^1(S_n,\Fun(X_n,\CC^\times))$.  
Clearly, $q$ is determined by its 
values $q(\sigma,\tau)$ where $\sigma,\tau\in X_n$, and these values
must satisfy the \emph{rack cocycle} condition \cite{AFGV}.
Then one can define a Yetter\dash Drinfeld module structure 
on the vector space 
$$
\CC X_n=\mathrm{span}\{e_{(i\,j)}\mid (i\,j)\in X_n\}
$$
as follows. The basis vector
$e_\tau\in\CC X_n$ has $S_n$\dash degree $\tau$; this gives the $S_n$\dash grading 
on $\CC X_n$. 
The action of $S_n$ on $\CC X_n$ is via
$$
\sigma \act\nolimits_q e_\tau = q(\sigma,\tau) e_{\sigma\tau\sigma^{-1}}. 
$$
It is easy to see that the Yetter\dash Drinfeld condition holds; 
denote the resulting 
Yetter\dash Drinfeld module for $\CC S_n$ by $(X_n,q)$. One can consider
the Nichols algebra $\mathcal B(X_n,q)$ which is an algebra
in $\YD{S_n}$. 

Two particular examples of $\mathcal B(X_n,q)$ arose in the literature. 
First, consider the function defined for $\sigma\in X_n$ and $i<j$ by
$$
q_1(\sigma,(i\,j))=\begin{cases}1, &\text{if } \sigma(i)<\sigma(j),
\\ -1, &\text{if }\sigma(i)>\sigma(j).
\end{cases}
$$
The \emph{Fomin\dash Kirillov algebra}
$\mathcal E_n$, introduced in \cite{FominKirillov}
as a noncommutative model for the Schubert calculus,
can be defined by taking the generators and the \emph{quadratic}
relations of $\mathcal B(X_n,q_1)$.
%
 Conjecturally, 
$\mathcal B(X_n,q_1)$ is a quadratic algebra and hence coincides with 
$\mathcal E_n$. The same Nichols algebra $\mathcal B(X_n,q_1)$ 
was defined by Milinski and Schneider in \cite{MilinskiSchneider}. 
It was shown by Fomin and Kirillov \cite[section 5]{FominKirillov}
that the elements 
$$
\theta_j = - \sum_{1\le i<j}e_{(i\,j)} + \sum_{j<i\le n}e_{(j\,i)}, \qquad
j=1,2,\ldots,n,
$$
commute pairwise and generate a subalgebra of $\mathcal E_n$ isomorphic to
the cohomology ring of the complex flag manifold $\mathit{Fl}_n$. 
In \cite{B}, the same was proved for the Nichols algebra 
$\mathcal B(X_n,q_1)$. See \cite{VendraminAbstract} for 
a recent brief survey on the Fomin\dash Kirillov algebras. 

The second example we would like to mention is
$$
q_{-1}(\sigma,(i\,j)) =-1, 
\qquad \forall\sigma,\tau\in X_n,
$$
which, too, yields an element of $Z^1(S_n,\Fun(X_n,\CC^\times))$ and 
leads to a Nichols algebra 
$\mathcal B(X_n,q_{-1})$.
This Nichols algebra and its quadratic cover 
$\Lambda_n$ were considered by Majid in \cite{MajidQuadratic}. 
Conjecturally, $\Lambda_n=\mathcal B(X_n,q_{-1})$.
Majid showed that the elements 
$$
\alpha_j = -\sum_{1\le i<j}e_{(i\,j)} - \sum_{j<i\le n}e_{(j\,i)}
=-\sum_{i\ne j}e_{(i\,j)},
\qquad j=1,2,\ldots,n,
$$ 
pairwise anticommute in $\Lambda_n$ and generate a subalgebra
of $\Lambda_n$ 
termed ``the subalgebra of flat connections''. 

In fact, 
$q_1$ and $q_{-1}$ are the only elements of $Z^1(S_n,\Fun(X_n,\CC^\times))$
for which the corresponding Nichols algebra generated by $X_n$ has a chance 
to be finite\dash dimensional, see \cite{V}.
It is known that $\mathcal B(X_n,q_1)$ and $\mathcal B(X_n,q_{-1})$
are finite\dash dimensional for $n\le 5$. 

It was shown by Vendramin in \cite{V} that the Nichols algebras 
$\mathcal B(X_n,q_1)$ and $\mathcal B(X_n,q_{-1})$
in $\YD{S_n}$ are cocycle twists of each other. We can describe the 
relevant cocycle $\mu\in Z^2(S_n,\CC^\times)$ as
$$
      \mu=[1,z]|_{z=-1}=[1,-1],
$$
where $[1,z]\in Z^2(S_n,\cycl_2)$ is as defined in 
\ref{subsect:cover_symm}. 

\subsection{The Nichols $\CC \cycl_2$-algebra over $S_n$}

\label{subsect:nichols-z}

To finish this Section, 
we point out that both Nichols algebras 
$\mathcal B(X_n,q_{\pm1})$  are specialisations of an 
algebra over $R:=\CC\cycl_2=\CC[z]/\lgen z^2-1\rgen$, obtained 
by the cocycle extension construction. 
There is a cocycle $q_z\in Z^1(S_n,\Fun(X_n,R^\times))$ such that 
 for $\sigma\in X_n$ and $i<j$,
$$
q_z(\sigma,(i\, j))=\begin{cases}
z &\text{if }\sigma(i)<\sigma(j),
\\
-1, &\text{if }\sigma(i)>\sigma(j).
\end{cases}
$$
Consider $R X_n$ as an $RS_n$\dash module with the 
action
$$
\sigma\act e_{(i\,j)} = q_z(\sigma,(i\,j)) e_{(\sigma(i)\, \sigma(j))}.
$$
This, together with the standard $S_n$\dash grading, makes $RX_n$ an object 
of $\YDE{R}{S_n}$ with the braiding given by
$$
\Psi_z(e_\tau\tensor_R e_\upsilon)=
q_z(\tau,\upsilon)e_{\tau\upsilon\tau^{-1}}\tensor_R e_\tau,
\qquad 
\tau,\upsilon\in X_n,
$$
and leads to the algebra 
$$
\mathcal B(RX_n,q_z)
=T(RX_n)/\ker \Wor(\Psi_z)
\qquad\in\Ob(\YDE{R}{S_n}\hyp\Alg). 
$$
The theory of cocycle 
extensions developed so far leads to the following

\begin{theorem}
\label{thm:supernichols}
 Let $R=\CC \cycl_2$ as above. 

1. The algebra $\mathcal B(RX_n,q_z)$ is the cocycle extension 
of $\mathcal B(X_n,q_1)$ by the cocycle $\mu_{(1,z)}\in Z^2(S_n,\cycl_2)$,
and is also the cocycle extension 
of $\mathcal B(X_n,q_{-1})$ by the cocycle 
$\mu_{(1,z)}(\mu_{(1,z)}|_{z=-1})\in Z^2(S_n,R^\times)$.

2. $\mathcal B(RX_n,q_z)$ is a $\cycl_2$\dash flat deformation 
of both $\mathcal B(X_n,q_1)$ and $\mathcal B(X_n,q_{-1})$, 
in particulaaaaaaaaaar, a free $R$\dash module. One has
$$
\mathcal B(RX_n,q_z)|_{z=1}=\mathcal B(X_n,q_1),
\qquad
\mathcal B(RX_n,q_z)|_{z=-1}=\mathcal B(X_n,q_{-1}),
$$
hence 
$$
\mathcal B(RX_n,q_z) \cong
\frac{1+z}{2}\mathcal B(X_n,q_1) \dirsum
\frac{1-z}{2}\mathcal B(X_n,q_{-1}), 
$$
as direct sum of two\dash sided ideals. 
\qed
\end{theorem}

Note that
the algebra $\mathcal B(RX_n,q_z)$ is generated, as a $\CC$\dash algebra, 
by $z$, $e_{(i\,j)}$, $1\le i<j\le n$. The quadratic relations in 
$\mathcal B(RX_n,q_z)$ extend the relations in both $\mathcal E_n$ and 
$\Lambda_n$: they are
\begin{gather*}
z\text{ is central}, \qquad z^2=1,
\\
e_{(i\,j)}^2=0, \qquad
e_{(i\,j)}e_{(s\,t)}=ze_{(s\,t)}e_{(i\,j)} 
\quad\text{if } \{i,j\}\cap \{s,t\}=\varnothing,
\\
e_{(i\, j)} e_{(j\,k)} 
-z e_{(j\,k)}e_{(i\,k)}-ze_{(i\,k)}e_{(i\, j)}=0,
\quad\text{if }i<j<k.
\end{gather*}


\section{Braided doubles}
\label{sect:bdoubles}

We will now review the braided doubles, introduced 
in \cite{BBadv} and \cite{BBselecta}.
Although much of the setup carries over to a general 
abelian monoidal category, 
we work with braided doubles over $\CC$. 
This section contains general facts and constructions which will be needed later. 
In this Section, unadorned tensor products $\tensor$ always mean $\tensor_\CC$. 

\subsection{Deformations of semidirect products}
\label{subsect:defproblem}

Let $\act$ be a left action of a group $G$ on an algebra $A$. The semidirect product of $A$ and $G$, denoted $A\lcprod G$, is an algebra with underlying vector space $A\tensor \CC G$. The embeddings of $A$ and $\CC G$ in $A\ltimes G$ via the maps $a\mapsto a\tensor 1$ and $g\mapsto 1\tensor g$ are algebra morphisms, and the relation 
\[
ga=(g\act a)g
\]
holds in $A\lcprod G$ for all $g\in G$, $a\in A$. These conditions determine the multiplication in $A\lcprod G$ uniquely.

In particular, if $V$ is a $G$\dash module, the group $G$ acts on the free tensor algebra $T(V)=\bigoplus_{n=0}^\infty V^{\otimes n}$ by algebra automorphisms. Let 
\[
     R_0\ \subset\  T^{>0}(V)=\bigoplus_{n>0} V^{\otimes n} 
\]
be a graded $G$\dash submodule of $T(V)$. The two\dash sided ideal $I_0$ of $T(V)$ generated by $R_0$ is $G$\dash invariant, and there is an algebra 
\[
A_0 = (T(V)/I_0)\lcprod G = \CC G \ \dirsum \ (V/V\cap R_0)\lcprod G \ \dirsum \ \ldots,
\]
with a grading where the degree of $G$ is $0$ and the degree of $V$ is $1$.
We say that a subspace
\[
R \ \subspace \ R_0\dirsum\CC G \ \subspace\  T(V)\lcprod G
\]
is a \emph{deformation} of $R_0$ if the projection $R_0\dirsum\CC G\sur R_0$ restricts to a one-to-one linear map $R\iso R_0$. 
In this situation, let $I$ be the ideal of $T(V)\lcprod G$ generated by $R$. The algebra 
\[
A = (T(V)\lcprod G)/I
\]
is viewed as a deformation of $A_0$. Observe that putting $G$ in degree zero and $V$ in degree one defines an ascending filtration on $A$.
In general, there is a surjective map 
\[
A_0\sur \gr A,
\] 
where $\gr A$ denotes the associated graded algebra of $A$ with
respect to this filtration. The algebra $A$ is said to be a \emph{flat
  deformation} of $A_0$, if this map is an isomorphism. 

One is interested in the following problem: find deformations
$R\subset R_0\dirsum\CC G$ of a given $R_0\subset T^{>0}(V)$ such that
$A$ is a flat deformation of $A_0$.

\begin{example}[Degenerate affine Hecke algebras and symplectic reflection algebras]

In the case where $R_0=\bigwedge^2V$, i.e., the span of $\{x\tensor y
- y\tensor x \mid x,y\in V\}$, the above deformation problem was
studied by Drinfeld in \cite{Dr}. Here $A_0=S(V)\lcprod G$ where
$S(V)=T(V)/\lgen \bigwedge^2 V\rgen$ is the algebra of symmetric
tensors on $V$ (a free commutative algebra generated by $V$). Flat
deformations $A$ of $A_0$ are referred to in \cite{Dr} as
\emph{degenerate affine Hecke algebras}. If $V$ is a symplectic vector
space and the action of a finite group $G$ preserves the symplectic
form on $V$, such algebras $A$ are symplectic reflection algebras
introduced by Etingof and Ginzburg in \cite{EG}. 
\end{example}

\subsection{Braided doubles over a group $G$} 
\label{subsec:bd}

Braided doubles were introduced by the authors in \cite{BBadv} for a general bialgebra $H$ in place of the group algebra $\CC G$. Braided doubles are a class of solutions to the  deformation problem considered in \ref{subsect:defproblem}, in a special case where the space $V$ is ``split''.  

Namely, let $V^-$, $V^+$ be finite\dash dimensional $G$\dash modules, $V=V^-\dirsum V^+$. We view $T(V^\pm)$ as graded subalgebras of $T(V)$. Let a graded space $R_0$ of relations in $T^{>0}(V)$ be of the form 
\[
R_0=R^- + R^+ + \mathrm{span}\{f\tensor v - v\tensor f\mid f\in V^+,v\in V^-\}, 
\]
where $R^-\subset T^{>0}(V^-)$ and $R^+\subset T^{>0}(V^+)$ are graded subspaces in $T^{>0}(V)$. Let 
\[
A_0(R^-,R^+) =(T(V)/\lgen R_0\rgen)\lcprod G. 
\]
Then it is not difficult to observe the following isomorphism of vector spaces:
\[
A_0(R^-,R^+)  \cong T(V^-)/\lgen R^-\rgen \tensor \CC G \tensor T(V^+)/\lgen R^+\rgen.
\]
That is, the algebra $A_0(R^-,R^+)$ has three subalgebras $U^\pm\cong T(V^\pm)/\lgen R^\pm\rgen$, $U^0\cong \CC G$ generated by $V^\pm $ and $G$, respectively, and the multiplication map of $A_0(R^-,R^+)$ yields a vector space isomorphism  $U^-\tensor U^0\tensor U^+\xrightarrow{\sim}A_0(R^-,R^+)$. The ``straightening'' relations between $U^0$ and $U^\pm$ that allow us to write any element of $A_0(R^-,R^+)$ as a linear combination of products $u^-u^0u^+$ with $u^i\in U^i$ are: 

--- the semidirect product relations $gv=(g\act v)g$, $fg=g(g^{-1}\act f)$ for $g\in G$, $v\in V^-$, $f\in V^+$;

--- the commutation relation $fv-vf=0$.                  
                                                                                      
Braided doubles are obtained from $A_0(R^-,R^+)$ by deforming the latter commutation relation: zero on the right is replaced by a linear combination of elements of the group $G$, with coefficients which depend linearly on $f$ and $v$. Formally,  
the deformation parameter is a linear map
\[
\beta\colon V^+\tensor V^-\to \CC G,
\]                                  
and the following algebra is a deformation of $A_0(R^-,R^+)$:
\[
   A_\beta(R^-,R^+)=\frac{T(V^-\dirsum V^+)\lcprod G}{\lgen R^-, R^+, 
   \{f\tensor v - v\tensor f-\beta(f\tensor v) \mid f\in V^+, v\in V^-\}\rgen}.
\]
\begin{definition} 
The algebra $A_\beta(R^-,R^+)$ is called a \emph{braided double} over $G$, if it is a flat deformation of $A_0(R^-,R^+)$. 
\end{definition}

Note that $A_\beta(R^-,R^+)$ coincides with $A_0(R^-,R^+)$ when $\beta=0$, justifying the chosen notation. Furthermore, note that the linear map 
\[
T(V^-)/\lgen R^-\rgen \tensor \CC G \tensor T(V^+)/\lgen R^+\rgen \sur A_\beta(R^-,R^+),
\]
given by the multiplication in $A_\beta(R^-,R^+)$, is surjective for any $\beta$. Flatness of the deformation means that this map is one-to-one. 

It follows that, whenever $A_\beta(R^-,R^+)$ is a braided double, 
the algebras $T(V^-)/\lgen R^-\rgen$, $\CC G$ and $T(V^+)/\lgen R^+\rgen$ are embedded in $A_\beta(R^-,R^+)$ as subalgebras --- as are, in fact, $(T(V^-)/\lgen R^-\rgen) \lcprod G$
and $G \rcprod ( T(V^+)/\lgen R^+\rgen)$. For this reason, we say that a braided double
$A_\beta(R^-,R^+)$ has \emph{triangular decomposition} over $\CC G$. 

\begin{example}
\label{ex:bdoubles}
A standard example of a braided double is the quantum universal
enveloping algebra $U_q(\mathfrak g)$ where $\mathfrak g$ is a
semisimple complex Lie algebra. The triangular decomposition of
$U_q(\mathfrak g)$ is $U_q^-\tensor \CC K \tensor U_q^+$ with $K$ a
free Abelian group of rank $l=\mathrm{rk}\,\mathfrak g$ and $U_q^\pm$
quotients of free tensor algebras of rank $l$ by the quantum Serre
relations, see \cite{L}.  

Another example, of prime importance to the present paper, is the rational Cherednik algebras $H_{t,c}(G)$ of a finite complex reflection group $G\subset \GL(V)$. Here the triangular decomposition is of the form $S(V)\tensor \CC G \tensor S(V^*)$, see \cite{EG}. 

The classical universal enveloping algebra $U(\mathfrak g)$ of a
semisimple Lie algebra $\mathfrak g$ is also a braided double, but
over a commutative and cocommutative Hopf algebra $U(\mathfrak h)$
rather than over a group (so not of the type we consider in the
present paper). Here $\mathfrak h$ is a Cartan subalgebra of
$\mathfrak g$. A choice of a Borel subalgebra of $\mathfrak g$
containing $\mathfrak h$ leads to the direct sum decomposition
$\mathfrak g=\mathfrak n^-\dirsum \mathfrak h \dirsum \mathfrak n^+$,
where $\mathfrak n^-$, respectively $\mathfrak n^+$, is spanned by
negative, respectively positive, root vectors. This gives rise to the
triangular decomposition $U(\mathfrak g) \cong U(\mathfrak n^-)\tensor
U(\mathfrak h)\tensor U(\mathfrak n^+)$ where $U(\mathfrak n^+)$ is
generated by simple root vectors modulo the Serre relations, see
\cite{Serre}.  
\end{example}                                                                                                                                            
\subsection{Hierarchy of braided doubles}
                                        
Suppose that the $G$\dash modules $V^-$ and $V^+$ are fixed. The algebra $A_\beta(R^-,R^+)$ depends on the triple $(\beta,R^-,R^+)$ of parameters, where $\beta\in \Hom_\CC(V^+\tensor V^-,\CC G)$ and $R^\pm$ is a graded $G$\dash submodule of $T^{>0}(V^\pm)$.
One is interested in the class of such triples for which $A_\beta(R^-,R^+)$ is a braided double. Clearly, this algebra does not change if $R^-$ is replaced by the $G$\dash invariant two\dash sided ideal $I^-$ of $T(V^-)$ generated by $R^-$; same for $R^+$. 

\begin{definition} Recall that 
the $G$\dash module $(\CC G)_\ad$ (the \emph{adjoint module} for $G$) is the group algebra $\CC G$ where $G$ acts by conjugation, $g\adact h=ghg^{-1}$.
We say that a linear map $\beta\colon V^+\tensor V^-\to \CC G$ is \emph{$G$\dash equivariant}, if $\beta$ is a $G$\dash module map $V^+\tensor V^-\to (\CC G)_\ad$.
\end{definition}

Basic facts about braided doubles, listed in the next theorem, were proved in \cite{BBadv}. We do not reproduce their proof here but explain, see Remark~\ref{rem:notdual} below, how to remove the assumption $V^+=(V^-)^*$ made in \cite{BBadv}. 
We denote by $[a,b]$ the commutator $ab-ba$ in a given associative algebra $A$, and extend this notation to subspaces $U$, $V$ of $A$, writing $[U,V]=\mathrm{span}\{[u,v]\mid u\in U,v\in V\}$.

\begin{theorem} 
\label{thm:basicdoubles}
Let $G$ be a group and $V^-$, $V^+$ be finite\dash dimensional $G$\dash modules. Let $\beta$ denote a linear map $V^+\tensor V^-\to \CC G$, and $I^\pm$ denote proper graded ideals of $T(V^\pm)$. 

1. $A_\beta(0,0)$ is a braided double, if and only if $\beta$ is $G$\dash equivariant. 

2. Let $\beta$ be $G$\dash equivariant. Then $A_\beta(I^-,I^+)$ is a braided double iff $I^\pm$ are $G$\dash invariant and 
\[
[V^+,I^-]\subset I^-\tensor \CC G, \quad [I^+,V^-]\subset \CC G \tensor I^+\quad 
\text{in the algebra $A_\beta(0,0)$}.
\]

3. The sum $I^-_\beta$ of all ideals $I^-\subset T^{>0}(V^-)$ that satisfy 2.\ also satisfies 2., same for $I^+_\beta$. 
For any braided double $A_\beta(I^-,I^+)$ there are surjective algebra homomorphisms
\[
A_\beta(0,0) \sur A_\beta(I^-,I^+) \sur A_\beta(I^-_\beta,I^+_\beta).
\]
%

4. The group $G$ acts on any braided double $A_\beta(I^-,I^+)$. 
As a $G$\dash module, $A_\beta(I^-,I^+)$ is the same as the tensor product $T(V^-)/I^-\tensor (\CC G)_\ad \tensor T(V^+)/I^+$ given by the triangular decomposition.
\qed
\end{theorem}

\begin{remark}
\label{rem:notdual}
Both \cite{BBadv} and \cite{BBselecta} define braided doubles as in \ref{subsec:bd} above but with additional restriction $V^+=(V^-)^*$. It might seem that braided doubles considered in the present paper are more general than in \cite{BBadv,BBselecta}, but in fact they are not. Indeed, let $V^-$, $V^+$ be two finite\dash dimensional $G$\dash modules, and let $\beta\colon V^+\tensor V^-\to\CC G$ be a linear map. Define a new $G$\dash module $U=V^-\dirsum (V^+)^*$. The dual of $U$ is $U^*=V^+\dirsum (V^-)^*$. 
One has 
\[
U^*\tensor U= (V^+\tensor V^-) \dirsum (V^+\tensor {V^+}^*) \dirsum ({V^-}^*\tensor V^-) 
 \dirsum ({V^-}^*\tensor {V^+}^*),
\]
a direct sum of $G$\dash modules. Define 
\[
\beta_U\colon U^*\tensor U \to \CC G, \qquad \beta_U|_{V^+\tensor V^-}=\beta,
\qquad \beta_U|_{V^+\tensor {V^+}^*}=\beta_U|_{{V^-}^*\tensor V^-}=\beta_U|_{{V^-}^*\tensor {V^+}^*}=0.
\]
Then $\beta_U$ is a $G$\dash equivariant map if and only if $\beta$ is. 
  
Now suppose that $\beta$ and $\beta_U$ are $G$\dash equivariant.
In the algebra $A_{\beta_U}(0,0)$ which has triangular decomposition $T(U)\tensor \CC G\tensor T(U^*)$,  the subspace ${V^+}^*$ of $U$ commutes with $U^*$, due to the way $\beta_U$ is defined. Similarly, ${V^-}^*\subset U^*$ commutes with $U$. 
By \cite[Proposition 1.6]{BBselecta}, $A_{\beta_U}({V^+}^*,{V^-}^*)$ is a braided double. 
But $A_{\beta_U}({V^+}^*,{V^-}^*)\cong T(U)/\lgen {V^+}^*\rgen\tensor \CC G\tensor T(U^*)/\lgen {V^-}^*\rgen$ is exactly the algebra $A_\beta(0,0)\cong T(V^-)\tensor \CC G\tensor T(V^+)$.
Moreover, 
braided doubles $A_\beta(I^-,I^+)$ as defined in \ref{subsec:bd}
coincide with braided doubles of the form $A_{\beta_U}(I^- + \lgen
{V^-}^*\rgen,I^+ + \lgen {V^+}^*\rgen)$, introduced in \cite{BBadv}
and~\cite{BBselecta}. 
\end{remark}


\subsection{Morphisms of braided doubles}
\label{subsect:morphisms}

If $A$ and $B$ are braided doubles over the same group $G$, it makes sense
to consider algebra maps between them which preserve the 
braided double structure. 
A morphism between $A$ and $B$ will mean a $G$\dash module algebra map 
$$
A \cong T(V^-)/I^-\tensor \CC G \tensor T(V^+)/I^+
\xrightarrow{f}  
B \cong T(W^-)/J^-\tensor \CC G \tensor T(W^+)/J^+,
$$
such that 
$$
f(V^-)\subseteq W^-, \qquad f(V^+)\subseteq W^+, 
\qquad f|_{\CC G}=\id_{\CC G}. 
$$
Note that the restrictions of $f$ to $V^-$ and to $V^+$ determine the 
morphism $f$ uniquely. 
Also note that, pedantically, 
the above condition should be written as 
$f(V^-/(V^-\cap I^-))\subseteq W^-/(W^-\cap J^-)$ etc,
because $V^-$ may not be a subspace of $A$ if the graded ideal $I^-$ 
of $T(V^-)$ 
has a non\dash zero component in degree $1$.

\subsection{Minimal doubles}

Let $V^-$, $V^+$ be two finite\dash dimensional modules over a group $G$, and let $\beta\colon V^+\tensor V^-\to (\CC G)_\ad$ be a $G$\dash module map. By part 3 of Theorem~\ref{thm:basicdoubles}, there are largest possible ideals $I_\beta^\pm$ of $T^{>0}(V^\pm)$ such that $A_\beta(I_\beta^-,I_\beta^+)$ is a braided double. 
\begin{definition}
The algebra 
\[
\co A_\beta := A_\beta(I_\beta^-,I_\beta^+) \ \cong\  T(V^-)/I_\beta^- \tensor \CC G \tensor T(V^+)/I_\beta^+
\]
is called the \emph{minimal braided double} associated to $\beta$.
\end{definition}

Note that $\co A_\beta$ is a quotient of every braided double with given $V^\pm$ and $\beta$.
In general, the ideals $I^\pm_\beta$ of relations in $\co A_\beta$ can only be described implicitly as kernels of certain linear maps \cite[Theorem 4.11]{BBadv}. For example, in degree $1$ the relations are the left and right kernels of the bilinear map $\beta$:
\[
I_\beta^-\cap V^- = \{v\in V^- : \beta(f\tensor v)=0\ \forall f\in V^+\}, 
\quad
I_\beta^+\cap V^+ = \{f\in V^+ : \beta(f\tensor v)=0\ \forall v\in V^-\}. 
\] 

\subsection{Braided doubles in the category of free $R$\dash modules}
\label{subsect:gvect}

We finish this Section by observing that the definition of 
a braided double and some of the properties of braided doubles 
may carry over to monoidal categories more general
than the $\CC$\dash vector spaces $\Vect$. 

Let $R$ be a commutative ring. Consider the abelian monoidal category 
$(R\hyp\mathrm{Mod},\tensor_R,R)$ of $R$\dash modules. 
This category admits free algebras $T_R(V)$ of $V\in \Ob(R\hyp\mathrm{Mod})$. 
If $V$ is a \emph{free} $R$\dash module, 
then so is 
$$
T_R(V)=R \oplus V \oplus 
(V\tensor_{R}V) \oplus \ldots.
$$
Let $G$ be a finite group and $V^-$, $V^+$ be $RG$\dash modules 
which are free of finite rank over $R$.
Suppose that $\beta\colon V^+\tensor_R V^-\to RG_\ad$
is a $G$\dash equivariant map. Then the free $R$\dash module
$$
A_\beta(0,0)=T_R(V^-)\tensor_R RG \tensor_R T_R(V^+)
$$
has an $R$\dash algebra structure of a free braided double over $RG$. 
The multiplication on $A_\beta(0,0)$ is defined in the same way as for $R=\CC$, 
because one can repeat the construction from \cite[Theorem 3.3]{BBadv}, 
using the fact that $T_R(V^\pm)$ have bases over $R$. 

To construct general braided doubles over $R$, 
assume that  $I^\pm\subset T_R(V^\pm)$ are $RG$\dash invariant ideals 
such that $T_R(V^\pm)/I^\pm$ are free $R$\dash modules.
Then the following generalisation of part 2 of Theorem~\ref{thm:basicdoubles} 
holds. If 
$$
[V^+,I^-]\subset I^-\tensor_R R G, \quad [I^+,V^-]\subset R G \tensor_R I^+\quad 
\text{in the algebra $A_\beta(0,0)$},
$$ 
then the $R$\dash algebra $A_\beta(I^-,I^+):=A_\beta(0,0)/\lgen I^-,I^+\rgen$ 
has triangular decomposition
$T_R(V^-)/I^-\tensor_R$ $ RG$ $ \tensor_R T_R(V^+)/I^+$.
This algebra 
$A_\beta(I^-,I^+)$ is then called a \emph{braided double in  the category of free 
$R$\dash modules}. 

We will be especially interested in quadratic doubles. In the case $R=\CC$, 
they were studied in \cite{BBselecta}. From the above we 
deduce the following 
\begin{proposition}
\label{prop:R-quadratic}
Let $V^-$, $V^+$ be $RG$\dash modules which are free $R$\dash modules 
of finite rank, and let $\beta\colon V^+\tensor_R V^-\to RG_\ad$
be a $G$\dash equivariant map.
Let $R^-\subset V^-\tensor_R V^-$, $R^+\subset V^+\tensor_R V^+$ 
be submodules of quadratic   relations such that 
\begin{itemize}
\item
$T_R(V^-)/\lgen R^-\rgen$ and $T_R(V^+)/\lgen R^+\rgen$ 
are free $R$\dash modules;
\item 
$V^+$ commutes with $R^-$ and  $R^+$ commutes with $V^-$ in the free braided double 
$A_\beta(0,0)$. 
\end{itemize}
Then $A_\beta(R^-,R^+)$ is a braided double 
in  the category of free $R$\dash modules.
\qed
\end{proposition}
 
Finally, as elsewhere in the present paper, our primary example of 
$R$ is $\CC\Gamma$ where $\Gamma$ is an abelian group. 
We have the following monoidal category:
$$
(\GVect,\tensor,\mathbb I): = (\text{free }\CC\Gamma\text{-modules},
\tensor_{\CC \Gamma}, \CC \Gamma).
$$
Let $\widetilde G$ be a group containing $\Gamma$ as a central 
subgroup. Then $\CC\widetilde G$ is an algebra over 
$\CC\Gamma$, hence an algebra in $\GVect$.
There will be braided doubles in $\GVect$ with triangular decomposition 
of the form 
$A\cong U^-\tensor_{\CC \Gamma} \CC \widetilde G \tensor_{\CC \Gamma}
U^+$. 
Morphisms between braided doubles in $\GVect$ are defined 
in the same way as in \ref{subsect:morphisms}.

Because $\GVect$ is an additive category which is not abelian in general, 
a quotient of an algebra in $\GVect$ by a two\dash sided ideal 
does not always exist in the category. 
For this reason, we do not claim the existence 
of a minimal double in $\GVect$ corresponding any given $\beta$.

Observe that any braided double in $\GVect$, being free as 
a $\CC \Gamma$\dash module, can be specialised 
to a braided double in $\Vect$: the quotient of $A$ modulo 
$\CC\Gamma_+A$ has triangular decomposition over $\CC G$ where 
$G=\widetilde G / \Gamma$. In particular, $A$ is always a $\Gamma$\dash 
flat deformation of $A/\CC\Gamma_+A$ in the sense of \ref{subsect:flat}.
When $\Gamma=\cycl_k=\langle z\mid z^k=1\rangle$, one has
$$
A|_{z=1} \cong U^-|_{z=1} \tensor_\CC \CC G \tensor_\CC U^+|_{z=1},
\qquad G=\widetilde G |_{z=1}.
$$


\section{Braided Heisenberg doubles and braided Weyl algebras}

Among all braided doubles over a group $G$, there 
is a distinguished class of braided Heisenberg 
doubles, which correspond to Yetter\dash Drinfeld modules for $G$. 
We review their construction.
A new result in this section is that a braided
Heisenberg double is isomorphic to a semidirect
product of $G$ with a braided Weyl algebra. 

\subsection{Braided Heisenberg doubles}
\label{subsect:Heisenberg}

Let $Y$ be a
finite\dash dimensional Yetter\dash Drinfeld module $Y$ for $G$.
Put $Y^-=Y$ and $Y^+=Y^*$. For $f\in Y^*$, $v\in Y$ define
\[
\beta(f\tensor v)=\langle f,v^{(0)}\rangle v^{(1)},
\]
where $\langle\,,\rangle$ is the $G$\dash invariant pairing between $Y^*$ and $Y$, and $\delta(v)=v^{(0)}\tensor v^{(11)}$ is the 
$G$\dash coaction. Then
$\beta$ is a $G$\dash equivariant map because of the Yetter\dash
Drinfeld condition, see \ref{def:YD}.
Therefore, there is the minimal double 
\[
\mathcal H_Y := \co A_\beta
\]
called a \emph{braided Heisenberg double} \cite[Section 5]{BBadv}.
The defining relations in $\mathcal H_Y$ are given in terms of the braiding $\Psi$ on $Y$. Identify $Y^*\tensor Y^*$ with the dual space of $Y\tensor Y$ 
via the pairing 
$$
\langle e\tensor f,v\tensor w\rangle = \langle f,v\rangle\langle e,w\rangle
$$ 
as in \ref{subsect:pairing}.
By \cite[Theorem 5.4]{BBadv}, one has
\[
    I_\beta^-=I_\Psi,
    \qquad
    	I_\beta^+=I_{\tau\Psi^*\tau},
\]
where $I_\Psi$ denotes the kernel of the 
Woronowicz symmetrisers as in Section~\ref{sect:nichols}, and 
$\tau$ is the transposition map on $Y^*\tensor Y^*$.
Observe that 
$\Psi^*$ denotes the adjoint of $\Psi$ with respect to the above pairing
and is the braiding on the Yetter\dash Drinfeld module $Y^*$. 
The $G$\dash module map $\tau \Psi^*\tau \colon Y^*\tensor Y^*\to Y^*\tensor Y^*$ satisfies the braid equation because $\Psi^*$ does. One has the following isomorphism $\cong$ of vector spaces:
\begin{align*}
    	\mathcal H_Y & = (T(Y\oplus Y^*)\lcprod G)\ /\ {\lgen I_\Psi,\ I_{\tau\Psi^*\tau},\ \{f\tensor v -v\tensor f - \langle f,v^{(0)}\rangle v^{(1)}: f\in Y^*,v\in Y\}\rgen} 
\\    	
    	& \cong \mathcal B(Y,\Psi)\tensor \CC G\tensor \mathcal B(Y^*,\tau\Psi^*\tau).
\end{align*}
That is, the triangular decomposition of a braided Heisenberg double
associated to a Yetter\dash Drinfeld module $Y$ is into the group
algebra and two Nichols algebras.

\subsection{Braided Heisenberg double as a semidirect product}


Recall that $\delta\colon Y \to Y \tensor \CC G$  denotes the coaction
on $Y$ and $\Psi$ denotes the braiding on $Y$ defined by the
Yetter\dash Drinfeld module structure. We write $\tau$ for the trivial
braiding, $\tau(u\tensor v)=v\tensor u$ for $u,v\in Y$. 
As before, it is easy to
see that $\tau\Psi\tau$ is also a braiding on $Y$. We start with the
following lemma.  

\begin{lemma}
\label{lem:shift}
There exists an algebra automorphism $\widetilde\delta$ of $T(Y)\lcprod G$ defined by $\widetilde\delta(v)=\delta(v)$ and $\widetilde\delta(g)=g$ for $v\in Y$, $g\in G$. One has $\widetilde\delta(I_\Psi\tensor \CC G)=I_{\tau\Psi\tau}\tensor \CC G$, so that $\widetilde \delta$ induces an algebra isomorphism between $\mathcal B(Y,\Psi)\lcprod G$ and $\mathcal B(Y,\tau\Psi\tau)\lcprod G$. 
\end{lemma}

\begin{proof}
The algebra $T(Y)\lcprod G$ is the quotient of $T(\CC G\dirsum Y)$ modulo the ideal generated by $g\tensor v-(g\act v)\tensor g$ and $g\tensor h - gh$, where $g,h\in G$, $v\in Y$ and $gh$ denotes the product of $g$ and $h$ in $G$. Consider the algebra endomorphism $\widetilde\delta$ of $T(\CC G\dirsum Y)$ defined on $\CC G$ and on $Y$ as specified in the lemma. Then $\widetilde \delta$ preserves the relation $g\tensor h - gh$. Assuming $v\in Y_x$ where $x\in G$, $\widetilde \delta$ maps $g\tensor v-(g\act v)\tensor g$ into $g\tensor v \tensor x - (g\act v)\tensor gxg^{-1} \tensor g$, which is zero modulo the defining relations of $T(Y)\lcprod G$. Therefore, $\widetilde\delta$ is a well\dash defined algebra endomorphism of $T(Y)\lcprod G$. 
It is not difficult to see that the inverse $\widetilde\delta^{-1}$ of $\widetilde\delta$ is given by $\widetilde\delta^{-1}(g)=g$ and $\widetilde\delta^{-1}(v)=v\tensor x^{-1}$ for $v\in Y_x$. Hence $\widetilde\delta$ is an automorphism of $T(Y)\lcprod G$.

We will now compute the linear endomorphism $\widetilde\delta(\Psi\tensor \id_{\CC G})\widetilde\delta^{-1}$ of the vector space $Y\tensor Y \tensor \CC G$. 
Let $u\in Y_x$ and $v\in Y_y$ with $x,y\in G$. Let $g\in G$. Then 
\begin{align*}
    u\tensor v \tensor g \qquad \xrightarrow{\widetilde\delta^{-1}}
    & \quad u \tensor (x^{-1}\act v) \tensor x^{-1}y^{-1}g 
    \\
     \xrightarrow{\Psi\tensor \id_{\CC G}}
    & \quad x\act (x^{-1}\act v)\tensor u \tensor x^{-1}y^{-1}g = v\tensor u \tensor x^{-1}y^{-1}g
    \\
    \xrightarrow{\widetilde\delta}
    & \quad v\tensor (y\act u) \tensor g.
\end{align*}
This shows that $\widetilde\delta(\Psi\tensor \id_{\CC
  G})\widetilde\delta^{-1} = \tau \Psi\tau \tensor \id_{\CC G}$. Now it
easily follows that $\widetilde\delta([n]!_\Psi
\tensor \id_{\CC
  G})\widetilde\delta^{-1} =$ 
$[n]!_{\tau\Psi\tau}\tensor \id_{\CC G}$
as endomorphisms of $Y^{\tensor n}\tensor \CC G$. Taking the kernels
on both sides and summing over all $n\ge 2$ yields
$\widetilde\delta(I_\Psi\tensor \CC G)=I_{\tau\Psi\tau}\tensor \CC
G$. 
\end{proof}

\begin{remark}
The Nichols algebras $\mathcal B(Y,\Psi)$ and $\mathcal B(Y,\tau\Psi\tau)$ are, in general, not isomorphic. Rather, the identity map on $Y$ extends to an algebra isomorphism $\mathcal B(Y,\Psi)\iso \mathcal B(Y,\tau\Psi\tau)^{\mathrm{op}}$. The isomorphism $\widetilde\delta$ between 
$\mathcal B(Y,\Psi)\lcprod G$ and $\mathcal B(Y,\tau\Psi\tau)\lcprod G$, constructed in Lemma~\ref{lem:shift}, does not restrict to a map between  $\mathcal B(Y,\Psi)$ and $\mathcal B(Y,\tau\Psi\tau)$.
\end{remark}



Observe that the same result holds for $Y^*$:

\begin{corollary}
\label{cor:Y*}
The map $Y^*\to \CC G \tensor Y^*$ defined by $f\in (Y^*)_g \mapsto g\tensor f$ extends to an algebra isomorphism $G\rcprod \mathcal B(Y^*) \iso G\rcprod \mathcal B(Y^*,\tau\Psi^*\tau)$. 
\qed
\end{corollary}


\begin{definition}  The \emph{braided Weyl algebra} of a finite\dash dimensional Yetter\dash Drinfeld module $Y$ is 
\[
\mathcal A_Y = \mathcal B(Y) \tensor \mathcal B(Y^*)
\]
as a vector space, where $\mathcal B(Y)$ and $\mathcal B(Y^*)$ 
embed as subalgebras, and the relations 
\[ 
               fv - (g\act v)f = \langle f,v\rangle\cdot 1, \qquad f\in (Y^*)_g, \ v\in Y
\]
hold, determining the multiplication in $\mathcal A_Y$ uniquely. 
\end{definition}

\begin{remark}

It follows from \ref{subsect:functornichols} 
that $\mathcal B(Y)$ and $\mathcal B(Y^*)$ are dually paired braided Hopf algebras in the category $\YD G$:
they are quotients of dually paired $T(Y)$ and $T(Y^*)$ 
modulo the respective kernels of the pairing. They have
$Y$ and $Y^*$ as the respective spaces of primitive elements.
It then follows from a general categorical construction of a braided Weyl algebra in \cite{Mfree} that the multiplication in $\mathcal A_Y$ is well\dash defined. In the case of a trivial $G$\dash grading on $Y$, $\mathcal A_Y$ is the ordinary Weyl algebra with underlying vector space $S(Y)\tensor S(Y^*)$ and defining relation $fv-vf=\langle f,v\rangle\cdot 1$ for $f\in Y^*$, $v\in Y$.

\end{remark}

The following theorem is a version of \cite[Proposition 1.23 and Example 1.25]{BBselecta}: in \cite{BBselecta}, we dealt only with braided doubles with quadratic relations, whereas here we do not have that restriction. 

\begin{theorem}
\label{thm:weyl}
Let $Y$ be a finite\dash dimensional Yetter\dash Drinfeld module for a group $G$. Then there is an isomorphism 
\[ 
\phi\colon \mathcal A_Y \lcprod G \iso \mathcal H_Y
\]
of algebras, defined on generators $v\in Y$, $f\in Y^*$ and $g\in G$ of $\mathcal A_Y \lcprod G$ by $\phi(v)=v$, $\phi(f)=x\tensor f$ if $f\in (Y^*)_x$ where $x\in G$, and $\phi(g)=g$. 
\end{theorem}

\begin{proof}
We already know from Corollary~\ref{cor:Y*} that $\phi$, defined as above on  
$f\in Y^*$ and $g\in G$, extends to an algebra isomorphism between
$G\rcprod \mathcal B(Y^*)$ and $G\rcprod \mathcal B(Y^*,\tau \Psi^*\tau)
\subset \mathcal H_Y$. 
Trivially, $\phi|_Y$ extends to an embedding of $\mathcal B(Y,\Psi)$ in $\mathcal H_Y$. To check that $\phi$ is well defined, it remains to show that $\phi$ maps the commutation relations between $Y^*$, $G$ and $Y$ in $\mathcal A_Y$ into relations of $\mathcal H_Y$. But this was done in \cite[Proposition 1.23 and Example 1.25]{BBselecta}.  
\end{proof}

\subsection{The braided Heisenberg double is a $G$\dash graded algebra}

Our use of Theorem~\ref{thm:weyl} is that it allows us to view the braided Heisenberg double $\mathcal H_Y$ as an algebra in the category $\YD G$. 
Here, Examples \ref{subsect:ex1} and \ref{subsect:ex2} from the Introduction
come together to produce a $G$\dash grading on $\mathcal H_Y$:

%
%
%
%
%

\begin{proposition}
\label{prop:twist-h}
1. The group $G$ acts on $\mathcal H_Y$ by conjugation:
$g\act x=gxg^{-1}$, and makes $\mathcal H_Y$ a $G$\dash algebra. 

2. There is a $G$\dash grading on $\mathcal H_Y\cong \mathcal B(Y,\Psi)\tensor \CC G \tensor \mathcal B(Y^*,\tau\Psi^*\tau)$ 
such that the $G$\dash degree of $g\in G$ is $g$, 
the $G$\dash degree of $v\in Y_g$ is $g$ 
and the $G$\dash degree of $Y^*$ is $1$. 

3. The $G$\dash action and the $G$\dash grading make $\mathcal H_Y$ an algebra 
in the category $\YD G$. 

\end{proposition}

\begin{proof}
1.\ is a standard fact about braided doubles, cf.\ Theorem~\ref{thm:basicdoubles}.
Note that, by the semidirect product relations, 
the adjoint action of $G$ extends the $G$\dash action 
on $\mathcal B(Y)$ and on $\mathcal B(Y^*,\tau \Psi^*\tau)$. 

2.\ 
Recall the algebra isomorphism $\phi\colon \mathcal A_Y\lcprod G\to\mathcal H_Y$ from  Theorem~\ref{thm:weyl}. By the general construction in \cite{Mfree}, 
the braided Weyl algebra $\mathcal A_Y\cong \mathcal B(Y)\tensor \mathcal B(Y^*)$ is an algebra in the category $\YD G$, so is $G$\dash graded by extending the $G$\dash grading from $Y$ and $Y^*$. This grading and the natural $G$\dash grading on $\CC G$ extend to a $G$\dash grading on $\mathcal A_Y\lcprod G$. This, in turn, induces a $G$\dash grading on $\mathcal H_Y$ via the isomorphism $\phi$. The result is exactly as described in the Proposition. In particular, 
if $f\in (Y^*)_x$, then $\phi(f) = xf$ has degree $x$ in $\mathcal H_Y$ which means that $f$ must have trivial $G$\dash degree in $\mathcal H_Y$. 

3.\ The Yetter\dash Drinfeld condition on $\mathcal H_Y$ easily 
follows from the said condition on $\mathcal B(Y)$ and 
on $(\CC G)_\ad$. 
\end{proof}

The $G$\dash grading on a braided Heisenberg double $\mathcal H_Y$ 
allows us to extend $\mathcal H_Y$ by cocycles.
We do this in the next theorem.

\begin{theorem}[Cocycle extension of a braided Heisenberg double]
\label{thm:ext-heisenberg}
Let $G$ be a finite group, $Y$ be a finite\dash dimensional Yetter\dash Drinfeld 
module for $\CC G$, $R$ be a commutative $\CC$\dash algebra 
and $\mu\in Z^2(G,R^\times)$. 
The extension of $\mathcal H_Y$ by the cocycle $\mu$
is an algebra in the category $\YDE R G$ with the following 
triangular decomposition:
$$
(\widetilde{\mathcal H_Y})_\mu \cong 
\mathcal B(\widetilde Y_\mu) \tensor_R RG_\mu 
\tensor_R \mathcal B(\widetilde Y^*,\tau \Psi^*\tau),
$$
Here $\widetilde Y_\mu$  is as defined in Corollary~\ref{cor:Ymu} and 
$\widetilde Y^*=R\tensor_\CC Y^*$ is the trivial extension of $Y^*$. 
The product $\star$ on $(\widetilde{\mathcal H_Y})_\mu$ 
is completely described by the following. The $R$\dash submodules
$A_{-1}=\mathcal B(\widetilde Y_\mu)$, $A_0=RG_\mu$, 
$A_1=\mathcal B( \widetilde Y^*,\tau\Psi^*\tau)$ 
are subalgebras over $R$, as are $A_{-1}A_0$ and $A_0A_1$. The relations 
\begin{align*}
   & g\star v=(g\actchi v)\star g, \qquad g\star f=(g\act f)\star g,
\\    
   & f\star v - v\star f = \langle  f,v^{(0)} \rangle v^{(1)}
   \qquad\text{for }g\in G,\ v\in Y,\ f\in Y^*
\end{align*}
hold where $\chi=\chi_\mu$ and $\act$ is the original action of $G$ 
on $Y$ and $Y^*$.
\end{theorem}
\begin{proof}
The first part (tensor decomposition)
follows from Proposition~\ref{prop:twist-h} and 
Corollary~\ref{cor:Ymu}. 
Let us verify the cross\dash commutation relations between 
$Y$, $G$ and $Y^*$ in 
$(\widetilde{\mathcal H_Y})_\mu$. 
Assume that $v\in Y_h$ where $h\in G$.
By definition of $\star$ one has 
$g\star v=\mu(g,h)gv=\mu(g,h)(g\act v)g$ 
which can be written as 
$\mu(g,h)\mu(ghg^{-1},g )^{-1}(g\act v)\star g=(g\actchi v)\star g$.
Furthermore, since $\mu$ is a normalised cocycle and $f$,
$g\act f\in Y^*$ have  $G$\dash degree $1$,  one has $g\star f=gf=(g\act f)g=(g\act f)\star g$, and 
$f\star v-v\star f=fv-vf$ is as given in \ref{subsect:Heisenberg}.
\end{proof}


\subsection{The case $R=\CC \Gamma$}
\label{subsect:cgamma}

Consider a special case of the cocycle extension 
of $\mathcal H_Y$ where  $\Gamma$ an abelian group, 
$R=\CC \Gamma$ and $\mu\in Z^2(G,\Gamma)$. Note that 
$$
RG_\mu \cong \CC \widetilde G_\mu, 
$$
where $\CC \widetilde G_\mu$ is the group algebra 
of the central extension $\widetilde G_\mu$ of $G$ by $\Gamma$. 
Explicitly, recall that $\widetilde G_\mu$, as a set, is the 
cartesian product $\Gamma\times G=\{(z,g)\mid z\in \Gamma,g\in G\}$, whereby the above isomorphism is 
$$
zg\mapsto (z,g), \qquad z\in \Gamma, \qquad g\in G.
$$
Recall the monoidal category $\GVect$ introduced in \ref{subsect:gvect}. 
Theorem~\ref{thm:ext-heisenberg} implies the following
\begin{corollary}
The cocycle extension $(\widetilde{\mathcal H_Y})_\mu$ of $\mathcal H_Y$ 
is a braided double in the category $\GVect$. 
\qed
\end{corollary}

As a vector space over $\CC$, 
the algebra $(\widetilde{\mathcal H_Y})_\mu $ has the tensor decomposition
$$
\mathcal B(Y)\tensor \CC \Gamma \tensor \CC G \tensor \mathcal B(Y^*,\tau\Psi^*\tau).
$$
 However, the subspace $\mathcal B(Y)$ is not a subalgebra of  
 $(\widetilde{\mathcal H_Y})_\mu$, 
 because the relations between elements of $Y$ 
 in $(\widetilde{\mathcal H_Y})_\mu $ involve elements of the group $\Gamma$. 
Thus, the extension $(\widetilde{\mathcal H_Y})_\mu$ with $\mu\in Z^2(G,\Gamma)$ 
has triangular decomposition over $\CC \Gamma$ but not over $\CC$. 

As one might expect, if the cocycle $\mu$ is scalar\dash valued, the twist of $\mathcal H_Y$ by $\mu$ will have triangular decomposition over $\CC$. 
Denote by $(Y,\actchi)$ the Yetter\dash Drinfeld module 
which is the same $G$\dash graded space as $Y$ but where $G$ acts by $\actchi$.
The above arguments can be easily seen to imply

\begin{corollary}
If $\mu\in Z^2(G,\CC^\times)$, the twist 
$(\mathcal H_Y)_\mu$ has triangular decomposition 
$$
\mathcal B(Y, \actchi) \tensor \CC G_\mu \tensor \mathcal B(Y^*, \tau\Psi^*\tau),
$$ 
with cross\dash commutation relations 
$g\star v=(g\actchi v)\star g$, $g\star f=(g\act f)\star g$,   
$f\star v - v\star f = \langle  f,v^{(0)} \rangle v^{(1)}$
for $g\in G$, $v\in Y$, $f\in Y^*$, where $\chi=\chi_\mu$.
\qed
\end{corollary}

Note that the cocycle twist
$(\mathcal H_Y)_\mu$ is an algebra with triangular decomposition but
not a braided double: the twisted group algebra $\CC G_\mu$ is not, in
general, a Hopf algebra.


\section{Extensions and twists of other braided doubles}
\label{sect:embedding}

We now turn to braided doubles over $G$ which are not Heisenberg.
They will not in general be $G$\dash graded. 
However, a construction which we describe in this section
realises some braided doubles as subdoubles of braided Heisenberg doubles. 
We then define extensions of braided doubles, for simplicity 
restricting ourselves to cocycles $\mu\in Z^2(G,\Gamma)$ where 
 $\Gamma$ is an abelian group.

\subsection{Subdoubles of braided Heisenberg doubles}

Let $Y$ be a finite\dash dimensional Yetter\dash Drinfeld module for
$G$, and let $V^-\subset Y$, $V^+\subset Y^*$ be $G$\dash submodules,
not necessarily $G$\dash graded.  
Denote by $A_{V^-,V^+}$ the subalgebra of $\mathcal H_Y$ generated by
$V^-$, $G$ and $V^+$. Then $A_{V^-,V^+}$ has the structure of a
braided double on $V^-$ and $V^+$, with triangular decomposition  
\[
     A_{V^-,V^+} \cong U^- \tensor \CC G \tensor U^+ \quad \subset \quad \mathcal H_Y \cong \mathcal B(Y,\Psi)\tensor \CC G \tensor \mathcal B(Y^*,\tau\Psi^*\tau).
\]
Here  $\Psi$ is the braiding on the Yetter\dash Drinfeld module $Y$ and $\mathcal B(Y,\Psi)$ is the Nichols algebra of $Y$.
The subalgebra $U^-=T(V^-)/T(V^-)\cap I_\Psi$ of $\mathcal B(Y,\Psi)$ is generated in degree one by $V^-$; similarly for $U^+$. Note that both $U^-$ and $U^+$ are graded but not necessarily $G$\dash graded algebras. The commutator between $f\in V^+$ and $v\in V^-$ in $A_{V^-,V^+}$ is defined as their commutator in $\mathcal H_Y$:
\[
\beta_{V^+,V^-}(f\tensor v)=fv-vf \in \CC G \subset \mathcal H_Y, 
\]
so that
\[
A_{V^-,V^+}
=A_{\beta_{V^+,V^-}}(T(V^-)\cap I_\Psi,T(V^+)\cap I_{\tau\Psi^*\tau}).
\] 

\subsection{A counterexample to minimality of $A_{V^-,V^+}$}

In general, the braided double
$A_{V^-,V^+}=A_{\beta_{V^+,V^-}}(T(V^-)\cap I_\Psi,T(V^+)\cap
I_{\tau\Psi^*\tau})$ is not minimal, although it is embedded in a
minimal double $\mathcal H_Y$.  
An extreme example is where $V^-$ and $V^+$ are Yetter\dash Drinfeld 
submodules of $Y$, $Y^*$ orthogonal with respect to the evaluation
pairing between $Y^*$ and $Y$. Then $V^-$ and $V^+$ commute in
$\mathcal H_Y$, so that $\beta_{V^+,V^-}=0$. The minimal double
associated to $\beta_{V^+,V^-}$ is $\co A_0$ with triangular
decomposition $\CC \tensor \CC G \tensor \CC\cong\CC G$. This is not
the same as $A_{V^-,V^+}$.

\subsection{The embedding theorem}

The following theorem shows that for any $\beta$, there exists a 
braided double $A_\beta(I^-,I^+)$ with some relations --- not necessarily 
a minimal double --- which embeds in a braided Heisenberg double. 
We will not give a proof of this theorem because it follows from
\cite[Theorem 6.9]{BBadv}, subject to 
minor modifications explained above in Remark \ref{rem:notdual}.

Recall that a morphism $f$ between two braided doubles over $G$ is 
determined by the restriction of $f$ to degree one generators, 
see \ref{subsect:morphisms}. 

\begin{theorem}
\label{thm:embedding}
Let $G$ be a finite group, $V^\pm$ be $G$\dash modules, $\dim V^\pm<\infty$. Let  $\beta\colon V^+\tensor V^-\to \CC G$
be a $G$\dash equivariant map. Let $A_\beta(0,0)$ denote the free braided double
with triangular decomposition $T(V^-)\tensor \CC G \tensor T(V^+)$. 
There exists a finite\dash dimensional Yetter\dash Drinfeld module $Y$ and 
a  morphism
$$
f\colon A_\beta(0,0)\to \mathcal H_Y 
$$
of braided doubles.
If $I^\pm=\ker f|_{T(V^\pm)}\subset T(V^\pm)$, 
the morphism $f$ induces an embedding
$$
A_\beta(I^-,I^+)\hookrightarrow \mathcal H_Y
$$
of braided doubles over $G$.
\qed
\end{theorem}

\subsection{Extensions and twists of braided doubles}
\label{subsect:twist-other}

Let $V^-$, $V^+$ be finite\dash dimensional $G$\dash modules,  
$\beta\colon V^+\tensor V^-\to \CC G$ be a $G$\dash equivariant map, 
and $A=A_\beta(I^-,I^+)$ be a braided double where 
$I^\pm$ is a graded two\dash sided $G$\dash invariant ideal of $T(V^\pm)$. 
In general, $A$ is not a $G$\dash graded algebra.
However, motivated by the Embedding Theorem~\ref{thm:embedding}, we assume that 
there is a finite\dash dimensional Yetter\dash Drinfeld 
module $Y$ for $G$ and a morphism
$$
f\colon A \to \mathcal H_Y
$$
of braided doubles over $G$.

Let $\Gamma$ be an abelian group and $\mu\in Z^2(G,\Gamma)$. 
We would like to extend the braided double $A$ by the cocycle $\mu$. 
To simplify notation, assume $\Gamma=\cycl_k=\langle z\mid z^k=1\rangle$
and denote the specialisation map by $\cdot |_{z=1}$. 
%

\begin{definition}
\label{def:ext-double}
In the above setup, an \emph{extension of $A$ by $\mu$ covering 
the morphism $f$}
is a commutative diagram 
$$
\begin{CD}
\widetilde A @>\widetilde f>> (\widetilde {\mathcal H_Y})_\mu\\
@VV{\cdot|_{z=1}}V @VV{\cdot|_{z=1}}V\\
A @>f>> \mathcal H_Y
\end{CD}
$$
where $\widetilde A$ is a braided double in $\GVect$ and
$\widetilde f$ is a morphism of braided doubles in $\GVect$.
%
We will also say that $\widetilde A$ is the extension of $A$ if 
the rest of the diagram is implied. 
\end{definition}


We make some remarks about this definition. 
Observe that, while cocycle extensions of braided Heisenberg doubles 
are canonical, extensions of other braided doubles are not, and in general 
they depend on $f$.  An extension of $A$ by a given cocycle $\mu$ 
covering a given morphism $f$ is not guaranteed to exist; 
we do not know if it is unique. 
Note that, while $f(A)$ is a subdouble of $\mathcal H_Y$ 
and is itself a braided double, 
$\widetilde f(\widetilde A)$ may not be a braided double in $\GVect$
because it may fail to be a free $\Gamma$\dash module. 
So an extension of $A$ does not necessarily give an extension of $f(A)$. 
An example of this situation is given in the next section. 

%

Motivated by Corollary~\ref{cor:twist}, we now define
a cocycle twist of a braided double as 
follows.

\begin{definition}[Cocycle twist of a braided double]
In the above setting, a cocycle twist of 
a braided double $A$ is a specialisation 
at $z=q$ of some extension of 
$A$ by a cocycle $\mu\in Z^2(G,\cycl_k)$.
Here $q\in \CC^\times$ is a $k$th root of $1$. 
\end{definition}

%



\section{Covering Cherednik algebras and spin Cherednik algebras}

\label{sect:cherednik}

In this last section of the paper, 
we apply the cocycle extension and cocycle 
twist techniques to rational Cherednik algebras 
$H_{0,c}(S_n)$
to produce new algebras with triangular decomposition.

\subsection{The rational Cherednik algebra}

Let $n$ be fixed and %
let $t,c\in \CC$.   
The \emph{rational Cherednik algebra} of the group $S_n$, denoted $H_{t,c}$, 
was introduced by Etingof and Ginzburg 
in \cite[\S4]{EG} as a degenerate version of the double affine Hecke algebra
of Cherednik. 
The algebra $H_{t,c}$ has generators $x_1,\ldots, x_n$,  the group $S_n$,
and $y_1,\ldots,y_n$ and relations
\begin{align*}
& x_i x_j = x_j x_i, \qquad y_i y_j = y_j y_i, \qquad 
\sigma x_i = x_{\sigma(i)}\sigma, \qquad
\sigma y_i = y_{\sigma(i)}\sigma, 
\\
&y_i x_j - x_j y_i = c\cdot (i\, j),\quad i\ne j;
\qquad
y_i x_i - x_i y_i =  t\cdot 1 - c \sum_{k\ne i}(k\, i).
\end{align*}
Here $1\le i,j\le n$ and $\sigma\in S_n$.
By the Poincar\'e\dash Birkhoff\dash Witt type theorem
for rational Cherednik algebras proved in \cite{EG},  
$H_{t,c}$ is the braided double $A_\beta(\wedge^2 V^-, \wedge^2 V^+)$, 
where $V^-$ is the $\CC$\dash linear span of $x_1,\ldots, x_n$ and 
$V^+$ is the span of $y_1,\ldots,y_n$.
The map $\beta\colon V^+\tensor V^-\to \CC S_n$  can be read off 
the commutator relations between the $y_i$ and $x_j$ above. 

%

For the rest of the paper we assume $t=0$. 
In this case the rational Cherednik algebra has an 
interesting cocycle extension which is an algebra over $\CC \cycl_2$.

\subsection{The reduced Cherednik algebra and its embedding 
in a braided Heisenberg double}

\label{subsect:cheredik_emb}

The braided double $H_{0,c}$ is not minimal and admits a 
finite\dash dimensional quotient, the minimal double. 
It was shown in \cite[Proposition 7.13]{BBadv} that if 
$c\ne0$, this minimal double is the 
\emph{reduced Cherednik algebra} introduced by Gordon \cite{Gor}:
$$
\co H_{0,c} \cong P_n \tensor \CC S_n \tensor P_n^\vee,
$$
where $P_n=\CC[x_1,\ldots,x_n]/\lgen f_1,\ldots,f_n\rgen$ is the 
coinvariant algebra of $S_n$, the $f_i$ being the elementary symmetric functions,
$\deg f_i=i$. (Here $P_n^\vee$ denotes a copy of $P_n$ generated by 
$y_1,\ldots,y_n$.)
One has $\dim P_n=\dim P_n^\vee=n!$ and $\dim \co H_{0,c}=(n!)^3$. 
%
We note the morphism 
$$
f\colon H_{0,c} \sur \co H_{0,c} 
$$
of braided doubles. 

Let $X_n$ denote the set of transposition in $S_n$,
 and let 
 $$
 Y_n:=(X_n,q_1)
 $$ 
be the Yetter\dash Drinfeld module for $\CC S_n$ 
described in \ref{subsect:fkalgebras}. 
The underlying vector space of $Y_n$ is 
$\CC X_n=\CC\hyp\mathrm{span}\{e_\tau: \tau\in X_n\}$. 
To distinguish $Y_n$ from its dual space, 
we assume that $Y_n^*$ 
is spanned by $\{e_\tau^*: \tau\in X_n\}$, a basis dual to 
$\{e_\tau\}$. 
Note that, by definition of the braided 
Heisenberg double in \ref{subsect:Heisenberg},
$$
e_\tau^* e_\upsilon - e_\upsilon e_\tau^* = \delta_{\tau\upsilon}
\tau\in \mathcal H_{Y_n},\qquad \tau,\upsilon\in X_n.
$$
We recall the embedding of $\co H_{0,c}$ 
in the braided Heisenberg double $\mathcal H_{Y_n}$, 
constructed in \cite{BBadv} for the more general situation of a
complex reflection group. Here we quote the result from
\cite{BBadv} in the case of $S_n$, 
noting the appearance of the Fomin\dash Kirillov ``Dunkl elements''
$\theta_j\in Y_n$ as in \ref{subsect:fkalgebras}. 
We denote by $\theta_j^*$ a copy of $\theta_j$ in the space $Y_n^*$:

\begin{proposition}[see \cite{BBadv}]
\label{prop:embedding}
Let $c\ne 0$. There is an embedding of $\co H_{0,c}$ 
as a subdouble in the braided 
Heisenberg double $\mathcal H_{Y_n}$ given by
\begin{align*}
x_j & \mapsto \theta_j = - \sum_{1\le i<j}e_{(i\,j)} + \sum_{j<i\le
  n}e_{(j\,i)}\in Y_n,  
\\
y_j& \mapsto  -c\theta_j^* =
-c\bigl(\,- \sum_{1\le i<j}e^*_{(i\,j)} + \sum_{j<i\le n}e^*_{(j\,i)}
\bigr)\in Y_n^*,
\end{align*}
and $\sigma\mapsto \sigma$ for $\sigma\in S_n$. Therefore, 
the morphism $f\colon H_{0,c}\sur\co H_{0,c}$ 
can be viewed as the morphism 
$$
f\colon H_{0,c}\to \mathcal H_{Y_n}, 
\qquad f(x_j)=\theta_j, \qquad f(y_j)=-c\theta^*_j
$$
of braided doubles.
\qed
\end{proposition}

\subsection{The covering Cherednik algebra $\widetilde H_{0,c}$}

We will now construct an extension 
of $H_{0,c}$ 
by a cocycle $\mu\in Z^2(S_n,\cycl_2)$.
We are going to use the cocycle we studied above in \ref{subsect:chi}, namely
$$
\mu=[1,z].
$$ 
This is a non\dash trivial cocycle if $n\ge 4$, 
which we will assume from now on. 
Our extension will be covering the above braided double morphism 
$f$, which means that we are looking for a commutative diagram
%
$$
\begin{CD}
\widetilde H_{0,c} @>\widetilde f>> (\widetilde {\mathcal H_{Y_n}})_{[1,z]}\\
@VV{\cdot|_{z=1}}V @VV{\cdot|_{z=1}}V\\
H_{0,c} @>f>> \mathcal H_{Y_n}
\end{CD}
$$
Denote $\CC \cycl_2$ by $R$ and recall the
triangular decomposition
$$
(\widetilde {\mathcal H_{Y_n}})_\mu \cong 
\mathcal B(RX_n,q_z) \tensor_R \CC T_n
\tensor_R \mathcal B(\widetilde{Y}_n^*), 
$$
in $\cycl_2\text{-}\Vect$,
where $T_n=(\widetilde{S_n})_{[1,z]}$ is the Schur covering group of $S_n$,
see Theorem~\ref{thm:schur}.
The module 
$(\widetilde{Y_n})_{[1,z]}=(RX_n,q_z)\in \Ob(\YDE{R}{S_n})$ was defined in 
\ref{subsect:nichols-z}, and 
$\widetilde{Y}_n^*=R\tensor_\CC Y_n^*$ is a trivial extension 
of the Yetter\dash Drinfeld module $Y_n^*$. 
Explicitly, 
$$
\mathcal B(\widetilde{Y}_n^*) \cong R\tensor_\CC \mathcal B(Y_n),
$$
where the isomorphism just relabels 
the generators $e_\tau$, $\tau\in X_n$ on the right 
as $e_\tau^*$ on the left. This is because 
the Yetter\dash Drinfeld module $Y_n$ is self\dash dual, 
cf.\ \cite{B}.

The algebra $\widetilde H_{0,c}$ over $R$ must be a $\cycl_2$\dash flat 
deformation of $H_{0,c}$. It will be generated by 
$\widetilde x_1,\ldots,\widetilde x_n$, $T_n$ and 
$\widetilde y_1,\ldots,\widetilde y_n$, where $\widetilde x_i$ is some choice 
of a preimage of $x_i$ in $\widetilde H_{0,c}$, 
i.e., $\widetilde x_i|_{z=1}=x_i$. Because of the flatness of deformation 
requirement, we 
should lift the commutation relations between the $x_i$ in such a way 
that  the $\widetilde x_i$ generate the subalgebra of $\widetilde H_{0,c}$ 
isomorphic, as an $R$\dash module, to $R\tensor_\CC \CC[x_1,\ldots,x_n]$.

Therefore, it makes sense to study the commutation relations 
between $\widetilde f(\widetilde x_1),\ldots, \widetilde f(\widetilde x_n)$ 
in the algebra
$\mathcal B(RX_n,q_z)$. Note that $\widetilde f(\widetilde x_j)$ is the lift 
of $\theta_j\in Y_n$ to $(\widetilde {Y_n})_\mu$. 
(Also, $\widetilde f(\widetilde y_j)$ is a lift of $\theta_j^*$ 
to $\widetilde Y_n^*$, but because there is no twisting applied to $\mathcal
B(Y_n^*)$ to get $\mathcal B(\widetilde Y_n^*)$, 
any 
lifts of the $\theta_j^*$ will commute in $\mathcal 
B(\widetilde Y_n^*)$.)

In choosing the correct lift of the $\theta_j$, we are guided by Majid's 
formula for the flat connections $\alpha_j$, see \ref{subsect:fkalgebras}. 
The result is the following proposition, which plays a crucial role in 
constructing the algebra $\widetilde H_{0,c}$. 

\begin{proposition}
\label{prop:z-commute}
The elements 
$$
\widetilde \theta_j
= - \sum_{1\le i<j}e_{(i\,j)} + z \sum_{j<i\le n}e_{(j\,i)}
\quad \in RX_n, \qquad j=1,\ldots,n,
$$
$z$\dash commute in the algebra $\mathcal B(RX_n,q_z)$, that is, 
$\widetilde\theta_i\widetilde\theta_j = 
z\widetilde\theta_j\widetilde\theta_i$ if $i\ne j$. 
\end{proposition}
\begin{proof}
Consider the quadratic element $u_{ij}=\widetilde\theta_i\widetilde\theta_j -
z\widetilde\theta_j\widetilde\theta_i\in \mathcal B(RX_n,q_z)$. 
If we specialise $z$ to $1$, $u_{ij}$ becomes 
the commutator of the truncated Dunkl elements 
$\theta_i,\theta_j\in \mathcal B(Y_n)$ of 
Fomin\dash Kirillov. So by a result of \cite{FominKirillov}, 
$u_{ij}|_{z=1}=0$. Now consider $u_{ij}|_{z=-1}$. This is the anticommutator 
of the flat connections $\alpha_i$ and $\alpha_j$ in Majid's algebra 
$\Lambda_n$. So by a result of \cite{MajidQuadratic}, 
$u_{ij}|_{z=-1}=0$. 
Recall the direct sum decomposition 
$\mathcal B(RX_n,q_z) \cong
\frac{1+z}{2}\mathcal B(X_n,q_1) \dirsum
\frac{1-z}{2}\mathcal B(X_n,q_{-1})$ from Theorem~\ref{thm:supernichols}.
That both specialisations of $u_{ij}$ are zero  means that 
the projections of $u_{ij}$ onto both summands of this direct sum 
are zero. But this means $u_{ij}=0$ in $\mathcal B(RX_n,q_z)$, as required.
Alternatively, the Proposition may be deduced from the quadratic
relations in $\mathcal B(RX_n,q_z)$ given at the end of 
Section~\ref{sect:YD}.
\end{proof}

We also denote
$$
\widetilde \theta_j^* = - \sum_{1\le i<j}e^*_{(i\,j)} 
+ \sum_{j<i\le n}e^*_{(j\,i)} \quad\in RX_n,
$$
i.e., $\widetilde \theta_j^*$ is the trivial lift 
of $\theta_j^*$ from $Y_n^*$ to $\widetilde Y_n^*=R\tensor_\CC Y_n^*$. 
Note that the $\widetilde \theta_j^*$ commute in  $\mathcal
B(\widetilde Y_n^*)$. 

Let us now see what are the cross\dash commutator relations between
the $\widetilde \theta_j$, 
the elements of the group $T_n$ and the $\widetilde \theta_i^*$ in the algebra
$(\widetilde{\mathcal H_{Y_n}})_{[1,z]}$. 
Recall from \ref{subsect:cover_symm} that elements of $T_n$ can be written as 
$(1,\sigma)$ of $(z,\sigma)$ with $\sigma\in S_n$.
In particular, $t_i=(1,(i\ i+1))$. The group product 
$\star_{[1,z]}$ on $T_n$ is twisted by the cocycle $[1,z]$. Let
$$
t_{ij} =\begin{cases}
(1,(i\, j)), & \text{if }i<j,\\
(z,(i\, j)), &\text{if }i>j,
\end{cases}
$$   
so that $t_{ij}$ and $t_{ji}$ are the two lifts of the transposition $(i\, j)$
to the Schur covering group $T_n$, and $t_i=t_{i,i+1}$. 
Note that, by Theorem \ref{thm:ext-heisenberg} and the explanation
of the case $R=\CC \Gamma$ in \ref{subsect:cgamma},  
$$
\text{if }i<j,\ (i\, j)\ne (k\, l), \qquad 
e_{(i\,j)}^* e_{(i\, j)} -e_{(i\,j)} e_{(i\, j)}^* = t_{ij};
\qquad
e_{(i\,j)}^* e_{(k\, l)} -e_{(k\,l)} e_{(i\, j)}^* =0
$$
in the algebra $(\widetilde{\mathcal H_{Y_n}})_{[1,z]}$. 
%
%

\begin{lemma}
\label{lem:relations}
The following relations hold in 
the algebra $(\widetilde{\mathcal H_{Y_n}})_{[1,z]}$:
\begin{itemize} 
\item[(a)] $t_{i}\widetilde \theta_j = z \widetilde \theta_j t_{i}$,
\quad $t_{i}\widetilde \theta_j^* =  \widetilde \theta_j^* t_{i}$,
\quad $j\ne i,i+1$;
\item[(b)]
$t_i\widetilde \theta_i = \widetilde \theta_{i+1}t_i$,
\quad
$t_i\widetilde \theta_i^* = \widetilde \theta_{i+1}^*t_i$;
\item[(c)] $\widetilde\theta_i^*\widetilde\theta_j -
\widetilde\theta_j\widetilde\theta_i^* = -t_{ij}$,
\quad 
$j\ne i$;
\item[(d)]
$\widetilde\theta_i^*\widetilde\theta_i -
\widetilde\theta_i\widetilde\theta_i^* = \sum_{k\ne i}t_{ki}$.
\end{itemize}
\end{lemma}
\begin{proof}
If $\sigma\in S_n$ and $\widetilde\sigma$ is any one of the two elements 
of $T_n$ satisfying $\widetilde\sigma|_{z=1}=\sigma$,
then $\widetilde\sigma e_\tau \widetilde\sigma^{-1}=q_z(\sigma,\tau)
e_{\sigma \tau\sigma^{-1}}$ 
in $(\widetilde {\mathcal H_{Y_n}})_{[1,z]}$. The function 
$q_z\in Z^1(S_n,\Fun(X_n,R^\times))$ was defined in \ref{subsect:nichols-z}. 
In particular, if $\sigma=(i\ i+1)$, 
this gives $t_i e_\tau=ze_\tau t_i$ whenever $\tau\ne (i\ i+1)$.
Observe that  $t_i e_\tau^*=e_\tau^* t_i$ whenever $\tau\ne (i\ i+1)$
because the module $\widetilde Y_n^*$ is not twisted. 
So (a) follows by direct verification. 
Also, 
$t_i e_{(i\ i+1)}=-e_{(i\ i+1)} t_i$ and 
$t_i e_{(i\ i+1)}^*=-e_{(i\ i+1)}^* t_i$ and 
which implies (b). 
For (c), use the commutation relations 
between $e_{(i\, j)}$ and $e_{(k\, l)}$ given before the Lemma, 
and observe that the commutator 
$[\widetilde\theta_i^*, \widetilde\theta_j]$ is equal to 
$[e_{(i\, j)}^*, -e_{(i\, j)}]$ if $i<j$ and to 
$[-e_{(i\, j)}^*, ze_{(i\, j)}]$ if $i>j$.
In both cases the result is $-t_{ij}$. 
Finally, observe that 
$[\widetilde\theta_i^*, \widetilde\theta_i]=
\sum_{k<i}[-e^*_{(k\, i)}, -e_{(k\, i)}]+\sum_{k>i}
[e^*_{(i\, k)}, ze_{(i\, k)}]$, so (d) easily follows. 
\end{proof}

We are now ready to construct the covering Cherednik algebra.

\begin{definition}
$\widetilde H_{0,c}$ is the $\CC$\dash algebra with generators
\begin{itemize}
\item[] $z,\quad 
\widetilde x_1,\ldots,\widetilde x_n, 
\quad
t_1,\ldots,t_{n-1},
\quad
\widetilde y_1,\ldots,\widetilde y_n$,
\end{itemize}
and relations
\begin{itemize}
\item[(i)] $z$ is central, $z^2=1$;
\item[(ii)] Schur's $T_n$ relations between $z,t_1,\ldots,t_n$ 
from Theorem~\ref{thm:schur};
\item[(iii)] $\widetilde x_i \widetilde x_j = z\widetilde x_j \widetilde x_i$,   
\quad $i\ne j$;
\item[(iv)] $\widetilde y_i \widetilde y_j = \widetilde y_j \widetilde y_i$,   
\quad $i\ne j$;
\item[(v)]
$t_{i}\widetilde x_j = z \widetilde x_j t_{i}$,
\quad $t_{i}\widetilde y_j =  \widetilde y_j t_{i}$,
\quad $j\ne i,i+1$; 
\item[(vi)]
$t_i \widetilde x_i=\widetilde x_{i+1}t_i$, 
\quad $t_i \widetilde y_i=\widetilde y_{i+1}t_i$;
\item[(vii)] $\widetilde y_i\widetilde x_j -
\widetilde x_j\widetilde y_i = ct_{ij}$,
\quad 
$j\ne i$;
\item[(viii)]
$\widetilde y_i\widetilde x_i -
\widetilde x_i\widetilde y_i = -c\sum_{k\ne i}t_{ki}$.
\end{itemize}
\end{definition}
Denote by $(\CC\Gamma)[\widetilde x_1,\ldots,\widetilde x_n]_z$ 
an algebra over $\CC\Gamma=\CC \cycl_2$ generated by 
$\widetilde x_1,\ldots,\widetilde x_n$ subject to the relations
(iii) above. This is a $z$\dash commutative analogue of the polynomial algebra in the 
category $\GVect$. As $\CC \Gamma$\dash modules (not as algebras), 
$(\CC\Gamma)[\widetilde x_1,\ldots,\widetilde x_n]_z\cong 
\CC \Gamma\tensor_\CC \CC[\widetilde x_1,\ldots,\widetilde x_n]$. 
On the other hand, 
the algebra $(\CC\Gamma)[\widetilde y_1,\ldots,\widetilde y_n]$ (without the $z$ subscript)
is a commutative polynomial algebra.

\begin{theorem}
The algebra $\widetilde H_{0,c}$ is a braided double in the category 
$\GVect$, where $\Gamma=\cycl_2$, with triangular decomposition
$$
\widetilde H_{0,c}\cong 
(\CC\Gamma)[\widetilde x_1,\ldots,\widetilde x_n]_z
\tensor_{\CC \Gamma} \CC T_n 
\tensor_{\CC\Gamma} (\CC\Gamma)[\widetilde y_1,\ldots,\widetilde y_n].
$$
That is, the elements 
$$
\widetilde x_1^{k_1}\ldots\widetilde x_n^{k_n}
(1,\sigma) 
\widetilde y_1^{l_1}\ldots\widetilde y_n^{l_n},
\qquad k_i,l_i\ge 0, \qquad \sigma\in S_n,
$$
are a basis of $\widetilde H_{0,c}$ as a free $\CC \Gamma$\dash module.
There is a morphism 
$$
\widetilde f\colon \widetilde H_{0,c}\to 
(\widetilde{\mathcal H_{Y_n}})_{[1,z]},
 \qquad \widetilde f(\widetilde x_i)=\widetilde \theta_i,
 \qquad \widetilde f(\widetilde y_i)=-c\widetilde \theta_i^*,
\qquad \widetilde f|_{T_n}=\id_{T_n}
$$
of braided doubles in $\GVect$ which covers the morphism 
$f\colon H_{0,c}\to \mathcal H_{Y_n}$, so that 
$\widetilde H_{0,c}$ is an extension  of $H_{0,c}$ 
by the cocycle $[1,z]\in Z^2(S_n,\cycl_2)$ covering $f$. 
\end{theorem}
\begin{proof}
It follows from Proposition~\ref{prop:z-commute} and Lemma~\ref{lem:relations}
that the map $\widetilde f$ is well\dash defined because 
the defining relations of $\widetilde H_{0,c}$ hold in 
$(\widetilde{\mathcal H_{Y_n}})_{[1,z]}$. 
It remains to establish that $\widetilde H_{0,c}$ is a
braided double in $\GVect$. 

We use Proposition~\ref{prop:R-quadratic} where $V^-$ is the free
$\CC\Gamma$\dash module with basis 
$\widetilde x_1,\ldots,\widetilde x_n$  
and $V^+$ is the free
$\CC\Gamma$\dash module with basis $\widetilde y_1,\ldots,\widetilde y_n$.
Denote by $R^-$ the subspace of $V^-\tensor_{\CC\Gamma}V^-$ 
spanned by the $z$\dash commutation relations (iii) and by $R^+$ 
the subspace of $V^+\tensor_{\CC\Gamma}V^+$ spanned by
the commutation relations (iv).

Crucially, because of the form of these quadratic relations 
we already know that the quadratic $\CC\Gamma$\dash algebras
$T_{\CC \Gamma}(V^-)/\lgen R^-\rgen = 
(\CC\Gamma)[\widetilde x_1,\ldots,\widetilde x_n]_z$ and 
$T_{\CC \Gamma}(V^+)/\lgen R^+\rgen = (\CC\Gamma)[\widetilde y_1,\ldots,\widetilde y_n]$ 
are free $\CC\Gamma$\dash modules. 
It remains to check that $R^+$ commutes with $V^-$ in the free 
double $A_\beta(0,0)$, and similarly $V^+$ commutes with $R^-$. We will 
check the latter, the former being analogous. 

Definitely the commutator $[V^+,R^-]$ lies 
in $V^- \tensor_{\CC\Gamma}\CC T_n$. 
Moreover, the map $\widetilde f$ can be viewed as a map of 
$\CC \Gamma$\dash algebras from $A_\beta(0,0)$ to 
$(\widetilde{\mathcal H_{Y_n}})_{[1,z]}$. 
Because $\widetilde f(R^-)=0$ in $(\widetilde{\mathcal
  H_{Y_n}})_{[1,z]}$
--- recall that $\widetilde\theta_1,\ldots,\widetilde\theta_n$ 
$z$\dash commute in $(\widetilde H_{Y_n})_{[1,z]}$ ---  
one trivially has 
$\widetilde f([V^+,R^-])=0$. So $[V^+,R^-]$ is in $\ker \widetilde f$.
But
$$
\ker \widetilde f|_{V^-}=(1+z)(\widetilde x_1+\ldots+\widetilde x_n)
$$
because one can observe that 
$(1+z)(\widetilde \theta_1+\ldots+\widetilde \theta_n)=0$
is the only $\CC\Gamma$\dash dependency between 
$\widetilde \theta_1,\ldots,\widetilde \theta_n$ in $Y_n$.
Furthermore,
$\ker \widetilde f|_{V^-\tensor_{\CC\Gamma}\CC T_n}
= \ker \widetilde f|_{V^-}\tensor_{\CC\Gamma}\CC T_n$.
Therefore, 
$[V^+,R^-]\subset 
 (1+z)(\widetilde x_1+\ldots+\widetilde x_n)
\tensor_{\CC\Gamma}\CC T_n$.
It remains to observe that the substitution
$z=1$ makes $[V^+,R^-]$ zero --- this is because 
the relations (i)--(viii) at $z=1$ become the relations in the 
rational Cherednik algebra $H_{0,c}$, which is a quadratic double in $\Vect$. 
\end{proof}
\begin{remark}
In the proof of the Theorem we observed that the $\widetilde f$\dash image of 
$\widetilde H_{0,c}$ in $(\widetilde{\mathcal H_{Y_n}})_{[1,z]}$, although it
automatically has triangular decomposition over $\CC T_n$, 
is not a braided double in $\GVect$. Indeed, the submodule
of $(\widetilde {Y_n})_{[1,z]}$ generated by  $\widetilde f(\widetilde x_1)=
\widetilde \theta_1,\ldots,\widetilde f(\widetilde x_n)=\widetilde \theta_n$
fails to be free over $\CC\Gamma$: recall 
the relation 
$(1+z)\sum_{i=1}^n \widetilde \theta_i=0$. 
Consequently, the subalgebra 
$$
(P_n)_z=\langle \widetilde \theta_1,\ldots,\widetilde \theta_n\rangle \subset 
\mathcal B(RX_n,q_z)
$$ 
is not a flat $\cycl_2$\dash deformation 
of the coinvariant algebra $P_n$ of $S_n$.
Observe that $(P_n)_z|_{z=1}=P_n$ while $(P_n)_z|_{z=-1}$ is 
Majid's algebra of flat connections, 
generated by $\alpha_1,\ldots,\alpha_n$ in $\mathcal B(X_n,q_{-1})$.
These two graded algebras differ already in degree $1$
of the grading, because 
the Fomin\dash Kirillov elements $\theta_1,\ldots,\theta_n$ are 
linearly dependent (their sum is zero), unlike $\alpha_1,\ldots,\alpha_n$.

It would be interesting to describe 
a set of homogeneous defining relations for the $\CC \cycl_2$\dash algebra 
$(P_n)_z$; they 
would be a $z$\dash version of primitive symmetric functions:
\end{remark}
\begin{question}
Describe the relations between 
the elements $\widetilde \theta_1,\ldots,\widetilde \theta_n$
in the algebra $\mathcal B(RX_n,q_z)$.
\end{question} 

\subsection{The spin Cherednik algebra}

We finish the paper by observing that the covering Cherednik algebra 
$\widetilde H_{0,c}$ admits, as any braided double in $\GVect$, a 
specialisation at $z=-1$. By definition, this algebra will be 
a cocycle twist of $H_{0,c}$. The resulting algebra 
has triangular decomposition 
$$
(\widetilde H_{0,c})|_{z=-1} \cong \CC[x_1,\ldots,x_n]_{(-1)}
\tensor (\CC S_n)_{[1,-1]} \tensor \CC[y_1,\ldots,y_n],
$$
where $\CC[x_1,\ldots,x_n]_{(-1)}$ is the algebra 
generated by the pairwise anticommuting variables $x_i$, and 
$(\CC S_n)_{[1,-1]}$ is the spin symmetric group as in \ref{subsect:spinsymm}. 

The algebra
$\widetilde H_{0,c}|_{z=-1}$ coincides with 
the rational double affine Hecke algebra for the spin symmetric group, 
constructed using a different approach by Wang in \cite{Wang}. Moreover, 
Khongsap and Wang \cite{KhongsapWang2} proposed a covering version of this 
algebra, which is isomorphic to our $\widetilde H_{0,c}$. 

Khongsap and Wang also gave versions of their construction for Weyl groups 
$W$ of types $B_n=C_n$ and $D_n$ in lieu of the symmetric group (the symmetric group 
$S_n$ is viewed as the Weyl group of type $A_{n-1}$). 
In our approach, it is natural to look for extensions 
of the rational Cherednik algebra of $W$ by a $2$\dash cocycle on $W$.
More generally, the technique should work for imprimitive 
complex reflection groups, denoted $W=G(m,p,n)$ in 
Shephard\dash Todd's classification; their 
Schur multipliers were computed by Read in \cite{Read}.
 We leave this generalisation to 
our upcoming paper \cite{BBnew}.

\bibliographystyle{amsplain}

\end{document}